\documentclass[11pt,DIV=15,oneside]{scrbook}
\usepackage[utf8]{inputenc}
\usepackage[english]{babel}
\usepackage[babel]{csquotes}
\usepackage[backend=bibtex, hyperref]{biblatex}
\bibliography{RMT}

\usepackage{amsmath}
\usepackage{amsfonts}
\usepackage{amssymb}
\usepackage{amsthm} 
\usepackage{dsfont} 
\usepackage[mathscr]{euscript} 
\usepackage{eurosym}
\usepackage{bbm} 
\usepackage[a4paper]{geometry}
\usepackage{comment} 
\usepackage{paralist}

\usepackage{graphicx}
\usepackage{xkeyval}
\usepackage{pstricks}
\usepackage{hyperref}
\usepackage{longtable}
\usepackage{tikz}
\usetikzlibrary{matrix}

\newcommand{\lebesgue}{\lambda\mspace{-7mu}\lambda}

\newcommand{\Cat}{\mathscr{C}}
\newcommand{\Prob}{\mathds{P}}

\newcommand{\E}{\mathds{E}}
\newcommand{\V}{\mathds{V}}
\newcommand{\C}{\mathbb{C}}
\newcommand{\R}{\mathbb{R}}

\newcommand{\N}{\mathbb{N}}
\newcommand{\K}{\mathbb{K}}

\newcommand{\mc}[1]{\mathcal{#1}}
\newcommand{\abs}[1]{|{#1}|} 
\newcommand{\bigabs}[1]{\left|{#1}\right|} 
\newcommand{\one}{\mathds{1}}
\newcommand{\Mcal}{\mathcal{M}}

\newcommand{\Acal}{\mathcal{A}}

\newcommand{\Bcal}{\mathcal{B}}
\newcommand{\Ccal}{\mathcal{C}}
\newcommand{\Fcal}{\mathcal{F}}
\newcommand{\Pcal}{\mathcal{P}}
\newcommand{\Gcal}{\mathcal{G}}
\newcommand{\Qcal}{\mathcal{Q}}

\newcommand{\Wcal}{\mathcal{W}}
\newcommand{\Xcal}{\mathcal{X}}
\newcommand{\Tcal}{\mathcal{T}}
\newcommand{\Lcal}{\mathcal{L}}

\newcommand*{\defeq}{\mathrel{\vcenter{\baselineskip0.5ex \lineskiplimit0pt
                     \hbox{\scriptsize.}\hbox{\scriptsize.}}}%
                     =}
										
\newcommand{\integrala}[2]{\left\langle{#1},{#2}\right\rangle}
\newcommand{\integralb}[2]{\int{#2}\,\mathrm{d}{#1}}
\newcommand{\integralc}[3]{\int_{#3}{#2}\,{#1}}

\newcommand{\de}{\text{d}}
\newcommand{\supnorm}[1]{\|#1\|_{\infty}}

\newcommand{\opnorm}[1]{\|#1\|_{\mathrm{op}}}

\newcommand{\norm}[1]{\|#1\|}
\newcommand{\bignorm}[1]{\left\|#1\right\|}

\newcommand{\ubar}[1]{\underline{#1}}
\newcommand{\oneto}[1]{[{#1}]}

\DeclareMathOperator{\Cov}{Cov}
\newcommand{\Mat}[2]{\textrm{Mat}_{#1}(#2)}
\newcommand{\SMat}[2]{\textrm{SMat}_{#1}(#2)}
\renewcommand{\Re}{\operatorname{Re}}
\renewcommand{\Im}{\operatorname{Im}}
\DeclareMathOperator{\supp}{supp}
\DeclareMathOperator{\CovMat}{CovMat}
\DeclareMathOperator{\tr}{tr}
\DeclareMathOperator{\diag}{diag}

\DeclareMathOperator{\dbl}{d_{BL}}

\newcommand{\PePe}[1]{\mathcal{PP}(#1)}

\theoremstyle{plain}
\newtheorem{lemma}{Lemma}[chapter]
\newtheorem{theorem}[lemma]{Theorem}

\newtheorem{corollary}[lemma]{Corollary}
\theoremstyle{definition}
\newtheorem{definition}[lemma]{Definition}

\newtheorem{example}[lemma]{Example}
\newtheorem{remark}[lemma]{Remark}
\theoremstyle{remark}

\newcommand{\map}[5]{
		\begin{align}
		{#1}: {#2}\ &\longrightarrow \ {#3}\notag\\
		      {#4}\ &\longmapsto \ \notag {#5}
		\end{align}
		}

\begin{document}

\begin{titlepage}

\vspace{6cm}

\begin{center}
{\LARGE
\textsf{\textbf{Proof Methods in Random Matrix Theory}}
}

\vspace{0.8cm}

\vspace{0.8cm}
{\Large Michael Fleermann and Werner Kirsch
\vspace{2mm}

FernUniversität in Hagen\\
Fakultät für Mathematik und Informatik\\
Universitätsstraße 1\\
58097 Hagen, Germany\\
\texttt{michael.fleermann@fernuni-hagen.de}\\
\texttt{werner.kirsch@fernuni-hagen.de}}
\end{center}

\vspace{3cm}
\textbf{Abstract.} In this survey article, we give an introduction to two methods of proof in random matrix theory: The method of moments and the Stieltjes transform method. We thoroughly develop these methods and apply them to show both the semicircle law and the Marchenko-Pastur law for random matrices with independent entries. The material is presented in a pedagogical manner and is suitable for anyone who has followed a course in measure-theoretic probability theory.

\vspace{3mm}

\textbf{2020 Mathematics Subject Classification.} 60B20.

\vspace{2mm}

\textbf{Key words and phrases.} Random matrix theory, method of moments, Stieltjes transform method, semicircle law, Marchenko-Pastur law.

\clearpage
\thispagestyle{empty}

\end{titlepage}

\tableofcontents

\chapter{Introduction}

The goal of this article is to give a digestible yet concise introduction to random matrix theory. We focus on the tools and concepts that allow us to comprehend the results which marked the very beginnings of this theory: The semicircle law discovered in \cite{Wigner1955, Wigner1958} and the Marchenko-Pastur law established in \cite{MP1967}. These are statements pertaining to probabilistic weak convergence -- namely weak convergence in expectation resp.\ in probability resp.\ almost surely -- which is a framework also encountered in probability theory when studying the Glivenko-Cantelli theorem, for example. We thoroughly investigate the subtleties of probabilistic weak convergence in Chapter~\ref{chp:weakconvergence} of this text.

Statements about weak convergence -- such as the central limit theorem -- may be proved in numerous ways, two of them being the analysis of the moments of the distributions or the analysis of certain transforms of the distributions involved. Concerning the proof of the central limit theorem, see Chapter 30 in \cite{Billingsley:1995} for the use of moments, and Chapter 27 in \cite{Billingsley:1995} for the use of transforms. When studying statements of probabilistic weak convergence in random matrix theory, it turns out that again, moments and transforms can be employed with great success and in numerous settings. Therefore, we carefully develop the method of moments in Chapter~\ref{chp:methodofmoments} and the Stieltjes transform method in Chapter~\ref{chp:stieltjesprops}. We employ these methods to show both the semicircle law and the Marchenko-Pastur law in Chapters~\ref{chp:MomentsSCLMPL} and~\ref{chp:StieltjesSCLMPL}.

During the past decades, random matrix theory has evolved into a huge field of study. Both the results and the techniques to derive them have become rather sophisticated, making an entry into this field cumbersome. This text aims to alleviate this barrier of entry and can be followed after completing a basic course of measure-theoretic probability theory. It is based on the works \cite{FleermannMT, FleermannDiss} of the first author, but has also benefitted greatly from the research endeavors of both authors. Further, the techniques presented are employed in many contemporary research articles and are thus highly relevant for researchers aiming to contribute to random matrix theory.

\chapter{Weak Convergence}
\label{chp:weakconvergence}
\section{Spaces of Continuous Functions}

On the set $\R$\label{sym:realnumbers} of real numbers we will always consider the standard topology and the associated Borel $\sigma$-algebra $\Bcal$\label{sym:realborel}. To study convergence of probability measures on $(\R,\Bcal)$, it is useful to get acquainted with certain spaces of functions $\R\to\R$ first. If $f:\R\to\R$ is a function, we define the \emph{support} of $f$ as 
\[
\supp(f)\defeq \overline{\{ x\in\R: f(x)\neq 0\}}.\label{sym:support}
\]
Note that by definition, the support of $f$ is always a closed subset of $\R$, and it is immediate that a point $x\in\R $ lies in the support of $f$ if and only if for any $\varepsilon>0$ there is a $y\in B_{\varepsilon}(x)$, such that $f(y)\neq 0$. Here and later, $B_{\delta}(z)$\label{sym:openball} denotes the open $\delta$-ball around the element $z$ in a metric space which is clear from the context.

We say that a function $f:\R\to\R$ \emph{vanishes at infinity}, if
\[
\lim_{x\to\pm\infty} f(x) =0.
\]

Denote by $\Ccal(\R)$ \label{sym:continous} the vector space of continuous functions $\R\to\R$. We define the three subspaces
\begin{enumerate}
\item $\Ccal_b(\R)\defeq \{f:\R\to\R \,|\, f \text{ is continuous and bounded}\}$, \label{sym:contbounded}
\item $\Ccal_0(\R)\defeq \{f:\R\to\R \,|\, f \text{ is continuous and vanishes at infinity}\}$ and\label{sym:contvanish}
\item $\Ccal_c(\R)\defeq \{f:\R\to\R \,|\, f \text{ is continuous with compact support}\}$.\label{sym:contcompact}
\end{enumerate}
It is clear that
\[
\Ccal_c(\R)\subsetneq \Ccal_0(\R) \subsetneq \Ccal_b(\R) \subsetneq \Ccal(\R),
\]
since the function $x\mapsto \min(1,1/\abs{x})$ lies in $\Ccal_0(\R)\backslash \Ccal_c(\R)$, the function $x\mapsto \one_{\R}(x)$ \label{sym:indicator} lies in $\Ccal_b(\R)\backslash \Ccal_0(\R)$ and the function $x\mapsto x$ lies in $\Ccal(\R)\backslash \Ccal_b(\R)$. Since all functions in $\Ccal_c(\R)$, $\Ccal_0(\R)$ and $\Ccal_b(\R)$ are bounded, we can equip these spaces with the supremum norm $\supnorm{\cdot}$ defined by\label{sym:supnorm}
\[
\supnorm{f}\defeq \sup_{x\in\R}\abs{f(x)}.
\]
From now on, we will always consider the spaces $\Ccal_b(\R)$, $\Ccal_0(\R)$ and $\Ccal_c(\R)$ as vector spaces normed by the supremum norm. Convergence with respect to this norm is also called \emph{uniform convergence}. To analyze properties of these normed spaces, we introduce continuous cutoff-functions as in \autocite[8]{KirschMomentSurvey}:
\begin{definition}
\label{def:contcutoff}
For any real numbers $u>\ell\geq 0$ we define the function $\phi_{\ell}^u:\R\to[0,1]$ by\label{sym:contcutoff}
\[
\phi_{\ell}^u(x)\defeq
\begin{cases}
1 &	\text{ if } \abs{x}\leq \ell,\\
\frac{u-\abs{x}}{u-\ell} & \text{ if } \ell<\abs{x}<u,\\
0 & \text{ if } \abs{x} \geq u.
\end{cases}
\]
\end{definition}
Note that for any $u>\ell\geq 0$, $\phi_{\ell}^u$ is continuous with compact support $[-u,u]$.
The following theorem will summarize important properties of $\Ccal_b(\R)$, $\Ccal_0(\R)$ and $\Ccal_c(\R)$.
\begin{theorem}
\label{thm:functionspaces}
The following statements hold:
\begin{enumerate}[i)]
\item $\Ccal_b(\R)$ is complete, but not separable.
\item $\Ccal_0(\R)$ is complete and separable.
\item $\Ccal_c(\R)$ is not complete, but separable.
\item $\Ccal_c(\R)$ is dense in $\Ccal_0(\R)$.
\end{enumerate}
\end{theorem}
\begin{proof}
\underline{i)} 
If $(f_n)_n$ is Cauchy in $\Ccal_b(\R)$ and $x\in\R$, then $f_n(x)$ is Cauchy in $\R$, thus converges to a limit $f(x)\in \R$. Further, we can pick an $m\in\N$ such that $f_m$ is uniformly $\varepsilon$-close to all $f_n$ for $n$ large enough, from which it follows that $f_n\to f$ uniformly. From this, it easily follows that $f$ is bounded. It remains to show that $f$ is continuous for which we again choose an $f_m$ as above and utilize a standard $3\varepsilon$-argument.
To see that $\Ccal_b(\R)$ is not separable, we construct an uncountable subset $\Fcal\subseteq\Ccal_b$, such that for all $f,g\in\Fcal$ with $f\neq g$ we have $\supnorm{f-g}=1$. To this end, denote by $Z$ the set of $0$-$1$-sequences, so $Z=\{0,1\}^{\N}$. Note that $Z$ is uncountable. For any sequence $z\in Z$ we define
\[
\forall\, x\in\R: F_z(x) \defeq \sum_{i\in\N} z_i\cdot\phi_{0.1}^{0.2}(x-i)
\] 
and $\Fcal\defeq\{F_z\,|\,z\in Z\}$. Now $\Fcal$ is as desired.

\noindent
\underline{iii)/iv)} To show that $\Ccal_c(\R)$ is not complete, we show that it is not closed in the strict superset $\Ccal_0(\R)$. In fact, we show even more, that is, that $\Ccal_c(\R)$ is dense in $\Ccal_0(\R)$ (then since $\Ccal_c(\R)\subsetneq \Ccal_0(\R)$, $\Ccal_c(\R)$ cannot be closed). This fact is also needed for statements ii) and iv). So let $f\in\Ccal_0(\R)$ be arbitrary. Now consider the sequence of functions $(f_n)_n$, where 
\[
\forall\, n\in\N: \,\forall\, x\in\R:\, f_n(x)\defeq \phi_{n,n+1}(x)f(x).
\] 
Then $(f_n)_n$ is a sequence in $\Ccal_c(\R)$ which converges uniformly to $f$. Hence, $\Ccal_c(\R)$ is dense in $\Ccal_0(\R)$. Next, we will show that $\Ccal_c(\R)$ is separable. To this end, denote by $\Pcal$ the countable set of all polynomials with rational coefficients and set
\[
\Qcal \defeq \{p\cdot \phi_n^{n+1} \, |\, p\in\Pcal, n\in\label{sym:naturalnumbers}\N\}.
\]
Then $\Qcal$ is easily identified as a dense countable subset of $\Ccal_c(\R)$. 

\noindent
\underline{ii)} To show that $\Ccal_0(\R)$ is complete, let $(f_n)_n$ be an arbitrary Cauchy sequence in $\Ccal_0(\R)$. This is also a Cauchy sequence in $\Ccal_b(\R)$, so with i) we know that there is an $f\in\Ccal_b(\R)$ such that $f_n\to f$ uniformly. It is easily seen that $f$ vanishes at infinity, so that $f\in\Ccal_0(\R)$. To see that $\Ccal_0(\R)$ is separable, note that we have already seen that $\Ccal_c(\R)$ is separable and dense in $\Ccal_0(\R)$.  

\end{proof}

\section{Convergence of Probability Measures}

 We will denote the set of measures on $(\R,\Bcal)$ by $\Mcal(\R)$\label{sym:measures}, the set of finite measures by $\Mcal_f(\R)$\label{sym:finitemeasures}, the set of probability measures by $\Mcal_1(\R)$\label{sym:probmeasures}, and the set of sub-probability measures by $\Mcal_{\leq 1}(\R)$\label{sym:subprobmeasures}. Here, a measure $\mu$ on $(\R,\Bcal)$ is called \emph{sub-probability measure}, if $\mu(\R)\in[0,1]$. Note that
\[
\Mcal_{1}(\R) \subsetneq \Mcal_{\leq 1}(\R)\subsetneq \Mcal_{f}(\R)\subsetneq \Mcal(\R).
\]
As a shorthand notation, if $\mu\in\Mcal(\R)$ and $f:\R\to\R$ is measurable, we write
\[
\integrala{\mu}{f}\defeq \integralb{\mu}{f}\label{sym:muintegral}
\] 
with the convention that when in doubt, $x$ is the variable of integration:
\[
\integrala{\mu}{x^k} = \int{x^k \mu(\mathrm{d}x)}.\label{sym:muxintegral}
\]

\begin{definition}
Let $\Fcal\subseteq \Ccal_b(\R)$ be a linear subspace, then a \emph{positive linear bounded functional} $I$ on $\Fcal$ is a bounded $\R$-linear map $\Fcal\to\R$ with $I(f)\geq 0$ for all $f\in\Fcal$ with $f\geq 0$.
\end{definition}

\begin{lemma}
\label{lem:Imu}
Let $\Fcal\subseteq \Ccal_b(\R)$ be a linear subspace with $\Ccal_c(\R)\subseteq \Fcal$. Then for any $\mu\in\Mcal_f(\R)$, the map
\map{I_{\mu}}{\Fcal}{\R}{f}{I_{\mu}(f)\defeq\integrala{\mu}{f}}
defines a positive linear bounded functional on $\Fcal$ with operator norm $\mu(\R)$.
\end{lemma}
\begin{proof}
We only need to show that the operator norm is indeed $\mu(\R)$. To see this, note that for any $k>0$, we have $\phi_{k}^{k+1} \in \Fcal$, $\phi_{k}^{k+1}\geq 0$ and $\supnorm{\phi_{k}^{k+1}}=1$. Further,
\[
I_{\mu}(\phi_{k}^{k+1})=\integrala{\mu}{\phi_{k}^{k+1}}\geq \mu([-k,k]).
\]
Thus, the operator norm of $I_{\mu}$ is at least $\mu([-k,k])$ for all $k>0$, hence at least $\mu(\R)$. On the other hand, the operator norm is at most $\mu(\R)$, since for any $f\in\Fcal$ we find $\abs{\integrala{\mu}{f}}\leq \integrala{\mu}{\abs{f}} \leq \mu(\R)\cdot \supnorm{f}$.
\end{proof}

The representation theorem of Riesz now states that \emph{any} positive linear bounded functional $I$ on a linear space $\Fcal$ with $\Ccal_c(\R)\subseteq \Fcal \subseteq \Ccal_0(\R)\}$ has the form $I=I_{\mu}$ as in Lemma~\ref{lem:Imu}.

\begin{theorem}
\label{thm:Rieszrep}
Let $\Fcal$ be a linear space with $\Ccal_c(\R)\subseteq \Fcal \subseteq \Ccal_0(\R)$ and equipped with the supremum norm. Then for any positive linear bounded functional $I$ on $\Fcal$, there exists exactly one $\mu\in\Mcal_f(\R)$ with $I=I_{\mu}$. It then holds $\opnorm{I}=\mu(\R)$.\end{theorem}
\begin{proof}
The statement is well-known, see e.g.\ \cite{Elstrodt}.
\end{proof}

The next lemma will help us infer equality of two finite measures. Notationally, if $A$ is a subset of a topological space, we denote its boundary by $\partial A$.

\begin{lemma}
\label{lem:equalityoffinitemeasures}
Let $\mu$ and $\nu$ be two finite measures on $(\R,\Bcal)$ and let $\Fcal\subseteq\Ccal_c(\R)$ be a dense subset. Then
\begin{enumerate}[i)]
\item $\mu=\nu \quad\Leftrightarrow\quad \mu(I)=\nu(I)$ for all bounded intervals $I$ with $\mu(\partial I)= \nu(\partial I)= 0$,
\item $\mu=\nu \quad\Leftrightarrow \quad\forall\, f\in\Ccal_c(\R): \integrala{\mu}{f}=\integrala{\nu}{f}\quad\Leftrightarrow \quad \forall\, f\in\Fcal: \integrala{\mu}{f}=\integrala{\nu}{f}$.
\end{enumerate}
\end{lemma}
\begin{proof}
\underline{i)} "$\Rightarrow$" is clear, and for "$\Leftarrow$" we show that $\mu$ and $\nu$ agree on all finite open intervals. To this end, note that for any finite measure $\rho\in\Mcal_f(\R)$, the set of atoms $A_{\rho}\defeq \{x\in\R\,|\,\rho(x)>0\}$ is at most countable. As a result $\R\backslash (A_{\mu}\cup A_{\nu})$ is dense in $\R$. For arbitrary $a<b$ in $\R$, we find sequences $(a_n)_n$ and $(b_n)_n$ in $\R\backslash (A_{\mu}\cup A_{\nu})$ with $a_n \searrow a$ and $b_n \nearrow b$ as $n\to\infty$ and $a_n<b_n$ for all $n\in\N$. Then we obtain with continuity of measures from below (note that $\mu$ and $\nu$ agree on all intervals $(a_n,b_n)$):
\[
\mu((a,b)) = \lim_{n\to\infty}\mu((a_n,b_n)) = \lim_{n\to\infty}\nu((a_n,b_n)) = \nu((a,b)).
\]
\underline{ii)} This follows immediately with Theorem~\ref{thm:Rieszrep}.
\end{proof}

We are especially interested in convergence behavior of sequences in $\Mcal_1(\R)$, where the limit may lie in $\Mcal_{\leq 1}(\R)$.
\begin{definition}
Let $(\mu_n)_{n\in\N}$  be a sequence in in $\Mcal_1(\R)$.
\label{def:weaknvague}
\begin{enumerate}[i)]
\item
The sequence $(\mu_n)_{n\in\N}$ is said to converge \emph{weakly} to an element $\mu\in\Mcal_1(\R)$, if
\begin{equation}
\label{eq:weaknvague1}
\forall f\in \Ccal_b(\R): \lim\limits_{n \rightarrow \infty}\integrala{\mu_n}{f} = \integrala{\mu}{f}.
\end{equation}
\item
 The sequence $(\mu_n)_{n\in\N}$ is said to converge \emph{vaguely} to an element $\mu\in\Mcal_{\leq 1}(\R)$, if
\begin{equation}
\label{eq:weaknvague2}
\forall f\in \Ccal_c(\R): \lim\limits_{n \rightarrow \infty}\integrala{\mu_n}{f} = \integrala{\mu}{f}.
\end{equation}
\end{enumerate}
\end{definition}

\begin{remark}
\label{rem:weaknvague}
We would like to shed light on the seemingly innocent Definition~\ref{def:weaknvague}:

\begin{enumerate}
\item Weak convergence clearly implies vague convergence. Further, due to Lemma~\ref{lem:equalityoffinitemeasures}, weak and vague limits are unique.
\item In light of Theorem~\ref{thm:functionspaces}, it is appropriate to say that the set of test functions for weak convergence is considerably larger than the set of test functions for vague convergence. As a result, weak limits are much more restrictive than vague limits, as clarified by the next two points.
\item The target measures $\mu\in\Mcal(\R)$, for which \eqref{eq:weaknvague1} can be satisfied for some sequence $(\mu_n)_n$ of probability measures are exactly all $\mu\in\Mcal_1(\R)$. To see this, if \eqref{eq:weaknvague1} holds for some $\mu\in\Mcal(\R)$ and a sequence  $(\mu_n)_n$ in $\Mcal_1(\R)$, then we must have $\mu(\R)=1$, since $\one_{\R}\in\Ccal_b(\R)$. On the other hand, if $\mu\in\Mcal_1(\R)$ is arbitrary, then \eqref{eq:weaknvague1} is satisfied for the sequence $(\mu_n)_n$, where $\mu_n=\mu$ for all $n\in\N$.
\item The measures $\mu\in\Mcal(\R)$, for which \eqref{eq:weaknvague2} can be satisfied for some sequence $(\mu_n)_n$ of probability measures are (somewhat surprisingly) exactly all $\mu\in\Mcal_{\leq 1}(\R)$. To see this, if \eqref{eq:weaknvague2} holds for some $\mu\in\Mcal(\R)$ and a sequence  $(\mu_n)_n$ in $\Mcal_1(\R)$, then we have for any $m\in\N$ that $\integrala{\mu_n}{\phi_{m,m+1}} \to_n \integrala{\mu}{\phi_{m,m+1}}$, so $\integrala{\mu}{\phi_{m,m+1}}\leq 1$, which entails $\mu([-m,m])\leq 1$ for all $m\in\N$. Since measures are continous from below, we conclude that also $\mu(\R)\leq 1$, so $\mu$ is a sub-probability measure. On the other hand, if $\mu\in\Mcal_{\leq 1}(\R)$ is arbitrary, then define $\alpha\defeq 1 - \mu(\R)\in[0,1]$ and for all $n\in\N: \mu_n\defeq \mu + \alpha\delta_n$\label{sym:Diracmeasure}. Then $(\mu_n)_n$ is a sequence of probability measures and
 \eqref{eq:weaknvague2} is satisfied for the sequence $(\mu_n)_n$. To see this, let $f\in \Ccal_c(\R)$ be arbitrary and $N\in\N$ be so large that $\supp(f)\subseteq[-N,N]$. Then it holds for all $n\geq N$ that $\integrala{\mu_n}{f}= \integrala{\mu}{f} + \alpha f(n) = \integrala{\mu}{f}$.
\item As a result of points 3.\ and 4., the limit domains for weak and vague convergence in Definition~\ref{def:weaknvague} are  exact. The probability measures lie \emph{vaguely dense} in the sub-probability measures.
\end{enumerate}	
\end{remark}

\begin{lemma}
\label{lem:subsequential}
Let $(\mu_n)_{n\in\N}$ be a sequence of probability measures and  $\mu$ a sub-probability measure on $(\R,\Bcal)$. Then $(\mu_n)_{n\in\N}$ converges vaguely (resp.\ weakly) to $\mu$ if and only if every subsequence $(\mu_n)_{n\in J}$, $J\subseteq \N$, has a subsequence $(\mu_n)_{n\in I}$, $I\subseteq J$, that converges vaguely (resp.\ weakly) to $\mu$.
\end{lemma}
\begin{proof}
Of course, we only need to show "$\Leftarrow$". We assume the statement to be false, that is, that it is not true that $( \mu_n)_{n\in\N}$ converges vaguely (resp.\ weakly) to $\mu$. Then we find a continuous function $f:\R\to\R$ which has compact support (resp.\ which is bounded) and an $\varepsilon>0$ such that  $\abs{\integrala{\mu_n}{f}-\integrala{\mu}{f}} \geq \varepsilon$ for all $n\in J$, where $J\subseteq \N$ is an infinite subset. But now we find a subsequence $(\mu_n)_{n\in I}$, $I\subseteq J$ that converges vaguely (resp.\ weakly) to $\mu$. In particular, we find an $n\in I\subseteq J$ such that $\abs{\integrala{\mu_n}{f}-\integrala{\mu}{f}} < \varepsilon$, which leads to a contradiction to our assumption that the statement is false.
\end{proof}

Vague convergence of probability measures can also be characterized by convergence of the integrals $\integrala{\mu_n}{f}$ for all $f\in\Ccal_0(\R)$.

\begin{lemma}
\label{lem:vagueczero}
A sequence $(\mu_n)_n$ in $\Mcal_1(\R)$ converges vaguely to an element $\mu\in\Mcal_{\leq 1}(\R)$, if and only if
\[
\forall f\in \Ccal_0(\R): \lim\limits_{n \rightarrow \infty}\integrala{\mu_n}{f} = \integrala{\mu}{f}.
\]
	
\end{lemma}

\begin{proof}
This follows easily with the fact that $\Ccal_c(\R)\subseteq\Ccal_0(\R)$ is dense.
\end{proof}

If $\mu_n\to\mu$ weakly, we know that $\integrala{\mu_n}{f}\to\integrala{\mu}{f} $ for all $f\in\Ccal_b(\R)$. Often, we would like to be able to conclude $\integrala{\mu_n}{f}\to\integrala{\mu}{f}$ for more general functions $f$. The next lemma will be of great use in this respect, see also \autocite[107]{Durrett}.
\begin{lemma}
\label{lem:integrationtolimit}
Let $(\mu_n)_n$ and $\mu$ be probability measures such that $\mu_n\to\mu$ weakly as $n\to\infty$. Let $h:\R\to\R$ be continuous. Then to show
\[
\integrala{\mu_n}{h} \xrightarrow[n\to\infty]{} \integrala{\mu}{h},
\]
it is sufficient to show that there is a strictly positive continuous function $g:\R\to (0,\infty)$ such that $\sup_{n\in\N}\integrala{\mu_n}{g}<\infty$ and $h/g$ vanishes at infinity.
 \end{lemma}
\begin{proof}
Let $C\defeq \sup_{n\in\N}\integrala{\mu_n}{g}\in[0,\infty)$. Then also $\integrala{\mu}{g}\leq C$, since $g\phi_{k}^{k+1}\nearrow g$ pointwise as $k\to\infty$, so by monotone convergence $\integrala{\mu}{g\phi_{k}^{k+1}}\nearrow \integrala{\mu}{g}$ as $k\to\infty$. But for any fixed $k$, $\integrala{\mu}{g\phi_{k}^{k+1}}= \lim_n\integrala{\mu_n}{g\phi_{k}^{k+1}}\leq C$.
Now let $\varepsilon>0$ be arbitrary, then $k>0$ so large that $\abs{h}/g \leq \varepsilon/C$ on $[-k,k]^c$ (where if $A$ is a set, we denote its complement by $A^c$\label{sym:setcomplement}, where we assume that the superset of $A$ is clear from the context. For example, $[-k,k]^c=\R\backslash [-k,k]$). We conclude that for all $\nu\in\{\mu,(\mu_n)_n\}$,
\[
\abs{\integrala{\nu}{h(1-\phi_{k}^{k+1})}} \leq  \integrala{\nu}{\frac{\abs{h}}{g}\cdot g(1-\phi_{k}^{k+1})} \leq \frac{\varepsilon}{C}\cdot C = \varepsilon.
\]

In particular, these integrals are well-defined. Since also for any $\nu\in\{\mu,(\mu_n)_n\}$, $\integrala{\nu}{h\phi_k^{k+1}}$ is well-defined, $h$ is $\nu$-integrable. We find for $\varepsilon>0$ and $k>0$ as picked above, that for all $n\in\N$:
\begin{align*}
&\abs{\integrala{\mu_n}{h}-\integrala{\mu}{h}}\\
&\leq  \abs{\integrala{\mu_n}{h(1-\phi_k^{k+1})}-\integrala{\mu}{h(1-\phi_k^{k+1})}} + \abs{\integrala{\mu_n}{h\phi_k^{k+1}}-\integrala{\mu}{h\phi_k^{k+1}}} \\
&\leq \varepsilon + \abs{\integrala{\mu_n}{h\phi_k^{k+1}}-\integrala{\mu}{h\phi_k^{k+1}}},
\end{align*}
where the last summand converges to $0$ as $n\to\infty$, such that 
\[
\limsup_{n\to\infty}\abs{\integrala{\mu_n}{h}-\integrala{\mu}{h}}\leq \varepsilon.
\]
Since $\varepsilon>0$ was arbitrary, we find $\integrala{\mu_n}{h}\to\integrala{\mu}{h}$ as $n\to\infty$.
\end{proof}

As we just saw in Remark~\ref{rem:weaknvague}, vague convergence allows the escape of probability mass. The concept of tightness prevents this from happening:

\begin{definition}
\label{def:tight}
A sequence of probability measures $(\mu_n)_n$ on $(\R,\Bcal)$ is called \emph{tight}, if for all $\varepsilon>0$ there exists a compact subset $K\subseteq \R$ such that
\[
\forall\, n\in\N:\,\mu_n(K^c) \leq \varepsilon.
\]
\end{definition}

A sufficient condition for tightness is given in the next Lemma, which we adopted from \autocite[106]{Durrett}:
\begin{lemma}
\label{lem:tight}
Let $(\mu_n)_n$ be a sequence of probability measures on $(\R,\Bcal)$. If there exists a measurable non-negative function $\phi:\R\to\R$ with $\phi(x)\to\infty$ for $x\to\pm\infty$ and
\[
\sup_n\integrala{\mu_n}{\phi} <\infty,
\]
then $(\mu_n)_n$ is tight. In particular, this holds true if
\[
\sup_n\integrala{\mu_n}{x^2} <\infty.
\]
\end{lemma}
\begin{proof}
Let $C\defeq \sup_n\integrala{\mu_n}{\phi}<\infty$. Then it holds for any $n\in\N$ and $k>0$ that
\[
C \geq \integrala{\mu_n}{\phi}\geq \integrala{\mu_n}{\one_{[-k,k]^c}\cdot \inf_{\abs{x}> k} \phi(x)} =  \integrala{\mu_n}{\one_{[-k,k]^c}}\cdot \inf_{\abs{x}> k} \phi(x).
\]
Since $\inf_{\abs{x}> k} \phi(x)\to\infty$ as $k\to\infty$, the statement follows.
\end{proof}

\begin{lemma}
\label{lem:vaguetoweak}
Let $(\mu_n)_n$ be a sequence in $\Mcal_1(\R)$ and $\mu\in \Mcal_{\leq 1}(\R)$ such that $\mu_n \to\mu$ vaguely as $n\to\infty$,
then the following statements are equivalent:
\begin{enumerate}[i)]
\item $(\mu_n)_n$ is tight.
\item $\mu$ is a probability measure.
\item $\mu_n$ converges weakly to $\mu$.
\end{enumerate}	
\end{lemma}

\begin{proof}

\underline{$i)\Rightarrow iii)$} Let $f\in\Ccal_b(\R)$ be arbitrary and set $s\defeq \max(\supnorm{f},1)$.
Let $\varepsilon>0$ be arbitrary, then due to tightness of $(\mu_n)_n$ and continuity from below of $\mu$, we find a $k>0$ such that $\mu_n([-k,k]^c) \leq \frac{\varepsilon}{2s}$ and $\mu([-k,k]^c)\leq \frac{\varepsilon}{2s}$. Now for $n\in\N$ arbitrary we find
\begin{align*}
&\abs{\integrala{\mu_n}{f} - \integrala{\mu}{f}}\\
&\leq\abs{\integrala{\mu_n}{f} - \integrala{\mu_n}{f\phi_k^{k+1}}} + \abs{\integrala{\mu_n}{f\phi_k^{k+1}} - \integrala{\mu}{f\phi_k^{k+1}}} + \abs{\integrala{\mu}{f\phi_k^{k+1}} - \integrala{\mu}{f}}\\
&\leq \integrala{\mu_n}{\abs{f}\cdot\abs{1-\phi_k^{k+1}}} + \abs{\integrala{\mu_n}{f\phi_k^{k+1}} - \integrala{\mu}{f\phi_k^{k+1}}} + \integrala{\mu}{\abs{f}\cdot\abs{\phi_k^{k+1}-1}}\\
& \leq s\cdot\frac{\varepsilon}{2s} + \abs{\integrala{\mu_n}{f\phi_k^{k+1}} - \integrala{\mu}{f\phi_k^{k+1}}} + s\cdot\frac{\varepsilon}{2s}
\end{align*}
It follows that $\limsup_n\abs{\integrala{\mu_n}{f} - \integrala{\mu}{f}}\leq \varepsilon$.

\underline{$iii)\Rightarrow ii)$} This statement is obvious. Consider $\one_{\R}\in\Ccal_b(\R)$.

\underline{$ii)\Rightarrow i)$}. Let $\varepsilon>0$ be arbitrary. Then for $k>0$ we find
\[
\mu_n([-(k+1),k+1])\geq \integrala{\mu_n}{\phi_k^{k+1}}\geq \integrala{\mu}{\phi_k^{k+1}} - \abs{\integrala{\mu}{\phi_k^{k+1}}-\integrala{\mu_n}{\phi_k^{k+1}}}
\]
Now first choose $k$ large enough such that the first summand on the r.h.s.\ is larger than $1-\varepsilon/2$, then choose $N\in\N$ large enough such that for all $n> N$ the absolute value on the r.h.s.\ is at most $\varepsilon/2$. Then we obtain for all $n > N$ that $\mu_n([-(k+1),k+1])\geq 1 - \varepsilon$. On the other hand, we find $k_1,\ldots,k_{N}>0$ such that
\[
\forall\, i\in\{1,\ldots,N\}:~\mu_i([-k_i,k_i]) \geq 1 - \varepsilon.
\]
Let $k^*\defeq \max\{k+1, k_1,\ldots,k_N\}$, then we obtain for all $n\in\N$ that $\mu_n([-k^*,k^*])\geq 1-\varepsilon$. Therefore, $(\mu_n)_n$ is tight.

\end{proof}

\begin{lemma}
\label{lem:convergentsubsequence}
Let $(\mu_n)_n$ be a sequence in $\Mcal_1(\R)$, then the following statements hold:
\begin{enumerate}[i)]
\item $(\mu_n)_n$ has a subsequence converging vaguely to some $\mu\in\Mcal_{\leq 1}(\R)$.
\item If $(\mu_n)_n$ is tight, it has a subsequence converging weakly to some $\mu\in\Mcal_{1}(\R)$.
\end{enumerate}
\end{lemma}
\begin{proof}
\underline{i)} Let $(g_m)_m$ be a dense sequence in $\Ccal_c(\R)$, then for all $m\in\N$, $(\integrala{\mu_n}{g_m})_n$ is a sequence in $\R$ whose absolute value is bounded by $\supnorm{g_m}<\infty$, thus has a convergent subsequence by Bolzano-Weierstrass. By a diagonal argument, we can find a subsequence $J\subseteq \N$, such that for all $m\in\N$, $(\integrala{\mu_n}{g_m})_{n\in J}$ converges. But since $(g_m)_m$ is dense in $\Ccal_c(\R), \lim_{n\in J}\integrala{\mu_n}{f}$ exists for all $f\in\Ccal_c(\R)$ (it can be shown that $(\integrala{\mu_n}{f})_n$ is Cauchy). The function 
\map{I}{\Ccal_c(\R)}{\R}{f}{I(f)\defeq\lim_{n\in J}\integrala{\mu_n}{f}}
is a linear bounded positive functional on $\Ccal_c(\R)$ with operator norm at most $1$, since $\abs{\integrala{\mu_n}{f}}\leq \supnorm{f}$ for all $n\in\N$ and $f\in\Ccal_c(\R)$. With Theorem~\ref{thm:Rieszrep}, we find an element $\mu\in\Mcal_{\leq 1}(\R)$ such that $I=I_{\mu}$, which entails $\mu_n \to \mu$ vaguely for $n\in J$.\newline
\underline{ii)} With $i)$ we find a subsequence $J\subseteq\N$ and a $\mu\in\Mcal_{\leq 1}(\R)$ such that $(\mu_n)_{n\in J}$ converges to $\mu$ vaguely. But Lemma~\ref{lem:vaguetoweak} yields that $\mu\in\Mcal_{1}(\R)$ and $\mu_n\to\mu$ weakly for $n\in J$.
\end{proof}

Note that statement $i)$ of Lemma~\ref{lem:convergentsubsequence} is the well-known Helly's selection theorem contained in most standard books on probability theory, see \autocite{Durrett} or \autocite{Klenke}, for example. However, we give a new proof here that differs completely from the standard proofs which utilize distribution functions.

So far we have discussed the intricacies of weak and vague convergence of probability measures. Our next goal is to better understand the topology of weak convergence on $\Mcal_1(\R)$, which will deepen our understanding of stochastic weak convergence to be discussed in the next section.
Our first goal will be to reduce the number of test functions for weak convergence to a countable subset of $\Ccal_b(\R)$. However, $(\Ccal_b(\R),\supnorm{\cdot})$ is large; it is not even separable. But there is no reason for despair, since the following theorem holds, which we adopted from our previous work \autocite{FleermannMT}.

\begin{theorem}
\label{thm:weakconvergenceonR}
Fix a sequence $(g_k)_{k\in\N}$ in $\Ccal_c(\R)$ which lies dense in $\Ccal_c(\R)$. Then the following statements hold:
\begin{enumerate}[i)]
	\item Let $\mu,\, (\mu_n)_n\in\Mcal_1(\R)$, then the following statements are equivalent:
	\begin{enumerate}[a)]
	\item $\mu_n\to\mu$ weakly.
	\item $\forall\, k\in\N: \integrala{\mu_n}{g_k} \xrightarrow[n\to\infty]{} \integrala{\mu}{g_k}$.
	\end{enumerate}
	\item Define for all $\mu, \nu \in \Mcal_1(\R)$:
	\[
	d_M(\mu,\nu) \defeq \sum_{k\in\N} \frac{\abs{\integrala{\mu}{g_k}-\integrala{\nu}{g_k}}}{2^k\cdot (1+\abs{\integrala{\mu}{g_k}-\integrala{\nu}{g_k}})}.\label{sym:dmmetric}
	\]
	Then $d_M$ forms a metric on $\Mcal_1(\R)$ which metrizes weak convergence. That is, a sequence $(\mu_n)_{n\in\N}$ in $\Mcal_1(\R)$ converges weakly to $\mu\in\Mcal_1(\R)$ iff $d_M(\mu_n,\mu)\to 0$ as $n\to\infty$.
	\item $(\Mcal_1(\R),d_M)$ is a separable, but not complete, metric space.
\end{enumerate}
\end{theorem}

\begin{proof}
\underline{i)} Let $(\mu_n)_{n\in\N}$ and $\mu$ be probability measures. If $\mu_n\to \mu$ weakly, then surely we have for all $k\in\N$ that $\integrala{\mu_n}{g_k}\to\integrala{\mu}{g_k}$ as $n\to\infty$. If on the other hand we have for all $k\in\N$ that $\integrala{\mu_n}{g_k}\to\integrala{\mu}{g_k}$ as $n\to\infty$, then one easily sees that  $\mu_n$ converges vaguely to $\mu$, and then also weakly by Lemma~\ref{lem:vaguetoweak}.

\underline{ii) and iii)}:
From Lemma~\ref{lem:equalityoffinitemeasures}, we find for any $\mu,\nu\in\Mcal_1(\R)$ that 
\[
\mu=\nu~\Leftrightarrow~ \forall\, k\in\N:\integrala{\mu}{g_k} = \integrala{\nu}{g_k}.
\]

Next, we will inspect the space $\R^\N$ endowed with the product topology. With respect to this topology, a sequence $(z_n)_n$ in $\R^\N$ converges to a $z\in\R^\N$ iff for all $i\in\N$ the coordinates $z_n(i)$ in $\R$ converge to $z(i)$ as $n\to\infty$. Further, it is well-known that the topology on $\R^\N$ is metrizable through the metric $\rho$ with
\[
\forall\, x,y\in\R^{\N}:~~\rho(x,y)\defeq \sum_{k\in\N}\frac{\abs{x(k)-y(k)}}{2^k\cdot(1+\abs{x(k)-y(k)})}.
\]
This follows (for example) with 3.5.7 in 
\autocite[121]{Shirali}
in combination with Theorem 4.2.2 in 
\autocite[259]{Engelking}.
Further, $(\R^\N, \rho)$ is a  \emph{separable} metric space (Theorem 16.4 in 
\autocite[109]{Willard}).

We now define the following map (see \autocite[43]{Parthasarathy}):
\map{T}{\Mcal_1(\R)}{\R^\N}{\mu}{(\integrala{\mu}{g_1}, \integrala{\mu}{g_2},\ldots)}
Then surely, $T$ is injective, since if $T(\mu) = T(\nu)$, then also for all $k\in\N: \integrala{\mu}{g_k}=\integrala{\nu}{g_k}$ and then ${\mu=\nu}$. 
Additionally, we have for all $\mu,\nu\in\Mcal_1(\R)$ that
\begin{equation} 
\label{eq:isometric}
d_M(\mu,\nu) 	=   \sum_{k\in\N} \frac{\abs{\integrala{\mu}{g_k}-\integrala{\nu}{g_k}}}{2^k\cdot (1+\abs{\integrala{\mu}{g_k}-\integrala{\nu}{g_k}})} = \rho(T(\mu),T(\nu)).
\end{equation}

Since $T$ injective and $\rho$ is a metric, $d_M$ is a metric as well, so that $(\Mcal_1(\R), d_M)$ is a metric space. With equation \eqref{eq:isometric}
we see that $T: (\Mcal_1(X,d),d_M) \longrightarrow \R^\N$ is not only injective, but even isometric, especially continuous and a homeomorphism onto its image. Surely, the image is separable as a subspace of a separable metric space 
. Thus, $(\Mcal_1(\R), d_M)$, being homeomorphic to a separable space, is also separable (Corollary 1.4.11 in \autocite[31]{Engelking}).

With what we have shown so far we obtain for arbitrary $(\mu_n)_{n\in\N},\mu\in\Mcal_1(\R)$:
\begin{equation*}
\begin{aligned}
&\mu_n \textrm{ converges weakly to } \mu \\
\Leftrightarrow &~ \forall\, k\in\N: \integrala{\mu_n}{g_k} \xrightarrow[n\to\infty]{} \integrala{\mu}{g_k}\\
\Leftrightarrow &~ T(\mu_n) \xrightarrow[n\to\infty]{} T(\mu) \textrm{ in } \R^{\N}\\
\Leftrightarrow &~ \rho(T(\mu_n),T(\mu)) \xrightarrow[n\to\infty]{} 0\\
\Leftrightarrow &~ d_M(\mu_n,\mu) \xrightarrow[n\to\infty]{} 0.
\end{aligned}
\end{equation*}
We showed the first equivalence in the first part of this proof, the second equivalence holds per definition of $T$ and the above mentioned characterization of convergence in $\R^\N$, the third equivalence follows with the metrizability of $\R^\N$ through $\rho$, and the last equivalence follows from above equation~\eqref{eq:isometric}. What is left to show is that $(\Mcal_1(\R),d_M)$ is not complete. To this end, let $(\mu_n)_n$ be any sequence in $\Mcal_1(\R)$ which converges vaguely to a sub-probability measure $\nu$ with $\nu(\R) < 1$. Then for all $k\in\N$, $\integrala{\mu_n}{g_k}\to \integrala{\nu}{g_k}$ as $n\to\infty$. Thus, $d_M(\mu_n,\nu)\to 0$ as $n\to\infty$ (the function $d_M$ makes sense even with sub-probability measures as arguments). Since for any $n,m\in\N$, $d_M(\mu_n,\mu_m)\leq d_M(\mu_n,\nu) + d_M(\mu_m,\nu)$, we find that $(\mu_n)_n$ is a Cauchy sequence in $(\Mcal_1(\R),d_M)$ that does not converge weakly to an element in $\Mcal_1(\R)$.
\end{proof}

\section{Random Probability Measures on $(\R,\Bcal)$}

As we saw in Theorem~\ref{thm:weakconvergenceonR}, the set $\Mcal_1(\R)$ can be metrized in such a way that the resulting convergence is exactly "weak convergence of probability measures." This shows that Definition~\ref{def:weaknvague} was adequate in the sense that it defined weak convergence for sequences of probability measures rather than for nets. The reason is that in metric spaces (or more generally, in spaces which satisfy the first axiom of countability, which means that any point has a countable neighborhood basis), the topology can be reconstructed from the knowledge of convergent sequences rather than nets. This is due to the fact that a set in such a space is closed iff any limit of a convergent \emph{sequence} in the set is an element of the set.

From now on, we will always view $\Mcal_1(\R)$ as equipped with the topology of weak convergence and the associated Borel $\sigma$-algebra. We know that $\Mcal_1(\R)$ is separable and that $d_M$ as in Theorem~\ref{thm:weakconvergenceonR} is a metric yielding the topology of weak convergence. It is then a triviality that for any $f\in\Ccal_b(\R)$, the function
\map{I_f}{\Mcal_1(\R)}{\R}{\mu}{I_{f}(\mu)\defeq \integrala{\mu}{f}}
is continuous on $\Mcal_1(\R)$.

Since $\Mcal_1(\R)$ is now considered also as a measurable space, we can study $\Mcal_1(\R)$-valued random variables, which is the subject of this section.

\begin{definition}
\label{def:stochastickernel}
Let $(\Omega,\Acal,\Prob)$\label{sym:probability} be a probability space.
\begin{enumerate}[i)]
\item A \emph{random probability measure} on $(\R,\Bcal)$	is a measurable map $\mu:\Omega\to\Mcal_1(\R)$, $  \omega\mapsto\mu(\omega,\cdot)$.
\item A	\emph{stochastic kernel} from $(\Omega,\Acal)$ to $(\R,\Bcal)$ is a map $\mu: \Omega\times\Bcal\longrightarrow\R$, so that the following holds:
	 \begin{enumerate}[a)]
	\item For all $\omega\in\Omega$, $\mu(\omega,\cdot)$ is a probability measure on $(\R,\Bcal)$.
	\item For all $B\in\Bcal$, $\mu(\cdot,B)$ is $\Acal$-$\Bcal$-measurable.
\end{enumerate}
\end{enumerate}
\end{definition}

\begin{lemma}
\label{lem:randommeasureiskernel}
Let $(\Omega,\Acal,\Prob)$ be a probability space.
\begin{enumerate}[i)]
\item A map $\mu: \Omega\times\Bcal\longrightarrow\R$ is a random probability measure iff it is a stochastic kernel.
\item If $\mu$ is a stochastic kernel from $(\Omega,\Acal)$ to $(\R,\Bcal)$ and $f:\R\to\R$ is measurable and bounded, then $\omega\mapsto\integrala{\mu(\omega)}{f}$ is measurable and bounded by $\supnorm{f}$.
\end{enumerate}
  
\end{lemma}
\begin{proof}
We first show $ii)$:
Surely, the indicated map is bounded by $\supnorm{f}$, since we have for all $\omega\in\Omega$:
\[
\abs{\integrala{\mu(\omega)}{f}} \leq \integrala{\mu(\omega)}{\abs{f}} \leq \integrala{\mu(\omega)}{\supnorm{f}} \leq \supnorm{f}.
\]
To show measurability, we employ a standard extension argument: $\omega\mapsto\mu(\omega,B)$ is measurable for all $B\in\Bcal$. Let $f$ be a simple function on $(\R,\Bcal)$, that is, $f=\sum_{i=1}^{n}{\alpha_i\cdot\one_{B_i}}$ for some $n\in\N$, $\alpha_i\in\left[0,\infty\right)$ and $B_i\in\Bcal$, $i=1,\ldots,n$, then also $\omega\mapsto\integrala{\mu(\omega)}{f}=\sum_{i=1}^{n}{\alpha_i\cdot\mu(\omega,B_i)}$ is measurable as a linear combination of finitely many measurable functions.
Now let $f\geq 0$ be measurable and bounded, then there exists sequence of simple functions $(f_n)_{n\in\N}$ such that $f_n \nearrow_n f$ pointwise. For $\omega\in\Omega$ arbitrary it follows per monotone convergence that $\integrala{\mu(\omega)}{f_n}\nearrow_n \integrala{\mu(\omega)}{f}$, so also $\omega\mapsto\integrala{\mu(\omega)}{f}$ is measurable as a pointwise limit of measurable functions.
Now if $f:\R\longrightarrow\R$ is measurable and bounded, then also the positive and negative parts $f_+$ and $f_-$ (then $f_+, f_-\geq 0$ with $f=f_+ - f_-$). Then $\omega\mapsto\integrala{\mu(\omega)}{f}= \integrala{\mu(\omega)}{f_+} - \integrala{\mu(\omega)}{f_-}$ is measurable as a difference of measurable functions.

We now show $i)$:\newline
\underline{"$\Leftarrow$" } 
We have just shown that for all $f\in\Ccal_b(\R)$ the map $\omega\mapsto\integrala{\mu(\omega)}{f}$ is measurable. Then we obtain for all $\nu\in\Mcal_1(\R)$ that the map $\omega\mapsto d_M(\mu(\omega),\nu)$ is measurable as a limit of measurable functions, since
\[
d_M(\mu(\omega),\nu) = \sum_{k\in\N} \frac{\abs{\integrala{\mu(\omega)}{g_k}-\integrala{\nu}{g_k}}}{2^k\cdot (1+\abs{\integrala{\mu(\omega)}{g_k}-\integrala{\nu}{g_k}})}.
\]
To show the measurability of $\omega\mapsto\mu(\omega,\cdot)$, it suffices to show that preimages of open balls from $(\Mcal_1(\R),d_M)$ are measurable, since the $\sigma$-algebra on $\Mcal_1(\R)$ is generated by the topology which is generated by the metric $d_M$, and the space $\Mcal_1(\R)$ is separable with respect to the topology of weak convergence. So let $\nu\in\Mcal_1(\R)$ and $\varepsilon>0$ be arbitrary, then it holds with $B_{\varepsilon}^{\Mcal_1(\R)}(\nu)\defeq \{\nu'\in\Mcal_1(\R): d_M(\nu',\nu)<\varepsilon\}$:
\[
\mu^{-1}\left(B_{\varepsilon}^{\Mcal_1(\R)}(\nu)\right) = \{\omega\in\Omega: d_M(\mu(\omega),\nu) < \varepsilon\} = d_M(\mu(\cdot),\nu)^{-1}(\left[0,\varepsilon\right)) \in\Acal,
\]
since above we already recognized $d_M(\mu(\cdot),\nu)$ as measurable.

\underline{"$\Rightarrow$"} If $\mu$ is a random probability measure, then for all $\omega\in\Omega$, $\mu(\omega,\cdot)$ is a probability measure on $(\R,\Bcal)$. We now argue that for any $B\in\Bcal$, $\omega\mapsto\mu(\omega,B)$ is measurable. We first prove this for all open bounded intervals in $\R$, since these intervals generate $\Bcal$. So let $a<b\in\R$ be arbitrary and define $\varepsilon\defeq (b-a)/4$. Then define for all $n\in\N$ the function $\phi_n:\R\to\R$ so that $\phi_n\equiv 1$ on $[a+\frac{1}{n}\varepsilon,b-\frac{1}{n}\varepsilon]$, $\phi_n\equiv 0$ on $(a,b)^c$ and $\phi_n$ is affine on the intervals $[a,a+\frac{1}{n}\varepsilon]$ and $[b-\frac{1}{n}a,b]$ in such a way that it is continuous. Then $\phi_n$ is bounded, continuous and $\phi_n(x)\nearrow_n \one_{(a,b)}(x)$ for all $x\in\R$. We know that for all $n\in\N$, $\omega\mapsto \integrala{\mu(\omega)}{\phi_n}$ is measurable as a composition of a measurable and a continuous map (see remark before Definition~\ref{def:stochastickernel}). Now for any $\omega\in\Omega$:
\[
\lim_{n\to\infty}\integrala{\mu(\omega)}{\phi_n} = \integrala{\mu(\omega)}{\one_{(a,b)}} = \mu(\omega,(a,b)).
\]
by monotone convergence. As a result, $\mu(\cdot,(a,b))$ is $\Acal$-$\Bcal$-measurable as the pointwise limit of measurable functions. Now define the set
\[
\Gcal\defeq\{B\in\Bcal\,|\, \omega\mapsto\mu(\omega,B) \text{ is measurable}\}.
\]
Surely, all open intervals lie in $\Gcal$ as we have just shown. If we can show that $\Gcal$ is a Dynkin system we can conclude that $\Gcal=\Bcal$, which is our goal.
 First of all, $\emptyset$ \label{sym:emptyset}, $\R\in\Gcal$, since constant functions are always measurable. Second, since $\mu(\cdot,B^c)=1-\mu(\cdot,B)$, we have that $B^c\in\Gcal$ whenever $B\in\Gcal$. Third, if $(B_n)_n$ is a sequence of pairwise disjoint sets in $\Gcal$, then $\mu(\cdot,\cup_n B_n)=\sum_n\mu(\cdot,B_n)$, so since all $\mu(\cdot,B_n)$ are measurable, then so is $\mu(\cdot,\cup_n B_n)$ as a pointwise limit of a sequence of measurable functions. This shows that $\cup_n B_n \in\Gcal$ so that $\Gcal$ is indeed a Dynkin system.
\end{proof}

Random probability measures are not so uncommon in probability theory. Consider the next example:

\begin{example}
\label{ex:empiricaldist}
Let $Y_1,\ldots,Y_n$ be real-valued random variables on a probability space $(\Omega,\Acal,\Prob)$. Then
\[
\rho \defeq \frac{1}{n}\sum_{i=1}^n\delta_{Y_i}
\]
is a random probability measure on $(\R,\Bcal)$, which we call \emph{empirical distribution} (of the $Y_i$).
Indeed, for any $\omega\in\Omega$,
\[
\rho(\omega) = \frac{1}{n}\sum_{i=1}^n\delta_{Y_i(\omega)}
\]
is a convex combination of probability measures and thus again a probability measure on $(\R,\Bcal)$. On the other hand, if $B\in\Bcal$ is arbitrary, then 
\[
\omega\mapsto \rho(\omega,B) = \frac{1}{n}\sum_{i=1}^n\delta_{Y_i(\omega)}(B) = \frac{1}{n}\sum_{i=1}^n\one_B(Y_i(\omega))
\]
is certainly measurable. Thus, we recognize the empirical distribution $\rho$ as a random probability measure on $(\R,\Bcal)$ via Lemma~\ref{lem:randommeasureiskernel}. For any measurable set $B$, $\rho(B)$ yields the proportion of the $Y_i$'s that fall into the set $B$. Connected to the empirical distribution $\rho$ is its empirical distribution function $F_{\rho}(x) \defeq \rho((-\infty,x])$ defined for all $x\in\R$. This is a random distribution function and the protagonist of the famous Glivenko-Cantelli theorem and the Dvoretzky–Kiefer–Wolfowitz inequality, see \autocite[553]{Witting2}.
\end{example}

Now, let us resume our study. If $\mu$ is a random probability measure and $B\in\Bcal$, then $\mu(B)$ is a bounded random variable. It is natural to consider its expectation $\E\mu(B)$\label{sym:expectedvalue} as the expected mass that $\mu $ prescribes to the set $B$. But as it turns out, $B\mapsto\E\mu(B)$ is yet another (deterministic) probability measure:

\begin{theorem}
\label{thm:expectedmeasure}
Let $(\Omega,\Bcal,\Prob)$ be a probability space and $\mu$ be a random probability measure on $(\R,\Bcal)$. Then the following statements hold:
\begin{enumerate}[i)]
\item  The map
\map{\bar{\mu}}{\Bcal}{\left[0,1\right]}{B}{\bar{\mu}(B)\defeq\integralc{\Prob(\de \omega)}{\mu(\omega,B)}{\Omega} = \E\mu(B)}

is an element of $\Mcal_1(\R)$, the so called \emph{expected measure} of $\mu$.
\item  Any non-negative measurable function $f:\R\longrightarrow\R_+$\label{sym:realnonneg} is $\bar{\mu}$-integrable iff $\integrala{\mu}{f}$ is $\Prob$-integrable, and in this case it holds
\[
\integrala{\bar{\mu}}{f}=\integralc{\bar{\mu}(\de x)}{f(x)}{\R}=
\integralc{\Prob(\de \omega)}{\integralc{\mu(\omega,\de x)}{f(x)}{\R}}{\Omega} = \E\integrala{\mu}{f}.
\]
In particular, this equation is valid for any bounded measurable function $f:\R\to\R$.
\item If $f:\R\to\R$ is $\bar{\mu}$-integrable, then $\integrala{\mu}{f}$ is $\Prob$-integrable and $\integrala{\bar{\mu}}{f}=\E\integrala{\mu}{f}$.
\item Heed must be taken: If $f:\R\to\R$ is measurable and such that $\integrala{\mu}{f}$ is $\Prob$-integrable so that $\E\integrala{\mu}{f}$ is well-defined, $f$ need not be $\bar{\mu}$-integrable, so that it is \emph{not true} that $\integrala{\bar{\mu}}{f}=\E\integrala{\mu}{f}$ whenever one of the two exists. In particular, statement $ii)$ cannot be generalized to arbitrary measurable functions $f:\R\to\R$.
\end{enumerate}

Due to these interrelations we will also write $\E \mu$ \label{sym:expectedmeasure}instead of $\bar{\mu}$, and with what we have seen so far it holds for all function $f$ with $\E\integrala{\mu}{\abs{f}}<\infty$ that $f$ is $\E\mu$-integrable with
\[
\integrala{\E\mu}{f}=\integrala{\bar{\mu}}{f} = \E \integrala{\mu}{f}.
\]
\end{theorem}
\begin{proof}
\underline{$i)$} Clearly, $\E \mu(\emptyset) = 0$ and $\E \mu(\R)=1$. Now if $(B_n)_n$ is a sequence of pairwise disjoint elements in $\Bcal$, then 
\[
\E \mu (\cup_n B_n) = \E \sum_n \mu(B_n) = \sum_n \E\mu(B_n),
\]
where in the last step we used dominated convergence. This shows that $\bar{\mu}$ is indeed a probability measure.\newline
\underline{$ii)$} Let $f:\R\to\R$ be a simple function, that is, $f=\sum_{i=1}^{n}{\alpha_i\cdot\one_{B_i}}$ for some $n\in\N$, $\alpha_i\in\left[0,\infty\right)$ and $B_i\in\Bcal$, $i=1,\ldots,n$. Then
\[
\integrala{\bar{\mu}}{f} = \sum_{i=1}^{n}{\alpha_i\cdot\bar{\mu}(B_i)} = \E \sum_{i=1}^{n}{\alpha_i\cdot\mu(B_i)} = \E\integrala{\mu}{f}.
\]
Now let $f:\R\to\R$ be non-negative and measurable witnessed by a sequence of simple functions $(f_n)_n$ with $f_n\nearrow_n f$ pointwise, then clearly
\[
\integrala{\bar{\mu}}{f} = \lim_{n\to\infty}\integrala{\bar{\mu}}{f_n} = \lim_{n\to\infty} \E\integrala{\mu}{f_n} = \E\integrala{\mu}{f},
\]
where in the first and the last step we used monotone convergence. In particular, the non-negative $f$ is $\bar{\mu}$-integrable iff $\integrala{\mu}{f}$ is $\Prob$-integrable and in this case it holds $\integrala{\bar{\mu}}{f} = \E\integrala{\mu}{f}$. Now if $f:\R\to\R$ is bounded, then there exists a $C\in\R$ such that $f+C$ is non-negative (and of course, it remains bounded, thus integrable). Then we immediately obtain $\integrala{\bar{\mu}}{f} = \integrala{\bar{\mu}}{f+C} -C = \E\integrala{\mu}{f+C} - C = \E\integrala{\mu}{f}$.\newline
\underline{$iii)$} If now $f:\R\to\R$ is $\bar{\mu}$-integrable, then $f=f_+ - f_-$ where $f_+,f_-\geq 0$ are $\bar{\mu}$-integrable. By $ii)$, the non-negative random variables $\integrala{\mu}{f_+}$ and $\integrala{\mu}{f_-}$ are both $\Prob$-integrable. Then their difference $\integrala{\mu}{f_+}-\integrala{\mu}{f_-} = \integrala{\mu}{f}$ is also $\Prob$-integrable and we obtain with $ii)$:
\[
\integrala{\bar{\mu}}{f}=\integrala{\bar{\mu}}{f_+}-\integrala{\bar{\mu}}{f_-}=\E\integrala{\bar{\mu}}{f_+}-\E\integrala{\bar{\mu}}{f_-} = \E\integrala{\bar{\mu}}{f}.
\]
\underline{$iv)$} Unfortunately, this point appears to be overlooked in the literature. We need to construct a counter-example to show what we state. To this end, consider the random probability measure $\mu$ on $(\R,\Bcal)$ with
\[
\forall\,n\in\N : \Prob\left(\mu = \frac{1}{2}\delta_{-n} + 
\frac{1}{2}\delta_{n}\right) = \frac{1}{c n^2},
\]
where $c\defeq \sum_n \frac{1}{n^2}<\infty$. Further, let $f$ be the identity on $\R$, that is, $f(x)=x$ for all $x\in\R$. Then surely, $f$ is measurable, and since almost all realizations of $\mu$ are symmetric measures, we have $\integrala{\mu}{f}=0$ almost surely, which is $\Prob$-integrable with $\E\integrala{\mu}{f}=0$. We now assume that $f$ is $\bar{\mu}$-integrable and lead this to a contradiction: If $f$ were $\bar{\mu}$-integrable, then so would $\abs{f}$ and by $ii)$ we would have $\integrala{\bar{\mu}}{\abs{f}}=\E\integrala{\mu}{\abs{f}} <\infty$. But with probability $\frac{1}{c n^2}$, $\mu$ takes the value $\frac{1}{2}\delta_{-n}+\frac{1}{2}\delta_{n}$, so $\integrala{\mu}{\abs{f}}$ takes the value $n$, leading to the calculation
\[
\E\integrala{\mu}{\abs{f}} = \sum_{n\in\N}\frac{n}{c n^2}=\infty,
\]
which is a contradiction.
\end{proof}

In the remainder of this section, we will derive and discuss three notions of convergence of random probability measures on $(\R,\Bcal)$, namely weak convergence in expectation, weak convergence in probability and weak convergence almost surely.

\begin{definition}
\label{def:convergenceinexpectation}
Let $(\mu_n)_{n\in\N}$ and $\mu$ be random probability measures on $(\R,\Bcal)$, then we say that $(\mu_n)_n$ converges weakly in expectation to $\mu$, if the sequence of expected measures $(\E\mu_n)_{n\in\N}$ converges weakly to the expected measure $\E\mu$, so if:
\[
\forall\, f\in\Ccal_b(\R):\integrala{\E\mu_n}{f}\xrightarrow[n\to\infty]{}\integrala{\E\mu}{f},
\] 
which is equivalent to (see Theorem~\ref{thm:expectedmeasure})
\[
\forall\, f\in\Ccal_b(\R):\E\integrala{\mu_n}{f}\xrightarrow[n\to\infty]{}\E\integrala{\mu}{f}.
\] 
\end{definition}

The concept of weak convergence in expectation is extremely important for investigations in the field of random matrix theory, since it lies the foundation for stronger convergence types. This is due to the fact that weak convergence $\Prob$-almost surely or in probability will also imply weak convergence in expectation, so the latter convergence type is  a necessary condition for stronger convergence types (see also Theorem~\ref{thm:randommomentconvergence}). The exact interrelations between the three concepts of convergence for random probability measures are summarized in the end of this section in Theorem~\ref{thm:randomconvergenceimplications}.

Before turning to the next convergence types, we wish to remind the reader what convergence in probability and almost surely means for random variables in metric spaces:
\begin{definition}
Let $(Y_n)_{n\in\N}$ and $Y$ be random variables defined on a probability space $(\Omega,\Acal,\Prob)$, which take values in a metric space $(\Xcal,d)$\label{sym:metricspace}.
\begin{enumerate}[i)]
\item We say that $(Y_n)_{n\in\N}$ converges to $Y$ in probability, if	$d(Y_n,Y)$ converges to $0$ in probability.
\item We say that $(Y_n)_{n\in\N}$ converges to $Y$ almost surely, if $d(Y_n,Y)$ converges to $0$ almost surely.
\end{enumerate}
\end{definition}

Let us collect a quick lemma:

\begin{lemma}
\label{lem:metricsubsequence}
Let $(Y_n)_{n\in\N}$ and $Y$ be random variables defined on a probability space $(\Omega,\Acal,\Prob)$, which take values in a metric space $(\Xcal,d)$. If $(Y_n)_{n\in\N}$ converges to $Y$ almost surely, then also in probability.
\end{lemma}
\begin{proof}
Let $(Y_n)_{n\in\N}$ converge to $Y$ almost surely. This means that the sequence of real-valued random variables $(d(Y_n,Y))_n$ converges to $0$ almost surely. But this implies that $(d(Y_n,Y))_n$ converges to $0$ in probability, which is precisely what it means for $(Y_n)_n$ to converge to $Y$ in probability.
\end{proof}

Now let us define and analyze what it means for random probability measures to converge in probability and almost surely. Since random probability measures are nothing but random variables into the metric space $\Mcal_{1}(\R)$, we know what to do:

\begin{definition}
\label{def:stochasticweakconvergence}
Let $(\Omega,\Acal,\Prob)$ be a probability space,  $\mu$ and $(\mu_n)_{n\in\N}$ be random probability measures on $(\R,\Bcal)$.
\begin{enumerate}[i)]
	\item We say that $(\mu_n)_n$ converges weakly to $\mu$ in probability, if $d_M(\mu_n,\mu)$ converges to $0$ in probability.
\item We say that $(\mu_n)_n$ converges weakly to $\mu$ almost surely, if $d_M(\mu_n,\mu)$ converges to $0$ almost surely.
\end{enumerate}
\end{definition}

Although stochastic types of weak convergence can be defined solidly as in Definition~\ref{def:stochasticweakconvergence}, this definition is not convenient to work with in practice. In addition, we would like to see that these convergence concepts do \emph{not} depend on the choice of the metric that metrizes weak convergence on $\Mcal_1(\R)$. 

\begin{theorem}
\label{thm:stochasticweakconvergence}
Let $(\Omega,\Acal,\Prob)$ be a probability space, $\mu$ and $(\mu_n)_{n\in\N}$ be random probability measures on $(\R,\Bcal)$. 
\begin{enumerate}[i)]
	\item The following statements are equivalent:
		\begin{enumerate}[a)]
		\item $(\mu_n)_n$ converges weakly to $\mu$ in probability, that is, $d_M(\mu_n,\mu)\to 0$ in probability. 
		\item If $d$ is any metric on $\Mcal_1(\R)$ that metrizes weak convergence, then $d(\mu_n,\mu)\to 0$ in probability.
		\item For all $f\in\Ccal_b(\R)$, the sequence of bounded real-valued random variables $(\integrala{\mu_n}{f})_n$ converges in probability to $\integrala{\mu}{f}$, so
			\[
			\forall\, f\in \Ccal_b(\R):\,\forall\,\varepsilon>0: \Prob(\abs{\integrala{\mu_n}{f}-\integrala{\mu}{f}}>\varepsilon)\xrightarrow[n\to\infty]{} 0.
			\]
		\end{enumerate}	

	\item The following statements are equivalent:
	\begin{enumerate}[a)]
		\item $(\mu_n)_n$ converges weakly to $\mu$ almost surely, that is, $d_M(\mu_n,\mu)\to 0$ almost surely. 
		\item For $\Prob$-almost all $\omega\in\Omega$, $\mu_n(\omega)$ converges weakly to $\mu(\omega)$.
		\item If $d$ is any metric on $\Mcal_1(\R)$ that metrizes weak convergence, then $d(\mu_n,\mu)\to 0$  almost surely.
		\item For all $f\in\Ccal_b(\R)$, $\integrala{\mu_n}{f}$ converges almost surely to $\integrala{\mu}{f}$, that is,
			\[
			\forall\, f\in \Ccal_b(\R): \left[\integrala{\mu_n}{f}\xrightarrow[n\to\infty]{}\integrala{\mu}{f} ~ \textrm{almost surely}\right].
			\]
		\item Almost surely we find that for all $f\in\Ccal_b(\R)$, $\integrala{\mu_n}{f}$ converges  to $\integrala{\mu}{f}$, that is,
			\[
			\left[\forall\, f\in \Ccal_b(\R): \integrala{\mu_n}{f}\xrightarrow[n\to\infty]{}\integrala{\mu}{f}\right] \quad \textrm{almost surely}.
			\]
	\end{enumerate}
\end{enumerate}
\end{theorem}

\begin{remark}
\label{rem:stochasticweakconvergence}
\begin{enumerate}
\item Note that in Theorem~\ref{thm:stochasticweakconvergence} $ii)$ $d)$ and $e)$ we used careful bracketing $[\ldots]$ when it comes to almost sure convergence of multiple objects. This is done to avoid ambiguity. For example, questions could arise whether we find a set of measure $1$ on which all objects converge (as in $e)$), or if for each object, we find a set of measure $1$, \emph{possibly depending on that object}, on which the considered object converges (as in $d)$).  
\item	We consider Theorem~\ref{thm:stochasticweakconvergence} $i)$ as equivalent definitions for the concept "weak convergence in probability", and $ii)$ as equivalent definitions for "weak convergence almost surely." After the proof of the theorem, we will keep on working with this characterization without always referring to Theorem~\ref{thm:stochasticweakconvergence}. 
\end{enumerate}
\end{remark}

Before we begin with the proof of Theorem~\ref{thm:stochasticweakconvergence}, we will introduce two tools which we will make use of. For later use, we will formulate the lemmas in greater generality, that is, for complex-valued random variables.

\begin{lemma}
\label{lem:subsequence}
Let $(X_n)_n$ and $X$ be complex-valued random variables defined on a probability space $(\Omega,\Acal,\Prob)$. Then $(X_n)_{n\in\N}$ converges to $X$ in probability iff any subsequence  $J\subseteq\N$ has another subsequence $I\subseteq J$ so that $(X_n)_{n\in I}$ converges to $X$ almost surely.
\end{lemma}
\begin{proof}
The proof can be found in \autocite[134]{Klenke}	.
\end{proof}

The next extremely useful lemma generalizes the previous one by finding a \emph{simultaneous almost surely convergent subsequence} for a countable number of sequences of random variables.

\begin{lemma}
\label{lem:uniformsubsequence}
Let $(\Omega,\Acal,\Prob)$ be a probability space and for all $k\in\N$ let $X^{(k)}$ and $(X^{(k)}_n)_{n\in\N}$ be complex-valued random variables. Then the following statements are equivalent:
\begin{enumerate}[i)]
\item For all $k\in\N$, $(X_n^{(k)})_n$ converges to $X^{(k)}$ in probability.
\item For any subsequence $J\subseteq \N$, we find a subsequence $I\subseteq J$ and a set $N\in \Acal$ with $\Prob(N)=0$ such that 
\[
\forall\, \omega \in \Omega\backslash N: \forall\, k\in\N: X^{(k)}_n(\omega)\xrightarrow[n\in I]{} X^{(k)}(\omega).
\]	
\end{enumerate}
\end{lemma}
\begin{proof} 
The part $ii)\Rightarrow i)$ follows immediately with Lemma~\ref{lem:subsequence}. So we only need to show $i)\Rightarrow ii)$:
For $k=1$ we find that $(X^{(1)}_n)_{n\in J}$ converges in probability to $X^{(1)}$. Therefore, we find a subsequence $I_1\subseteq J$ such that
\[
X^{(1)}_n \xrightarrow[n\in I_1]{} X^{(1)} \quad\text{$\Prob$-a.s. witnessed by a set of measure zero $N_1$}.
\] 
Since $(X^{(2)}_n)_{n\in I_1}$ converges to $X^{(2)}$ in probability, we find a subsequence $I_2\subseteq I_1$ with $\min(I_2)>\min(I_1)$ such that 
\[
X^{(2)}_n \xrightarrow[n\in I_2]{} X^{(2)} \quad\text{$\Prob$-a.s. witnessed by a set of measure zero $N_2$}.
\]
We continue this approach for all $k\in\N$ and obtain subsequences
\[
\N \supseteq J \supseteq I_1 \supseteq I_2 \supseteq \ldots \supseteq I_k \supseteq \ldots
\]

such that for all $k\in\N$ we have $\min(I_{k+1}) > \min(I_k)$ and 
\[
X^{(k)}_n \xrightarrow[n\in I_k]{} X^{(k)} \quad\text{$\Prob$-a.s. witnessed by a set of measure zero $N_k$}.
\]
We set $N\defeq \cup_{k\in \N}N_k$ and for all $k\in\N: i_k \defeq \min(I_k)$, then we obtain that $(i_k)_{k\in\N}$ is strictly increasing in $\N$ and 
\[
\forall\, \omega\in\Omega\backslash N: \,\forall\,l\in\N:  X^{(l)}_{i_k}(\omega) \xrightarrow[k\in\N]{} X^{(l)}(\omega).
\]
To see this, let $\omega\in\Omega\backslash N$ and $l\in\N$ be arbitrary. Then we have that $\omega\in\Omega\backslash N_l$ and $i_k=\min(I_k)\in I_l$ for all $k\geq l$, so that indeed
\[
X^{(l)}_{i_k}(\omega)\xrightarrow[k\in\N]{} X^{(l)}(\omega).
\]
The proof is completed by setting $I\defeq \{i_k\,|\,k\in\N\}$.
\end{proof}

Now we are ready to prove Theorem~\ref{thm:stochasticweakconvergence}:
\begin{proof}[Proof of Theorem~\ref{thm:stochasticweakconvergence}]

We show $ii)$ first. \newline
Clearly, $a)$, $b)$ and $c)$ are equivalent, since the metrics metrize weak convergence. Also, $e)$ is just a reformulation of $b)$, thus equivalent. In addition, $d)$ follows immediately from $e)$, so we have
\[
a)~\Leftrightarrow ~b)~\Leftrightarrow ~c)~\Leftrightarrow ~e)~ \Rightarrow~ d)
\]
We now show $\underline{d)\Rightarrow b)}:$ For each $k\in\N$ we have that $\integrala{\mu_n}{g_k}$ converges to $\integrala{\mu}{g_k}$ almost surely on a set $A_k$ of measure $1$ (the functions $(g_k)_k$ are as in Theorem~\ref{thm:weakconvergenceonR}). Then the set $\Omega_1\defeq \cap_k A_k$ has measure $1$ and for all $\omega\in\Omega_1$ we find that
\[
\forall\, k\in \N: \integrala{\mu_n(\omega)}{g_k}\xrightarrow[n\to\infty]{}\integrala{\mu(\omega)}{g_k}.
\]
Therefore, with Theorem~\ref{thm:weakconvergenceonR}, we have for all $\omega\in\Omega_1$  that $\mu_n(\omega) \to \mu(\omega)$ weakly as $n\to\infty$ and hence $b)$.\newline
We now show $i)$:\newline
\underline{$a)\Leftrightarrow b)$} By exact symmetry in the argument, we will only argue $a)\Rightarrow b)$: Let $\mu_n\to\mu$ weakly in probability, that is, $(d_M(\mu_n,\mu))_{n\in\N}$ converges to $0$ in probability. We want to show that also $(d(\mu_n,\mu))_{n\in\N}$ converges to $0$ in probability. To use Lemma~\ref{lem:subsequence}, let $J\subseteq\N$ be an arbitrary subsequence. Then we find a subsequence $I\subseteq J$ such that $(d_M(\mu_n,\mu))_{n\in I}$ converges to $0$ almost surely. With part $ii)$ this means that also $(d(\mu_n,\mu))_{n\in I}$ converges to $0$ almost surely. But then $(d(\mu_n,\mu))_{n\in\N}$ converges to $0$ in probability. \newline
\underline{$a)\Rightarrow c)$} If $(\mu_n)_n$ converges weakly to $\mu$ in probability, then this means that $d_M(\mu_n,\mu)$ converges to $0$ in probability. Let $f\in\Ccal_b(\R)$ be arbitrary. We must show that $\integrala{\mu_n}{f}$ converges to $\integrala{\mu}{f}$ in probability. To this end, let $J\subseteq \N$ be an arbitrary subsequence. Then there is a subsequence $I\subseteq J$ such that $(d_M(\mu_n,\mu))_{n\in I}$ converges to $0$ almost surely on a measurable subset $\Omega_1\subseteq \Omega$ with measure $1$. Then it holds in particular for any $\omega\in\Omega_1$ that $(\integrala{\mu_n(\omega)}{f})_{n\in I}$ converges to $\integrala{\mu(\omega)}{f}$, so $(\integrala{\mu_n}{f})_{n\in I}$ converges to $\integrala{\mu}{f}$ almost surely. The statement follows with Lemma~\ref{lem:subsequence}.\newline
\underline{$c)\Rightarrow a)$} We find that for all $k\in\N$, $(\integrala{\mu_n}{g_k})_{n\in\N}$ converges to $\integrala{\mu}{g_k}$ in probability. We must show that $d_M(\mu_n,\mu)$ converges to zero in probability. Let $J\subseteq \N$ be any subsequence. With Lemma~\ref{lem:uniformsubsequence}, we find a subsequence $I\subseteq J$ and a measurable set $\Omega_1\subseteq\Omega$ of measure $1$, such that 
\[
\forall\, \omega \in \Omega_1: \forall\, k\in\N: \integrala{\mu_n(\omega)}{g_k}\xrightarrow[n\in I]{} \integrala{\mu(\omega)}{g_k}
\]
With Theorem~\ref{thm:weakconvergenceonR}, this entails that for all $\omega\in\Omega_1$, $(d_M(\mu_n(\omega),\mu(\omega)))_{n\in I}$ converges to $0$. With Lemma~\ref{lem:subsequence}, this means that $(d_M(\mu_n,\mu))_{n\in\N}$ converges to zero in probability.
\end{proof}

So, what we have seen so far is that random probability measures can converge in three different ways, namely weakly in expectation, weakly in probability and weakly almost surely. We have solidly defined and then characterized these convergence concepts. At last, we point out a hierarchy among them:

\begin{theorem}
\label{thm:randomconvergenceimplications}
Let $(\Omega,\Acal,\Prob)$ be a probability space, $(\mu_n)_{n\in\N}$ and $\mu$ be random probability measures on  $(\R,\Bcal)$.
\begin{enumerate}[i)]
	\item If $\mu_n\to\mu$ weakly almost surely, then also weakly in probability.
	\item If $\mu_n\to\mu$ weakly in probability, then also weakly in expectation.
\end{enumerate}
\end{theorem}
\begin{proof}
\underline{$i)$}	This follows directly with Lemma~\ref{lem:metricsubsequence}.\newline
\underline{$ii)$} If $\mu_n\to\mu$ weakly in probability, per Theorem~\ref{thm:stochasticweakconvergence} this means that for all $f\in\Ccal_b(\R)$ we find $\integrala{\mu_n}{f}\to\integrala{\mu}{f}$ in probability, thus $\E\integrala{\mu_n}{f}\to\E\integrala{\mu}{f}$ by the following Lemma~\ref{lem:stochastictoexpectation}, since $\abs{\integrala{\mu_n}{f}}\leq \supnorm{f}$ and $\abs{\integrala{\mu}{f}}\leq \supnorm{f}$.
\end{proof}

\begin{lemma}
\label{lem:stochastictoexpectation}
Let $(X_n)_{n\in\N}$ and $X$ be complex-valued random variables on a probability space $(\Omega,\Acal,\Prob)$ and $C\in\R$ such that $\abs{X_n}\leq C$ for all $n\in\N$ and $\abs{X}\leq C$. Then $X_n\to X$ in probability implies $\E\abs{X_n-X}\to 0$, in particular $\E X_n \to \E X$. 
\end{lemma}
\begin{proof}
Let $\varepsilon > 0$ be arbitrary, then we calculate:
\begin{align*}
\E \abs{X_n - X}&= \E \abs{X_n - X}\one_{\{\abs{X_n-X}\leq\varepsilon\}} + \E \abs{X_n - X}\one_{\{\abs{X_n-X}>\varepsilon\}}\\
&\leq \varepsilon + \Prob(\abs{X_n - X}>\varepsilon)\cdot 2 C.
\end{align*}
Therefore, we conclude
\[
\limsup_{n\to\infty}\E\abs{X_n - X}  \leq \varepsilon.
\]
\end{proof}

\section{Limit Laws in Random Matrix Theory}

We will now  introduce the types of random probability measures which we would like to investigate, namely the empirical spectral distribution of random matrices. To this end, let $\K\in\{\R,\C\}$\label{sym:realorcomplex} and denote by $(\Mat{n}{\K},\opnorm{\cdot})$ \label{sym:matrices} the normed $\K$-vector space of $n\times n$-matrices with $\K$-valued entries, where $\opnorm{\cdot}$ denotes the operator norm with respect to the euclidian norm $\norm{\cdot}$ \label{sym:norm} on $\K^n$, that is,
\[
\,\forall\, X\in\Mat{n}{\K}: \opnorm{X} = \sup\left\{\norm{Xv}: v\in\K^n, \norm{v}=1 \right\}.
\]
It is immediate that $(\Mat{n}{\K},\opnorm{\cdot})$ is a Banach-space and a sequence of matrices $(X_m)_m$ converges to a matrix $X$ in $\Mat{n}{\K}$, iff all entries $X_m(i,j)$ converge to $X(i,j)$ in $\K$ as $m\to\infty$. If $X\in\Mat{n}{\K}$ we denote its \emph{adjoint} by $X^*$, which is just the transpose of $X$ if $\K=\R$ and the conjugate transpose of $X$ if $\K=\C$. A matrix $X\in\Mat{n}{\K}$ is called \emph{self-adjoint} if $X^*=X$ (then $X$ is also called \emph{symmetric} if $\K=\R$ and \emph{Hermitian} if $\K=\C$) and we denote the subset of all self-adjoint matrices of $\Mat{n}{\K}$ by $\SMat{n}{\K}$\label{sym:selfadmatrices}. Then $\SMat{n}{\K}\subseteq \Mat{n}{\K}$ is a closed subset, since $X\mapsto X^*$ is continuous. Further, $\SMat{n}{\K}$ is closed under $\R$-linear combinations.
To introduce more notation, if $\lambda_1,\ldots,\lambda_n\in\R$ are arbitrary, we denote by $\diag(\lambda_1,\ldots,\lambda_n)$\label{sym:diagmat} the diagonal matrix $D\in\SMat{n}{\K}$ with entries $D(i,i)=\lambda_i$ for all $i\in\{1,\ldots,n\}$. Further, we denote by $\tr$\label{sym:trace} the trace functional $\Mat{n}{\K}\longrightarrow \K$, that is,
\[
\,\forall\, X\in\Mat{n}{\K}: \tr X = \sum_{t=1}^{n}X(t,
t) .
\]
The trace has some interesting properties, which are summarized in the following lemma:
\begin{lemma}
\label{lem:trace}
The trace $\tr$ is a continuous linear functional on $(\Mat{n}{\K},\opnorm{\cdot})$. Further, if $X,S\in \Mat{n}{\K}$ are arbitrary, where $S$ is invertible, then $\tr(X)=\tr(S^{-1}XS)$.
\end{lemma}
\begin{proof}
It is immediate that the trace is a continuous linear functional. The equality $\tr(X)=\tr(S^{-1}XS)$ is due to the fact that $X$ and $S^{-1} X S$ have the same characteristic polynomial. The trace is the $(n-1)$th coefficient of the characteristic polynomial (multiplied by $(-1)^{n-1}$). 
 For details we refer the reader to \autocite{Fischer}.
\end{proof}

The next lemma clarifies the eigenvalue structure of self-adjoint matrices: 
\begin{lemma}
\label{lem:symmetricspectrum}
For any matrix $X\in\SMat{n}{\K}$ we find an invertible matrix $S\in\Mat{n}{\K}$ and real numbers $\lambda^X_1\leq\ldots\leq\lambda^X_n$, such that  $S^{-1}XS = \diag(\lambda^X_1,\ldots,\lambda^X_n)$. In particular, $X$ has exactly $n$ real eigenvalues (counting multiplicities), and all eigenvalues are real.
\end{lemma}
\begin{proof}
We refer the reader to \autocite{Fischer}.
\end{proof}

In general, if $Y$ is a self-adjoint $n\times n$ matrix, we will denote its $n$ real eigenvalues by $\lambda^Y_1\leq\ldots\leq \lambda^Y_n$. \label{sym:eigenvalue}
The next theorem is a very versatile tool in random matrix theory. For example, it can be used to derive that eigenvalues are continuous functions of the entries of the matrix (Corollary~\ref{cor:eigencont}), or it can be used to analyze asymptotic equivalence of empirical spectral distributions via the bounded Lipschitz metric.
\begin{theorem}[Hoffman-Wielandt]
\label{thm:hoffman-wielandt}
For all $n\in\N$ and $X,Y\in\SMat{n}{\K}$ it holds:
\[
\sum_{i=1}^{n}\abs{\lambda_{i}^{X}-\lambda_{i}^{Y}}^2\leq \tr(X-Y)^*(X-Y)=\tr(X-Y)^2.
\]
\end{theorem}
\begin{proof}
See \autocite[217]{Hanke}.
\end{proof}

We can immediately conclude that eigenvalues are continuous functions of the matrices.

\begin{corollary}
\label{cor:eigencont}
Let $n\in\N$ and $l\in\{1,\ldots,n\}$ be arbitrary, then 
\map{\mathrm{Eig}_l}{\SMat{n}{\K}}{\R}{X}{\lambda_l^X}
is continuous.
\end{corollary}
\begin{proof}
Let $(X_m)_{m\in\N}$ and $X$ in $\SMat{n}{\K}$ so that $X_m\to X$ for $m\to\infty$ (which means convergence in operator norm, or equivalently, entry-wise convergence). Then we find with Theorem~\ref{thm:hoffman-wielandt} and Lemma~\ref{lem:trace} that
\[
\abs{\lambda_{l}^{X_m}-\lambda_{l}^{X}}^2\leq\sum_{i=1}^N\abs{\lambda_{i}^{X_m}-\lambda_{i}^{X}}^2\leq \tr(X_m-X)^2 \xrightarrow[m\to\infty]{}0.
\]
\end{proof}

Having studied eigenvalues of self-adjoint matrices, let us turn our attention to random matrices.

\begin{definition}
\label{def:randommatrix}
Let $(\Omega,\Acal,\Prob)$ be a probability space and $n\in\N$ be arbitrary then a ($n\times n$ self-adjoint) random matrix is a measurable map $X:(\Omega,\Acal)\to (\SMat{n}{\K},\Bcal^{(n^2)}_s$), where $\Bcal^{(n^2)}_s$\label{sym:matrixborel} denotes Borel $\sigma$-algebra on $\SMat{n}{\K}$.
\end{definition}

It is clear that a map $X:(\Omega,\Acal)\to (\SMat{n}{\K},\Bcal_{s}^{(n^2)})$ is measurable iff all entries $X(i,j):(\Omega,\Acal)\to(\K,\Bcal_{\K})$\label{sym:realcomplexborel} are measurable, where $\Bcal_{\K}$ denotes the Borel $\sigma$-algebra on $\K$. If $X$ is an $n\times n$ random matrix on $(\Omega,\Acal,\Prob)$, then for all $\omega\in\Omega$, $X(\omega)\in\SMat{n}{\K}$, such that $X(\omega)$ possesses eigenvalues $\lambda_{1}^{X(\omega)}\leq\ldots\leq\lambda_{n}^{X(\omega)}$. We wish to see that the maps $\omega\mapsto\lambda_{l}^{X(\omega)}$ for $l=1,\ldots,n$ are measurable.

\begin{lemma}
\label{lem:evmeasurable}
Let $X$ be an $n\times n$-random matrix on $(\Omega,\Acal,\Prob)$ and $l\in\{1,\ldots,n\}$ be arbitrary, then
\map{\lambda_{l}^{X}}{(\Omega,\Acal)}{(\R,\Bcal)}{\omega}{\lambda_{l}^{X(\omega)}}
is measurable, thus a real-valued random variable.
\end{lemma}
\begin{proof}
We know by Corollary~\ref{cor:eigencont} that
\map{\mathrm{Eig}_l}{\SMat{n}{\K}}{\R}{X}{\lambda_l^X}
is continuous, in particular measurable. Further, $X:\Omega\longrightarrow \SMat{n}{\K}$
is  measurable per definition, hence the composition $\lambda_{l}^{X}\defeq \mathrm{Eig}_l\circ X$ is measurable as well. 
\end{proof}

Lemma~\ref{lem:evmeasurable} allows us to study eigenvalues of random matrices in the context of probability theory. One aspect which gains a lot of attention is the behavior of the empirical distribution of the eigenvalues (see also Example~\ref{ex:empiricaldist}).

\begin{definition}
\label{def:esd}
Let $X$ be an $n\times n$ random matrix on $(\Omega,\Acal,\Prob)$, then the \emph{empirical spectral distribution (ESD)} $\sigma_{n}$ of $X$ is the random probability measure on $(\R,\Bcal)$ given by\label{sym:ESD}
\map{\sigma_{n}}{\Omega\times\Bcal}{\left[0,1\right]}{(\omega,B)}{\sigma_{n}(\omega,B)\defeq \frac{1}{n} \sum_{l=1}^{n}{\delta_{\lambda_{l}^{X(\omega)}}(B)}}
\end{definition}
It follows from our discussion in Example~\ref{ex:empiricaldist} that $\sigma_n$ really is a random probability measure.
How is $\sigma_n$ to be interpreted? For any interval $I\subseteq\R$, the random variable $\sigma_n(I)$ tells us the proportion of the $n$ eigenvalues that fall into the interval $I$. Thus, $\sigma_n$ carries information on the location of the eigenvalues, and it is of particular interest where the eigenvalues are located in the limit, that is, for $n\to\infty$. 

\subsubsection*{Wigner's semicircle law}

It is a famous theorem by Wigner that allows us to conclude under fairly weak assumptions  (mainly independence of matrix entries and uniformly bounded moments) that in the limit, eigenvalues will be spread according to the semicircle distribution:

\begin{definition}
\label{def:semicircle}
The semicircle distribution $\sigma$\label{sym:semicircdist} is the probability measure on $(\R,\Bcal)$ with Lebesgue-density $f_{\sigma}$ where
\map{f_{\sigma}}{\R}{\R}{x}{f_{\sigma}(x)\defeq \frac{1}{2\pi}\sqrt{4-x^2}\one_{\left[-2,2\right]}(x).}\label{sym:semicircdensity}
\end{definition}

Here and throughout this text, we will denote the Lebesgue measure on $(\R,\Bcal)$ by $\lebesgue$\label{sym:Lebesgue}. With respect to Definition~\ref{def:semicircle}, we have to prove that $f_{\sigma}\lebesgue$\label{sym:measurewithdensity} is actually a probability measure. We see immediately that the measure is finite, since $f_{\sigma}$ is bounded and has compact support. We will postpone the proof that the Lebesgue integral over $f_{\sigma}$ is $1$ to Lemma~\ref{lem:momentsSCD}. Since convergence to the semicircle distribution is an important and ubiquitous concept, we make the following definition.

\begin{definition}
If $(\sigma_n)_n$ are the ESDs of random matrices $(X_n)_n$ and $\sigma_n\to\sigma$ weakly in expectation resp.\ in probability resp.\ almost surely, then we say that \emph{the semicircle law holds for $(X_n)_n$} in expectation resp.\ in probability resp.\ almost surely.
\end{definition}

We now turn to Wigner's semicircle law. Notationally, for all $n\in\N$ we define the index set $\oneto{n}^2\defeq \oneto{n}\times\oneto{n}=\{1,2,\ldots,n\}\times\{1,2,\ldots,n\}$.\label{sym:numbersquare}

\begin{definition}
\label{def:Wignerscheme}
Let for all $n\in\N$, $X_n=(X_n(i,j))_{(i,j)\in\oneto{n}^2}$ be a family of real-valued random variables, then the sequence $(X_n)_n$ is called \emph{Wigner scheme}, if the following holds:
\begin{enumerate}[i)]
\item All random variables have uniformly bounded absolute moments, that is: For all $q\in\N$ there exists a constant $L_q\in(0,\infty)$ such that for all $n\in\N$ and all $(i,j)\in\oneto{n}^2$: $\E\abs{X_n(i,j)}^q\leq L_q$.
\item All random variables are standardized, that is: For all $n\in\N$ and all $(i,j)\in\oneto{n}^2$: $\E X_n(i,j) = 0$ and $\V X_n(i,j)=1$.
\item The families $X_n$ are symmetric, that is: For all $n\in\N$ and $(i,j)\in\oneto{n}^2$ we have $X_n(i,j)=X_n(j,i)$.
\item For all $n\in\N$ the family $(X_n(i,j))_{1\leq i \leq j \leq n}$ is independent.
\end{enumerate}
\end{definition}

Note in particular that in Definition~\ref{def:Wignerscheme} we \emph{do not} require that the whole family $((X_n(i,j))_{1\leq i \leq j \leq n})_{n\in\N}$ be independent. A very simple Wigner scheme is given in the following example:

\begin{example}
\label{ex:Wignerscheme}
Let $(X(i,j))_{1\leq i\leq j}$ be an i.i.d.\ family of real-valued random variables such that $\E\abs{X(1,1)}^q<\infty$ for all $q\in\N$, $\E X(1,1) = 0$ and $\V X(1,1) = 1$. Further, set $X(i,j)\defeq X(j,i)$ for all $1\leq j < i$. Now set for all $n\in\N$ and all $(i,j)\in\oneto{n}^2$: $X_n(i,j)\defeq X(i,j)$. Roughly speaking, $X_n$ is the $n\times n$ submatrix of the infinite matrix $X$. Then clearly, $(X_n)_n$ is a Wigner scheme as in Definition~\ref{def:Wignerscheme}. 
\end{example}

The following Theorem is called "Wigner's semicircle law."

\begin{theorem}
\label{thm:wigner}
Let $(X_n)_n$ be a Wigner scheme defined on a probability space $(\Omega,\Acal,\Prob)$. Define for all $n\in\N$ the Wigner matrix $W_n$\label{sym:randommatrix} by
\[
\forall\, (i,j)\in\oneto{n}^2:~ W_n(i,j)\defeq \frac{1}{\sqrt{n}}X_n(i,j).
\]
Then the semicircle law holds for $(W_n)_n$ almost surely.	
\end{theorem}
We will prove Theorem~\ref{thm:wigner} in various ways: In Section~\ref{sec:SemicircleByMoments} we will employ the method of moments to prove this theorem, whereas in Section~\ref{sec:SemicircleByStieltjes} we use the Stieltjes transform method.

\subsubsection*{The Marchenko-Pastur Law}

Another class of random matrix models besides the Wigner schemes fall into the category of covariance matrices.
Assume we have $n$ observations $x_1,\ldots,x_n$, each with $p$ real-valued covariates, where $n, p\in\N$, so that $x_i=(x_i(1),\ldots,x_i(p))^T$ for all $i\in\{1,\ldots,n\}$. Define the $p\times n$ data matrix $X_n\defeq (x_1|x_2|\ldots|x_n)$. The sample covariance matrix is then defined by
\[
\tilde{S}_n \defeq \frac{1}{n-1} \sum_{k=1}^n (x_k-\bar{x})(x_k-\bar{x})^T \quad\left(=\frac{n}{n-1}\cdot\left(\frac{1}{n}\sum_{k=1}^nx_kx_k^T -\bar{x}\bar{x}^T\right)\right),
\]
which is of dimension $p\times p$. Here, the vector $\bar{x}$ denotes the arithmetic mean of the vectors $x_k$. Assuming that the data stems from $n$ i.i.d. realizations of an $\R^p$-valued random vector $X$ with $\Lcal_2$-entries, $\tilde{S}_n$ is an unbiased estimator for its covariance matrix
\[
\E(X-\E X)(X-\E X)^T =
\begin{pmatrix}
\V X(1) & \cdots & \Cov(X(1),X(p))\\
\vdots &\ddots &\vdots\\
 \Cov(X(p),X(1)) & \cdots & \V X(p)
\end{pmatrix}.
\]
Many test statistics are based on the eigenvalues of the sample covariance matrix. When analyzing these eigenvalues in the limit, it suffices to consider 
\begin{equation}
\label{eq:Sn}	
S_n \defeq \frac{1}{n}\sum_{k=1}^n x_k x_k^T = \frac{1}{n}X_nX_n^T,
\end{equation}
since $\bar{x}\bar{x}^T$ is of rank $1$. Also, we will assume that the number of covariates $p$ grows with the number of observations $n$, so $p=p_n$. We further assume that $p/n\longrightarrow y\in(0,\infty)$, that is, $p$ grows proportionally with $n$. This leads to the definition of a Marchenko-Pastur scheme.

\begin{definition}
\label{def:MPscheme}
Let for all $n\in\N$, $p=p_n\in\N$ and $(X_n(i,j))_{i\in\oneto{p},j\in\oneto{n}}$ be a family of real-valued random variables. Then the sequence  $(X_n)_n$ is called \emph{Marchenko-Pastur scheme}, if the following holds:
\begin{enumerate}[i)]
\item All random variables have uniformly bounded absolute moments, that is: For all $q\in\N$ there exists a constant $L_q\in(0,\infty)$ such that for all $n\in\N$ and all $(i,j)\in\oneto{p}\times\oneto{n}$: $\E\abs{X_n(i,j)}^q\leq L_q$.
\item All random variables are standardized, that is: For all $n\in\N$ and all $(i,j)\in\oneto{p}\times\oneto{n}$: $\E X_n(i,j) = 0$ and $\V X_n(i,j)=1$.
\item For all $n\in\N$ the family $(X_n(i,j))_{i\in\oneto{p},j\in\oneto{n}}$ is independent.
\item There exists a constant $y\in(0,\infty)$ such that $p/n\to y$ as $n\to\infty$. 
\end{enumerate}
\end{definition}

We will see that eigenvalues of covariance matrices which are based on MP-schemes will spread according to the Marchenko-Pastur distribution:

\begin{definition}
The (standard) MP distribution with ratio index $y\in(0,\infty)$ is the probability measure $\mu^{y}$ on $(\R,\Bcal)$ given by
\[
\mu^{y} = \frac{1}{2\pi xy}\sqrt{((1+\sqrt{y})^2-x)(x-(1-\sqrt{y})^2)} \one_{(0,\infty)}(x)\lebesgue(\de x) + \left(1-\frac{1}{y}\right)\delta_0 \one_{y>1},
\]
where $\lebesgue$ denotes the Lebesgue measure on $(\R,\Bcal)$ and $\delta_0$ denotes the Dirac measure in $0$.
\end{definition}

\
\begin{definition}
Let $y\in(0,\infty)$. If $(\mu_n)_n$ are the ESDs of random matrices $(V_n)_n$ and $\mu_n\to\mu^y$ weakly in expectation resp.\ in probability resp.\ almost surely, then we say that \emph{the Marchenko-Pastur law holds for $(V_n)_n$} in expectation resp.\ in probability resp.\ almost surely.
\end{definition}

The following Theorem is called "Marchenko-Pastur law."

\begin{theorem}
\label{thm:MP}
Let $(X_n)_n$ be an MP-scheme defined on a probability space $(\Omega,\Acal,\Prob)$. Define for all $n\in\N$ the MP-matrix $V_n$ by
\[
V_n \defeq \frac{1}{n}X_n X_n^T.
\]
Then the MP-law holds for $(V_n)_n$ almost surely.	
\end{theorem}

We will prove Theorem~\ref{thm:MP} in various ways: In Section~\ref{sec:MPByMoments} we will employ the method of moments to prove this theorem, whereas in Section~\ref{sec:MPByStieltjes} we use the Stieltjes transform method.

\subsubsection*{Outlook}

Of course, a valid question is how to prove Theorem~\ref{thm:wigner} and Theorem~\ref{thm:MP}. We see that certain conditions are formulated for entries of these matrix models. In order to use these conditions in our analysis, how can we relate the ESDs $\sigma_n$ and $\mu_n$ back to the entries of their respective random matrices? And lastly, how can we conclude (stochastic) weak convergence of these ESDs? There are (at least) two standard ways to achieve this, namely the method of moments and the Stieltjes transform method. These methods will be discussed in depth in the following sections. We will also use these methods to prove the almost sure semicircle law and the Marchenko-Pastur law.

\chapter{The Method of Moments}
\label{chp:methodofmoments}

In Chapter~\ref{chp:weakconvergence} we have studied in depth the concepts of weak convergence of probability measures and random probability measures. In this chapter we want to discuss a tool which helps us to infer weak convergence: The method of moments. We will carefully develop this method for both deterministic and random probability measures. To be able to use this method correctly, we also need to delve into the moment problem.
But let us first define what the moments of a measure are:
\begin{definition}
Let $\mu$ be a measure on $(\R,\Bcal)$ and $k\in\N_0$\label{sym:naturalnumberszero}. If $\integrala{\mu}{\abs{x^k}}<\infty$ (where $x^0=1 \,\forall\, x\in\R$) we call the real number 
\[
m_k \defeq \integrala{\mu}{x^k}
\]
\emph{the $k$-th moment of $\mu$}. In this case, we say that $\mu$ has a finite $k$-th moment. On the other hand, if $\integrala{\mu}{\abs{x^k}}=\infty$, we say the $k$-th moment of $\mu$ does not exist.
\end{definition}
 
\section{The Moment Problem}

In numerous applications it is important to know the moments of a probability measure or at least some properties of the moments. In the Hamburger moment problem (see \autocite[145]{Reed}  and \autocite{Shohat}, for example), the question is reversed. Given a sequence of real numbers $(m_k)_{k\in\N_0}$, what can be said about the existence and uniqueness of a measure $\mu$ on $(\R,\Bcal)$ with moments $(m_k)_{k\in\N_0}$? To be more precise, does there exist a measure $\mu$ on $(\R,\Bcal)$ with moments $(m_k)_{k\in\N_0}$, and if so, is it the only measure with those moments? Of course, if such a measure exists, it is a probability measure iff $m_0=1$. It is rather surprising that the existence of such a measure can be nicely characterized:

\begin{theorem}
\label{thm:hamburg}
A sequence of real numbers $(m_k)_{k\in\N_0}$ constitutes the moments of at least one measure on $(\R,\Bcal)$, if and only if for all $N\in\N$ the corresponding Hankel matrix
\[
\begin{pmatrix}
m_0 & m_1 & m_2 & \hdots & m_N\\
m_1 & m_2 & m_3 & \hdots & m_{N+1}\\
m_2 & m_3 & m_4 & \hdots & m_{N+2}\\
\vdots & \vdots & \vdots & \ddots & \vdots\\
m_N & m_{N+1} & m_{N+2} & \hdots & m_{2N}
\end{pmatrix}
\]
is positive semi-definite, that is, if for all $N\in\N_0$ and all $\beta_0,\ldots,\beta_N\in\R$ it holds:
\[
\sum_{r,s = 0}^{N}{\beta_r \beta_s m_{r+s}} \geq 0.
\]
\end{theorem}
\begin{proof}
See \autocite[145]{Reed} in combination with the fact that a real symmetric matrix is positive definite in the real sense iff it is positive definite in the complex sense. 	
\end{proof}

Oftentimes it will not be of interest if a sequence of numbers $(m_k)_{k\in\N_0}$ really belongs to a probability measure, since we automatically obtain this result when employing the method of moments, see Theorem~\ref{thm:methodofmoments}. Theorem~\ref{thm:hamburg} still has two important applications: On the one hand, if the researcher is a priori assuming the target distribution to have specific moments,  Theorem~\ref{thm:hamburg} can be used to check whether this is a plausible assumption and can spare the researcher from trying to prove convergence to a non-existing probability measure. On the other hand, if one has already employed the method of moments and the moments of the target distribution have been calculated, one can a posteriori evaluate the plausibility of the calculations via Theorem~\ref{thm:hamburg}. Indeed, this is not uncommon practice, see  \autocite[15]{Dembo}, for example.
In any case, what will be essential for the method of moments is the knowledge about the uniqueness of a distribution with given moments,
that is, the answer to the question whether there is \emph{at most} one distribution with a given sequence of moments.

\begin{theorem}
\label{thm:uniquemeasure}
Let $(m_k)_{k\in\N}$ be a sequence of real numbers. If one of the following three conditions holds, there is at most one probability measure on $(\R,\Bcal)$\ with moments $(m_k)_{k\in\N}$:
\begin{enumerate}[i)]
	\item $\sum_{k=1}^{\infty}{\frac{1}{\sqrt[2k]{m_{2k}}}} = \infty \quad$ \hfill (Carleman condition),
	\item $\limsup_{k\to\infty}{\frac{\sqrt[2k]{m_{2k}}}{2k}} < \infty$,
	\item $\,\exists\, C,D\geq 1:\,\forall\, k\in\N: \abs{m_k}\leq C\cdot D^k\cdot k!$.
\end{enumerate}
Further, it holds that $iii)\Rightarrow ii)\Rightarrow i)$, that is, the Carleman condition is the weakest of the three.
\end{theorem}

\begin{proof}
\underline{$i)$:} See \autocite[85]{Akhiezer}.\newline
\noindent\underline{$ii)$:} See \autocite[123]{Durrett}.\newline
\noindent\underline{$iii)$:} See \autocite[205]{Reed}.\newline
\noindent\underline{Additional statement:} The additional statement also proves that $ii)$ and $iii)$ are sufficient when knowing that $i)$ is sufficient.

We assume that $ii)$ holds. Let for all $k\in\N: \alpha_k \defeq \sqrt[2k]{m_{2k}}\geq 0$, then we have to show $\sum_{k=1}^{\infty}{\frac{1}{\alpha_{k}}}=\infty$ under the condition that $r\defeq \limsup_{k\to\infty}{\frac{\alpha_k}{2k}} < \infty$. But there exists a $K\in\N$ such that for all $k\geq K$ we find $\frac{\alpha_k}{2k}\leq r+1$, thus $\alpha_k \leq 2k\cdot (r+1)$. Due to divergence of the harmonic series we obtain:
\[
\sum_{k=1}^{\infty}{\frac{1}{\alpha_k}}\geq \sum_{k\geq K}{\frac{1}{2k\cdot (r+1)}}=\infty.
\]
Therefore, $i)$ follows from $ii)$.
Now if $iii)$ holds, we find for all $k\in\N$:
\[
\frac{\sqrt[2k]{m_{2k}}}{2k} \leq \frac{\sqrt[2k]{C\cdot D^{2k}\cdot (2k)!}}{2k} \leq C\cdot D \cdot \frac{\sqrt[2k]{(2k)!}}{2k} \leq C\cdot D,
\]
since $(2k)^{2k}\geq (2k)!$ yields ${2k}\geq\sqrt[2k]{(2k)!}$ for all $k\in\N$. Thus, $ii)$ holds.
\end{proof}

In the next corollary we will see that the moments of probability measures with compact support possess moments of all orders, and that they are uniquely determined by their moments.
\begin{corollary}
\label{cor:compactuniquemoments}
Let $\nu$ be a probability measure on $(\R,\Bcal)$ with compact support which lies in $[-a,a]$ for some $a\in\N$. Then
\begin{enumerate}[i)]
\item $\nu$ has moments of all orders.
\item For all $k\in\N$: $\abs{\integrala{\nu}{x^k}}\leq a^k$.
\item $\nu$ is uniquely determined by its moments.
\end{enumerate}
\end{corollary}
\begin{proof}
We calculate for $k\in\N$ arbitrary:
\[
\bigabs{\integrala{\nu}{x^k}}=\integrala{\nu}{\abs{x}^k} = \integrala{\nu}{\one_{[-a,a]}\abs{x}^k}\leq a^k.
\]
This shows $i)$ and $ii)$, and $iii)$ follows immediately with Theorem~\ref{thm:uniquemeasure} $iii)$.
\end{proof}

\section{The Method of Moments for Probability Measures}

Now we are well-prepared to introduce the method of moments, which is a means to infer weak convergence of a sequence of distributions from the convergence of their moments.

\begin{theorem}
\label{thm:methodofmoments}
Let $(\mu_n)_{n\in\N}$ be a sequence in $\Mcal_1(\R)$, so that all moments of every $\mu_n$ exist. If there exists a sequence of real numbers $(m_k)_{k\in\N}$, so that 
\begin{equation}
\label{eq:moment}
\forall\, k\in \N: \lim\limits_{n \rightarrow \infty}{\integrala{\mu_n}{x^k}} = m_k,
\end{equation}
the following statements hold: 

There exists a $\mu \in \Mcal_1(\R)$ and a subsequence of $(\mu_n)_{n\in\N}$, which converges weakly to $\mu$. Then $\forall\, k\in\N: m_k = \integrala{\mu}{x^k}$. In particular, the $(m_k)_{k\in\N}$ are moments of a probability measure on $(\R,\Bcal)$. Further: If $\mu$ is uniquely determined by its moments, then the entire sequence $(\mu_n)_n$ converges weakly to $\mu$.
\end{theorem}
\begin{proof}
With \eqref{eq:moment} it follows with $k=2$ and Lemma~\ref{lem:tight} that $(\mu_n)_{n\in\N}$ is tight. Therefore, with Lemma~\ref{lem:convergentsubsequence} there exists a $\mu\in\Mcal_1(\R)$ and a subsequence $J\subseteq\N$ such that $(\mu_n)_{n\in J}$ converges weakly to $\mu$. With Lemma~\ref{lem:integrationtolimit}, we then obtain for all $k\in\N $ that $(\integrala{\mu_n}{x^k})_{n\in J}$ converges to $\integrala{\mu}{x^k}$, since the sequence $(\integrala{\mu_n}{1+x^{2k}})_{n\in J}$ is bounded and the function $x\mapsto \frac{x^k}{1+x^{2k}}$ vanishes at infinity. We conclude with \eqref{eq:moment} that for all $k\in\N$ we have $\integrala{\mu}{x^{k}} = m_k$, so $(m_k)_k$ are indeed moments of a probability measure.

Now, if $\mu$ is uniquely determined by its moments, then the entire sequence $(\mu_n)_{n\in\N}$ -- and not just a subsequence -- converges weakly to $\mu$. To see this, let $(\mu_n)_{n\in I}$ be an arbitrary subsequence. By Lemma~\ref{lem:subsequential}, it suffices to show that this subsequence has another subsequence that converges weakly to $\mu$. But as above (with swapped roles of $I$ and $\N$) we find a probability measure $\nu$ on $(\R,\Bcal)$ and a subsequence $J'\subseteq I$, such that that $(\mu_n)_{n\in J'}$ converges weakly to $\nu$ and the numbers $(m_k)_{k\in\N}$ are the moments of $\nu$. Since $\mu$ is uniquely determined by these moments, we must have $\mu = \nu$.
\end{proof}
\begin{remark}
\label{rem:pitfall}
A converse statement of Theorem~\ref{thm:methodofmoments} is not true in general, that is, there are probability measures $(\mu_n)_n$ and $\mu$ with 
\begin{enumerate}
\item All moments of $\mu$ and of all $\mu_n$ exist.
\item $\mu_n$ converges weakly to $\mu$.
\item The moments of $\mu_n$ do not converge to the moments of $\mu$.
\end{enumerate}
The construction is rather simple: Pick $\mu\defeq \delta_0$ and 
\[
\forall\, n\in\N:\quad \mu_n\defeq \frac{n-1}{n}\delta_0 + \frac{1}{n}\delta_{e^n}
\]
Then surely, conditions 1.\ and 2.\ are satisfied,  but 3.\ as well, since for all $k\in\N$:
\[
\integrala{\mu_n}{x^k} = \frac{1}{n}e^{kn}\to\infty \neq 0 = \integrala{\mu}{x^k}.
\] 	
\end{remark}

\section{The Method of Moments for Random Probability Measures}

The next theorem will generalize the method of moments to the convergence types of random probability measures, namely to weak convergence in expectation, in probability and almost surely. Although this could be presented in greater generality, we will restrict our attention to convergence of random probability measures to a \emph{deterministic} probability measure. This is the type of convergence we will encounter in our analyses ahead.

\begin{theorem}
\label{thm:randommethodofmoments}
Let $(\mu_n)_{n\in\N}$ be random probability measures on $(\R,\Bcal)$ and $\mu$ be a deterministic probability measure on $(\R,\Bcal)$ which is uniquely determined by its moments. Then assuming that all following expressions (random moments, expected random moments) are well-defined and finite, we conclude:
\begin{enumerate}[i)]
	\item If $\,\forall\, k\in\N: \E\integrala{\mu_n}{x^k}\xrightarrow[n\to\infty]{}\integrala{\mu}{x^k}$, then $\mu_n \xrightarrow[n\to\infty]{} \mu$ weakly in expectation.
	\item If $\,\forall\, k\in\N: \integrala{\mu_n}{x^k}\xrightarrow[n\to\infty]{}\integrala{\mu}{x^k}$ in probability, then $\mu_n \xrightarrow[n\to\infty]{} \mu$ weakly in probability.
	\item If $\,\forall\, k\in\N: \left[\integrala{\mu_n}{x^k}\xrightarrow[n\to\infty]{}\integrala{\mu}{x^k}~\Prob\textrm{-a.s.}\right]$, then $\mu_n \xrightarrow[n\to\infty]{} \mu$ weakly almost surely.
\end{enumerate}
\end{theorem}

\begin{proof}
\underline{i)} With Theorem~\ref{thm:methodofmoments} it suffices to show that for all $k\in\N$,  $\integrala{\E\mu_n}{x^k} \to \integrala{\mu}{x^k}$ as $n\to\infty$. Therefore, all we must argue is that for all $k\in\N$, $\integrala{\E\mu_n}{x^k}=\E\integrala{\mu_n}{x^k}$. But for $k\in\N$ arbitrary we find
\[
\integrala{\E\mu_n}{\abs{x^k}}^2\leq \integrala{\E\mu_n}{x^{2k}} = \E\integrala{\mu_n}{x^{2k}} < \infty,
\]
where we used Theorem~\ref{thm:expectedmeasure} ii), the fact that $x\mapsto x^{2k}$ is non-negative, and the assumption in the statement of the theorem that all expected moments exist. Therefore, $\E\mu_n$ has existing moments of all orders, so with Theorem~\ref{thm:expectedmeasure} iii) we obtain $\integrala{\E\mu_n}{x^k}=\E\integrala{\mu_n}{x^k}$.
\newline
\underline{ii)} We want to show that $\mu_n\to\mu$ weakly in probability, which means that for all $f\in\Ccal_b(\R)$, $\integrala{\mu_n}{f}$ converges to $\integrala{\mu}{f}$ in probability. To this end, let $f\in\Ccal_b(\R)$ be arbitrary. To show that $(\integrala{\mu_n}{f})_{n\in\N}$ converges to $\integrala{\mu}{f}$ in probability we will show that any subsequence has an almost surely convergent subsequence: Let $J\subseteq\N$ be a subsequence. Applying Lemma~\ref{lem:uniformsubsequence} we find a subsequence $I\subseteq J$ and a measurable set $\Omega_1\subseteq\Omega$ of measure $1$, such that
\[
\forall\,\omega\in\Omega_1:\,\forall\, k\in\N: \integrala{\mu_n(\omega)}{x^k}\xrightarrow[n\in I]{}\integrala{\mu}{x^k}.
\]
In particular, with Theorem~\ref{thm:methodofmoments} we find that for all $\omega\in\Omega_1$, $\mu_n(\omega)$ converges weakly to $\mu$ for $n\in I$, so that in particular, $\integrala{\mu_n(\omega)}{f}\to\integrala{\mu}{f}$ for $n\in I$. Therefore, $\integrala{\mu_n}{f}\to\integrala{\mu}{f}$ almost surely for $n\in I$. \newline
\underline{iii)} For all $k\in\N$ we find a measurable set $\Omega_k\subseteq\Omega$ with measure $1$ such that for all $\omega\in\Omega_k: \integrala{\mu_n(\omega)}{x^k}\to\integrala{\mu}{x^k}$ as $n\to\infty$. Then $\Omega'\defeq\cap_{k\in\N}\Omega_k$ has measure $1$ and for all $\omega\in\Omega'$ we find that $\integrala{\mu_n(\omega)}{x^k}\to\integrala{\mu}{x^k}$ for all $k\in\N$, so that with Theorem~\ref{thm:methodofmoments}, for all $\omega\in\Omega'$ we have that $\mu_n(\omega)$ converges weakly to $\mu$. Therefore, $\mu_n$ converges weakly to $\mu$ almost surely.
\end{proof}

We refer the reader to Remark~\ref{rem:stochasticweakconvergence} for an explanation on the use of brackets $[\ldots]$ in Theorem~\ref{thm:randommethodofmoments} $iii)$.

\begin{remark}
\label{bem:mighty}
The method of moments for random probability measures (Theorem~\ref{thm:randommethodofmoments}) works as follows: To show weak convergence of random probability measures in expectation, in probability or almost surely, it suffices to show that the random moments converge in expectation, in probability or almost surely. This is a very useful theorem, in particular considering we do not make any assumptions on the target measure $\mu$ except those mentioned in Theorem~\ref{thm:randommethodofmoments}. In particular, we do not require the target probability measure to have compact support. In the literature on random matrices, this condition is often used to justify the method of moments, see \autocite[11]{Anderson}, for example.
\end{remark}

The next theorem will help us determine when the conditions for Theorem~\ref{thm:randommethodofmoments} are met, to be more precise, when we are able to confirm convergence of the moments in probability or almost surely. Further, it does not assume a priori the knowledge of the target measure $\mu\in\Mcal_1(\R)$. In summary, this is the theorem that is used when applying the method of moments to random matrix theory, see also Theorems~\ref{thm:explorative} and~\ref{thm:expectationandvariance}.

\begin{theorem}
\label{thm:randommomentconvergence}
Let $(\mu_n)_{n\in\N}$ be random probability measures on $(\R,\Bcal)$ and $(m_k)_{k\in\N}$ be a sequence of real numbers, so that there is at most one probability measure on $(\R,\Bcal)$ with moments $(m_k)_{k\in\N}$. We formulate the following conditions, where we assume that all expressions (random moments, expectations and variances) are finite:
\emph{
\begin{enumerate}[(M1)]
	\item For all $k\in\N$,
	\[
	\E\integrala{\mu_n}{x^k} \xrightarrow[n\to\infty]{} m_k.
	\]
\end{enumerate}
}
For the following assumptions we assume that for all $k\in\N$ we can find a finite decomposition
\[
\integrala{\mu_n}{x^k} = D^{(k,1)}_n + \ldots + D^{(k,\ell_k)}_n,
\]
such that for all $k\in\N$ and all $i\in\oneto{\ell_k}$, $\E D^{(k,i)}_n$ converges to a constant as $n\to\infty$.
\emph{
\begin{enumerate}
	\item[(M2)] For all $k\in\N$ and $i\in\oneto{\ell_k}$, 
	\[
	\exists\, z\in\N:\ \E\bigabs{D^{(k,i)}_n-\E D^{(k,i)}_n}^z\xrightarrow[n\to\infty]{} 0,
	\]
	\item[(M3)] For all $k\in\N$ and $i\in\oneto{\ell_k}$,
	\[
	\exists\, z\in\N:\ \E\bigabs{D^{(k,i)}_n -\E D^{(k,i)}_n}^z\xrightarrow[n\to\infty]{} 0 \quad \text{summably fast.}
	\]
\end{enumerate}
}

Then we conclude:
\begin{enumerate}[i)]
	\item[i)] If \emph{(M1)} holds, then there is a $\mu\in\Mcal_1(\R)$ with moments $(m_k)_{k\in\N}$, so that $\E \mu_n\rightarrow \mu$ weakly (that is, $\mu_n\to\mu$ weakly in expectation). In particular, the numbers $(m_k)_{k\in\N}$ are the moments of a probability measure.
	\item[ii)] If \emph{(M1)} and \emph{(M2)} hold, we conclude
	\[
	\,\forall\, k\in\N: \integrala{\mu_n}{x^k}\xrightarrow[n\to\infty]{} \integrala{\mu}{x^k}\textrm{ in probability}
	\]
	and thus $\mu_n\rightarrow\mu$ weakly in probability via Theorem~\ref{thm:randommethodofmoments}.
	\item[iii)] If \emph{(M1)} and \emph{(M3)} hold, we conclude
	\[
	\forall\, k\in\N: \left[\integrala{\mu_n}{x^k} \xrightarrow[n\to\infty]{} \integrala{\mu}{x^k}~\Prob\textrm{-a.s.}\right]
	\]
	and thus $\mu_n\rightarrow\mu$ weakly almost surely via Theorem~\ref{thm:randommethodofmoments}.
\end{enumerate}
\end{theorem}
\begin{proof}
\underline{i)}
As we saw in the proof of Theorem~\ref{thm:randommethodofmoments}, we find that for all $n\in\N$, the expected measure $\E\mu_n$ has moments of all orders and that for all $k\in\N: \integrala{\E\mu_n}{x^k}=\E\integrala{\mu_n}{x^k}$. Now given (M1), statement $i)$ follows directly with Theorem~\ref{thm:methodofmoments}.\newline
\underline{ii)/iii)} If (M1) holds, then (M2) (resp.\ (M3)) together with Lemma~\ref{lem:convergencemodes} shows that for all $k\in\N$ and $i\in\oneto{\ell_k}$, $D^{(k,i)}_n$ converges to a constant in probability (resp.\ almost surely) as $n\to\infty$, so that by (M1),
\[
\integrala{\mu_n}{x^k} = D^{(k,1)}_n +\ldots + D^{(k,\ell_k)}_{n} \xrightarrow[n\to\infty]{} m_k
\]
in probability (resp. almost surely).
\end{proof}

\begin{lemma}
\label{lem:convergencemodes}
Let $z\in\N$ and $(Y_n)_n$ be random variables with $\E\abs{Y_n}^z < \infty$ for all $n\in\N$. If $\E Y_n \to y$ and $\E\abs{Y_n-\E Y_n}^z \to 0$, then $Y_n\to y$ in probability. If in addition, $\E\abs{Y_n-\E Y_n}^z$ is summable, then $Y_n\to y$ almost surely.
\end{lemma}
\begin{proof}
Using Markov's inequality, we calculate for $\varepsilon>0$ arbitrary:
\begin{align*}
\Prob(\abs{Y_n-y}>\varepsilon)&\ \leq\ \Prob\left(\abs{Y_n -\E Y_n}>\frac{\varepsilon}{2}\right)\ +\ \Prob\left(\abs{\E Y_n - y}>\frac{\varepsilon}{2}\right)\\ 
&\leq\ \frac{2^z}{\varepsilon^z}\E\abs{Y_n-\E Y_n}^z\ +\ \Prob\left(\abs{\E Y_n - y}>\frac{\varepsilon}{2}\right).
\end{align*}
The statement follows (also using Borel-Cantelli), since the very last summand vanishes for all $n$ large enough.
\end{proof}

\section{The Moments of the Semicircle Distribution}

In random matrix theory, the probability measure that appears as the limit of the empirical spectral distribution is typically the semicircle distribution as defined in Definition~\ref{def:semicircle}. What we mean by \emph{typically} is that it appears in Wigner's semicircle law, Theorem~\ref{thm:wigner}, which is the simplest non-trivial random matrix ensemble, for it has standardized entries which are independent up to the symmetry constraint. It is safe to say that the role of the semicircle distribution in random matrix theory resembles the role of the standard normal distribution in probability theory. To remind the reader, the semicircle distribution $\sigma$ is the probability measure on $(\R,\Bcal)$ with Lebesgue-density $f_{\sigma}$ where
\map{f_{\sigma}}{\R}{\R}{x}{f_{\sigma}(x)\defeq \frac{1}{2\pi}\sqrt{4-x^2}\one_{\left[-2,2\right]}(x).}

Since we would like to apply the method of moments to random matrix theory, we will proceed to derive the moments of the semicircle distribution. As it turns out, we will obtain that $\integrala{\sigma}{x^0}=1$, so that $\sigma$ is identified as a probability measure, which we still owed to the reader.

\begin{lemma}
\label{lem:momentsSCD}
The moments of the semicircle distribution $\sigma$ are given by\label{sym:semicircmoment}
\begin{equation}
\label{eq:semicirclemoments}
\text{For all } k\in\N_0:\ m^{\sigma}_{2k}= 
\frac{(2k)!}{k!(k+1)!} \quad \text{and}\quad  m^{\sigma}_{2k+1}= 0
\end{equation}
\end{lemma}

\begin{proof}
We follow the short proof in \autocite[16]{BaiSi}.
To this end, note that the integrand is compactly supported and bounded. Further, for odd moments the integrand is odd, so the statement follows for odd moments. For even moments, we obtain the statement by the following calculation:

\begin{align*}
m^{\sigma}_{2k}\  
&=\ \frac{1}{2\pi} \int_{-2}^{2} x^{2k} \sqrt{4-x^2}\de x
\ =\ \frac{1}{\pi} \int_{0}^{2} x^{2k} \sqrt{4-x^2}\de x\\
&=\ \frac{2^{2k+1}}{\pi}\int_0^1 y^{k-1/2}(1-y)^{1/2} \de y\ =\ \frac{2^{2k+1}}{\pi} B(k+1/2,\,3/2)\\
&=\ \frac{2^{2k+1}}{\pi}\frac{\Gamma(k+1/2)\Gamma(3/2)}{\Gamma(k+2)} = \frac{1}{k+1}\binom{2k}{k},
\end{align*}
where in the second step, we used that the integrand is even, in the third step we substituted $x$ by $2\sqrt{y}$, in the fourth step we used the definition of the beta function $B$, in the fifth step, we used that for all $x,y>0$: $B(x,y)=\Gamma(x)\Gamma(y)/\Gamma(x+y)$, where $\Gamma$ is the gamma function, and in the last step we used that for all $n\in\N$: $\Gamma(n)=(n-1)!$, and for all $n\in\N_0$: $\Gamma(n+1/2)= (2n)!\sqrt{\pi}/(n!4^n)$.
\end{proof}

To use the method of moments to prove weak convergence to the semicircle distribution, we need the following corollary:
\begin{corollary}
\label{cor:semicirlceuniquemoments}
The semicircle distribution $\sigma$ is uniquely determined by its moments.	
\end{corollary}
\begin{proof}
Since the support of $\sigma$ is compact, the statement follows with Lemma~\ref{cor:compactuniquemoments}.	
\end{proof}

The values of the even moments of the semicircle distribution bear a special name:

\begin{definition}
The \emph{Catalan numbers} are elements of the sequence of natural numbers $(\Cat_k)_{k\in\N_0}$ \label{sym:catalan}, where
\[
\forall\, k\in\N_0:\, \Cat_k \defeq \frac{(2k)!}{k!(k+1)!}.
\]
\end{definition}

Combining the results of Lemma~\ref{lem:momentsSCD} with the definition of the Catalan numbers, we obtain for the sequence $(m^{\sigma}_k)_{k\in\N_0}$ of the moments of the semicircle distribution:
\begin{equation}
\label{eq:semicirclecatalan}
m^{\sigma}_k=
\begin{cases}
\Cat_{k/2} & \text{for $k$ even},\\
0 			& \text{for $k$ odd}.
\end{cases}
\end{equation}
But the Catalan numbers are not only the (even) moments of the semicircle distribution. They also appear as the solution to various combinatorial problems, see  \autocite{Koshy} or \autocite{Stanley}, for example. 

\section{The Moments of the Marchenko-Pastur distribution}

For sample covariance matrices, the canonical limit is not Wigner's semicircle distribution, but the Marchenko-Pastur distribution $\mu^y$ with ratio index $y\in(0,\infty)$. As a reminder to the reader, $\mu^{y}$ is the sum of the point mass $(1-y^{-1})\one_{y>1}$ in zero and a Lebesgue-continuous part given by the density $f_{\mu}$ (where the parameter $y$ is suppressed) as \map{f_{\mu}}{\R}{\R}{x}{f_{\mu}(x)\defeq\frac{1}{2\pi xy}\sqrt{(y_+-x)(x-y_-)} \one_{(y_-,y_+)}(x),}
where $y_+\defeq (1+\sqrt{y})^2$ and $y_-\defeq(1-\sqrt{y})^2$.
In order to apply the method of moments to prove the Marchenko-Pastur law, we need to know the moments of $\mu^y$, which is the content of the following lemma:
\begin{lemma}
\label{lem:momentsMPD}
For all $y\in(0,\infty)$ and $k\in\N$, it holds
\[
\integrala{\mu^y}{x^k} = \sum_{r=0}^{k-1}\frac{y^r}{r+1}\binom{k}{r}\binom{k-1}{r}.
\]
\end{lemma}
\begin{proof}
The proof is rather lengthy and can be found in \autocite[40]{BaiSi}.
\end{proof}

\begin{corollary}
\label{cor:MPuniquemoments}
For every $y>0$, the Marchenko-Pastur distribution $\mu^y$ is uniquely determined by its moments.	
\end{corollary}
\begin{proof}
Since the support of $\mu^y$ is compact, the statement follows with Lemma~\ref{cor:compactuniquemoments}.	
\end{proof}

\section{Application of the Method of Moments to RMT}
\label{sec:momentmethinrmt}

So far, we have pointed out what the method of moments is and how it works in deterministic and stochastic settings. Now we want to build the bridge to random matrix theory. To this end, we need the following observation, where as before, $\K\in\{\R,\C\}$:

\begin{lemma}
\label{lem:eigentrace}
Let $n\in\N$ and $X\in\SMat{n}{\K}$, then we obtain for all $k\in\N$:
\[
\sum_{i=1}^{n}(\lambda_{i}^{X})^k=\tr X^k=
\sum_{t_1,\ldots,t_k=1}^n X(t_1,t_2)X(t_2,t_3)\cdots X(t_k,t_1).
\]
\end{lemma}
\begin{proof}
The second equality is clear. For the first equality, note that since $X\in\SMat{n}{\K}$, by Lemma~\ref{lem:symmetricspectrum}, there exists an invertible matrix $S\in\Mat{n}{\K}$ so that $X = S^{-1}DS$, where $D = \diag(\lambda^{X}_1,\ldots,\lambda^{X}_n)$. Then 
\[
X^k = \underbrace{S^{-1}DS\cdot S^{-1}DS \cdot \ldots \cdot S^{-1}DS}_{\text{$k$ factors}} = S^{-1} D^k S = S^{-1} \diag\left((\lambda^{X}_1)^k,\ldots,(\lambda^{X}_n)^k\right) S.
\]
With Lemma~\ref{lem:trace}, we obtain
\[
\tr(X^k) = \tr\diag\left((\lambda^{X}_1)^k,\ldots,(\lambda^{X}_n)^k\right) = \sum_{i=1}^{n}(\lambda_{i}^{X})^k.
\]
\end{proof}

\begin{corollary}
\label{cor:momenttosum}
Let $(X_n)_n$ be a sequence of random matrices with corresponding ESDs $(\sigma_n)_n$. Then for all $k\in\N$ we find
\begin{equation}
\label{eq:momenttosum}
\integrala{\sigma_n}{x^k}= \frac{1}{n}\tr X^k_n=\frac{1}{n}\sum_{t_1,\ldots,t_k=1}^{n}{X_n(t_1,t_2)X_n(t_2,t_3)\cdots X_n(t_k,t_1)}.
\end{equation}
\end{corollary}
\begin{proof}
Using Lemma~\ref{lem:eigentrace}, we calculate:
\[
\integrala{\sigma_n}{x^k} = \frac{1}{n}\sum_{i=1}^{n}(\lambda_{i}^{X_n})^k =  \frac{1}{n} \tr X_n^k = \frac{1}{n} \sum_{t_1,\ldots,t_k=1}^n X_n(t_1,t_2)X_n(t_2,t_3)\cdots X_n(t_k,t_1).
\]
\end{proof}

The next theorem will be of use in explorative settings where the target distribution is not known or assumed yet. This is the very first step in showing that the ESDs of random matrices converge to a probability measure. To clarify terminology that we use, if $Y$ is a $\K$-valued random variable, where $\K\in\{\R,\C\}$, and if $p\in\N_0$, then we call $\E\abs{Y}^p$ the $p$-th absolute moment of $Y$. Further, we say that $Y$ has absolute moments of all orders, if $\E\abs{Y}^p<\infty$ for all $p\in\N_0$. Note that $Y$ is integrable iff its first absolute moment exists.

\begin{theorem}
\label{thm:explorative}
Let $(\sigma_n)_n$ be the empirical spectral distributions of random matrices $(X_n)_n$, whose ($\K$-valued) entries have absolute moments of all orders. Then if 
\[
\forall\,k\in\N: \E\integrala{\sigma_n}{x^k}\xrightarrow[n\to\infty]{} m_k,
\]
where $(m_k)_k$ is a sequence of real numbers that satisfy the Carleman condition (cf. Theorem~\ref{thm:uniquemeasure}), then $(\sigma_n)_n$ converges weakly in expectation to a probability measure $\mu$ on $(\R,\Bcal)$ with moments $(m_k)_k$.
\end{theorem}

\begin{proof}
This follows with Theorem~\ref{thm:randommomentconvergence}, since by Corollary~\ref{cor:momenttosum}, for each $k\in\N_0$, the $k$-th random moment is given by
\[
\integrala{\sigma_n}{x^k}= \frac{1}{n}\sum_{t_1,\ldots,t_k=1}^{n}{X_n(t_1,t_2)X_n(t_2,t_3)\cdots X_n(t_k,t_1)},
\]
which is a real-valued random variable whose expectation is finite, see the following Lemma~\ref{lem:holder}.
\end{proof}

\begin{lemma}
\label{lem:holder}
Let $Y_1,\ldots,Y_k$ be $\K$-valued random variables such that  $\E\abs{Y_i}^k < \infty$ for all $i\in\{1,\ldots,k\}$, then
\[
\E\abs{Y_1 Y_2\cdots Y_k} \ \leq\ \left(\E\abs{Y_1}^k\right)^{\frac{1}{k}}\cdots \left(\E\abs{Y_k}^k\right)^{\frac{1}{k}}\ \leq\ \max_{i=1,\ldots,k} \E\abs{Y_i}^k
\]  
\end{lemma}
\begin{proof}
The second inequality is clear, so we only need to show the first one, which can be regarded as a generalization of the Cauchy-Schwarz inequality. We proceed by induction. The cases $k=1$ and $k=2$ are already known. By Hölder's inequality,
\[
E\abs{Y_1\cdots Y_k} \leq \left(\E\abs{Y_1\cdots Y_{k-1}}^{\frac{k}{k-1}}\right)^{\frac{k-1}{k}} \left(\E\abs{Y_k}^k\right)^{\frac{1}{k}}.
\]
Using the induction hypothesis, we calculate
\[
\E \abs{Y_1}^{\frac{k}{k-1}}\cdots \abs{Y_{k-1}}^{\frac{k}{k-1}} \leq \left(\E\abs{Y_1}^k\right)^{\frac{1}{k-1}}\cdots\left(\E\abs{Y_{k-1}}^k\right)^{\frac{1}{k-1}},
\]
from which the statement follows.
\end{proof}

We remind the reader that convergence in expectation is a necessity for stronger convergence types, see Theorem~\ref{thm:randomconvergenceimplications}. Therefore, Theorem~\ref{thm:explorative} is really the basis for any explorative analysis. The next theorem will be of use either after Theorem~\ref{thm:explorative} has been applied or if a priori, one has the target distribution of the ESDs in mind, for example if one wants to show a semicircle law.

\begin{theorem}\label{thm:expectationandvariance}
Let $(\sigma_n)_n$ be the empirical spectral distributions of Hermitian random matrices $(X_n)_n$, whose entries have absolute moments of all orders. Denote by $\mu$ a probability measure which is uniquely determined by its moments (cf. Theorem~\ref{thm:uniquemeasure}). Then \emph{
\begin{enumerate}[i)]
	\item $\sigma_n$ converges to $\mu$ weakly in expectation, if for all $k\in\N$,
	\[
	\E\integrala{\sigma_n}{x^k} \xrightarrow[n\to\infty]{} m_k.
	\]
\end{enumerate}
}
We assume that for all $k\in\N$ we find a finite decomposition
\[
\integrala{\sigma_n}{x^k} = D^{(k,1)}_n +\ldots + D^{(k,\ell_k)}_n
\]
such that for all $k\in\N$ and all $i\in\oneto{\ell_k}$, $\E D^{(k,i)}_n$ converges to a constant as $n\to\infty$. (This decomposition will become clear from the analysis, for example when showing that i) holds.) Then
\emph{
\begin{enumerate}
	\item[ii)] $\sigma_n$ converges to $\mu$ weakly in probability, if $i)$ holds and for all $k\in\N$ and $i\in\oneto{\ell_k}$, 
	\[
	\exists\, z\in\N:\ \E\bigabs{D^{(k,i)}_n-\E D^{(k,i)}_n}^z\xrightarrow[n\to\infty]{} 0,
	\]
	\item[iii)] $\sigma_n$ converges to $\mu$ weakly almost surely, if $i)$ holds and for all $k\in\N$ and $i\in\oneto{\ell_k}$,
	\[
	\exists\, z\in\N:\ \E\bigabs{D^{(k,i)}_n-\E D^{(k,i)}_n}^z\xrightarrow[n\to\infty]{} 0 \quad \text{summably fast.}
	\]
\end{enumerate}
}

\end{theorem}

\begin{proof}
This is a direct consequence of Theorem~\ref{thm:randommomentconvergence}, considering that since matrix entries have moments of all orders, Corollary~\ref{cor:momenttosum} and Lemma~\ref{lem:holder} imply that expected random moments and all other expectations are well-defined and finite.	
\end{proof}

Next, as an application, let us discuss the proof strategy behind Wigner's semicircle law, Theorem~\ref{thm:wigner}, where we restrict our attention to convergence in probability:

\begin{example}
Consider the setup of Theorem~\ref{thm:wigner}.
Let $(m^{\sigma}_k)_{k\in\N}$ denote the moments of the semicircle distribution, then we can use Theorem~\ref{thm:expectationandvariance} and show that 
\begin{enumerate}[1.]
	\item For all $k\in\N$:
	\[
	\E\integrala{\sigma_n}{x^k}=\frac{1}{n^{1+k/2}}\sum_{t_1,\ldots,t_k=1}^{n}{\E a(t_1,t_2)a(t_2,t_3)\cdots a(t_k,t_1)} \xrightarrow[n\to\infty]{} m^{\sigma}_k.
	\]
	\item For all $k\in\N$:
	\[
	\E\left(\integrala{\sigma_n}{x^k}^2\right) \xrightarrow[n\to\infty]{} (m^{\sigma}_k)^2.
	\]
\end{enumerate}
This will imply statements i) and ii) from the preceding theorem with $z=2$, thus the semicircle law in probability.

	This is also exactly what is shown in \autocite{Anderson}, as can be seen from their Lemma 2.1.6 in combination with the proof of their Lemma 2.1.7. However, although Theorem~\ref{thm:expectationandvariance} yields that above points 1.\ and 2.\ suffice for weak convergence in probability, in \autocite{Anderson} further cumbersome calculations are carried out, utilizing the compactness of the support of the semicircle distribution, which can be observed on their pages 10 and 11.
\end{example}

\chapter{The Semicircle and MP Laws by the Moment Method}
\label{chp:MomentsSCLMPL}

\section{General Strategy and Combinatorial Structures}
\label{sec:genstrat-MOMENTS}
Assume that $(\sigma_n)_n$ is a sequence of ESDs of Wigner matrices $W_n$ as in Theorem~\ref{thm:wigner} and $(\mu_n)_n$ is a sequence of ESDs of MP matrices $V_n$ as in Theorem~\ref{thm:MP}. We would like to argue that $\sigma_n \to \sigma$ and $\mu_n\to\mu^y$ weakly for some $y>0$, and in some stochastic sense, for example in probability or almost surely. Here, $\sigma$ denotes the semicircle distribution and $\mu^y$ denotes the Marchenko-Pastur distribution on the real line. To show these convergence results, we carry out the following two steps, where notationally, either $\rho_n =\sigma_n$ and $\rho=\sigma$, or $\rho_n=\mu_n$ and $\rho=\mu^y$:
\begin{enumerate}
\item We show that for each fixed $k\in\N$, the expected moments $\E\integrala{\rho_n}{x^k}$ of the ESDs $\rho_n$ converge to the deterministic moments $\integrala{\rho}{x^k}$ of the limit measure $\rho$, as $n\to\infty$. By Theorem~\ref{thm:expectationandvariance}, this will ensure that the limit law holds in expectation.
\item For each fixed $k\in\N$, we find a finite decomposition of the random moments, $\integrala{\rho_n}{x^k}= D^{(k,1)}_n+ \ldots + D^{(k,\ell_k)}_n$, such that for each $k\in\N$ and each $i\in\oneto{\ell_k}$, $D^{(k,i)}_n$ converges in expectation to a constant as $n\to\infty$. This decomposition becomes clear from the analysis, for example from the first step, and. Then we show that for each $k\in\N $ and $i\in\oneto{\ell_k}$, there is a $z\in\N$ such that
\begin{equation}
\label{eq:secondstep}
\E\bigabs{D^{(k,i)}_n-\E D^{(k,i)}_n}^z\xrightarrow[n\to\infty]{} 0.
\end{equation}
Oftentimes, but not always, $z=2$ or $z=4$ will suffice. If \eqref{eq:secondstep} holds (resp.\ holds almost surely), then this will show that the $D^{(k,i)}_n$ converge in probability (resp.\ almost surely) to a constant so that with the first step, we obtain that for all $k\in\N$,
\[
\integrala{\rho_n}{x^k} = \sum_{i=1}^{\ell_k} D^{(k,i)}_n \xrightarrow[n\to\infty]{} \integrala{\rho}{x^k}
\]
in probability (resp./ almost surely).
\end{enumerate}

For our analysis, we introduce some combinatorial concepts. 
\begin{definition}
\label{def:coloring}
Let $k\in\N$ be arbitrary, then
\begin{enumerate}[i)]
\item A \emph{coloring} is a tuple $\ubar{c}\in\oneto{k}^k$ with the property that $c_1=1$ and
 \[
\forall\, a\in\{1,\ldots,k-1\}:\ c_{a+1}\leq 1+ \max_{\ell\in\oneto{a}}c_{\ell}.
\] 
Entries in a coloring will be called \emph{colors}.
\item If $\ubar{t}\in\oneto{n}^k$ is a tuple and $\ubar{c}$ is a coloring, then we say that \emph{$\ubar{t}$ matches the coloring $\ubar{c}$} (and write $\ubar{t}\sim\ubar{c}$), if
\[
\forall\, i,j\in\oneto{k}: t_i=t_j \Leftrightarrow c_i=c_j.
\]
In this case, we also call $\ubar{c}$ \emph{the coloring of $\ubar{t}$} and write $\ubar{c}=\ubar{c}(\ubar{t})$.
\end{enumerate}
\end{definition}
A coloring is used to indicate at which places in a tuple there are equal or different entries. It is clear that each tuple $\ubar{t}\in\oneto{n}^k$ matches exactly one (that is, \emph{its}) coloring, which is constructed inductively as follows. Set $c_1\defeq 1$, and for $\ell\in\{1,\ldots,k-1\}$, if there is no $m\in\oneto{\ell}$ with $t_{\ell+1}=t_m$, set $c_{\ell+1}=\max\{c_1,\ldots,c_{\ell}\}+1$, whereas if $t_{\ell+1}=t_m$ for some $m\in\oneto{\ell}$, set $c_{\ell+1}\defeq c_{m}$. As an example, the coloring of the tuple $(5,1,4,13,4)$ is given by $(1,2,3,4,3)$.
\begin{lemma}
\label{lem:coloring}
Let $n,k\in\N$ with $n\geq k$.
\begin{enumerate}[i)]
\item There are at most $k!$ colorings in $\oneto{k}^k$.
\item Let $\ubar{c}\in\oneto{k}^k$ be a coloring with $\ell$ colors, then
\begin{equation}
\label{eq:TuplesEquivColoring}	
\#\{\ubar{t}\in\oneto{n}^k:\ubar{t}\sim\ubar{c}\} = (n)_{\ell}\defeq n\cdot(n-1)\cdots (n-\ell+1)
\end{equation}
In addition, it always holds that $\ubar{c}\sim\ubar{c}$.
\item For a tuple $\ubar{t}\in\oneto{n}^k$ denote by $V(\ubar{t})\defeq\{t_1,\ldots,t_k\}$. Then $\ubar{c}(\ubar{t})$ has $\#V(\ubar{t})$ colors, hence
\begin{equation}
\label{eq:TuplesEquivTuple}
\#\{\ubar{t}'\in\oneto{n}^k: \ubar{t}'\sim c(\ubar{t})\} = (n)_{\#V(\ubar{t})}.
\end{equation}
\end{enumerate}
\end{lemma}
\begin{proof}
To prove $i)$, note that always $c_1=1$ and $c_{\ell+1}\in\{c_1,\ldots,c_{\ell},c_{\ell}+1\}$. But \newline$\#\{c_1,\ldots,c_{\ell},c_{\ell}+1\}\leq\ell+1$.
For $ii)$, in order to construct a tuple $\ubar{t}\in\oneto{n}^k$ matching the coloring $\ubar{c}$ we have $n$ choices for $t_1$. Then if $c_2=c_1$ this indicates that $t_2\overset{!}{=}{t_1}$ so we are left with only one choice for $t_2$. If $c_2\neq c_1$, however, we have $(n-1)$ choices for $t_2$. Proceeding this way, if $c_m=c_a$ for some $a<m$ then $t_m\overset{!}{=}{t_a}$ so we are left with only one choice for $t_m$. Otherwise, if $c_m$ is \emph{new} color, we have $n-\#\{c_1,\ldots,c_{m-1}\}$ choices for $t_m$. Now since there exactly $\ell$ different colors in $\ubar{c}$, we will encounter a new color exactly $\ell-1$ times.
Statement $iii)$ follows directly from $ii)$.
\end{proof}

\section{The Semicircle Law}
\label{sec:SemicircleByMoments}
Let $W_n = n^{-1/2}X_n$ be a sequence of Wigner matrices with ESDs $\sigma_n$. In order to show $\sigma_n\to\sigma$ weakly almost surely, we follow the general strategy as outlined in Section~\ref{sec:genstrat-MOMENTS}.
To utilize this method, we need the moments of $\sigma_n$ and $\sigma$. By Lemma~\ref{lem:momentsSCD}, the moments of $\sigma$ are given by
\begin{equation}
\label{eq:momentsWIG}	
\forall\, k\in\N: \integrala{\sigma}{x^k} = 
\begin{cases}
\frac{1}{\frac{k}{2}+1}\binom{k}{\frac{k}{2}} & \quad\text{if $k$ is even,}\\
0 &\quad\text{if $k$ is odd,}	
\end{cases}
\end{equation} 
whereas we may calculate the moments of $\sigma_n$ by (cf. Corollary~\ref{cor:momenttosum})
\begin{align}
&\integrala{\sigma_n}{x^k}\ =\ \integrala{\frac{1}{n}\sum_{t=1}^n\delta_{\lambda_t}}{x^k}\ = \ \frac{1}{n}\sum_{t\in\oneto{n}}\lambda_t^k\notag
=\frac{1}{n}\tr[W_n^k]\ =\ \frac{1}{n}\tr\left[\left(\frac{1}{\sqrt{n}}X_n\right)^k\right]\notag\\
&=\frac{1}{n^{1+\frac{k}{2}}}\sum_{t\in\oneto{n}} (X_n)^k(t,t)
\ =\ \frac{1}{n^{1+\frac{k}{2}}}\sum_{t_1,\ldots,t_k\in\oneto{n}}X_n(t_1,t_2)X_n(t_2,t_3)\ldots X_n(t_k,t_1)\notag\\
&=\frac{1}{n^{1+\frac{k}{2}}} \sum_{\ubar{t}\in\oneto{n}^k} X_n(\ubar{t}), \label{eq:elaboratesumWIG}
\end{align}
where for all $\ubar{t}\in\oneto{n}^k$ we define
\begin{equation}
\label{eq:Twalk}
X_n(\ubar{t}) \defeq X_n(t_1,t_2)X_n(t_2,t_3)\ldots X_n(t_k,t_1).
\end{equation}

\subsection*{Combinatorial Preparations and Graph Theory}

As we saw above in \eqref{eq:elaboratesumWIG}, the random moments $\integrala{\sigma_n}{x^k}$ expand into elaborate sums. In order to be able to analyze these sums, we sort them with the language of graph theory and then establish basic combinatorial facts.

Recall \eqref{eq:Twalk}, then we adopt the view that each tuple $\ubar{t}\in\oneto{n}^k$ spans a Eulerian graph as in Figure~\ref{fig:EulerianWig}

\begin{figure}[htbp]
\centering
\includegraphics[clip, trim=10cm 18cm 35cm 4cm, width=8cm]{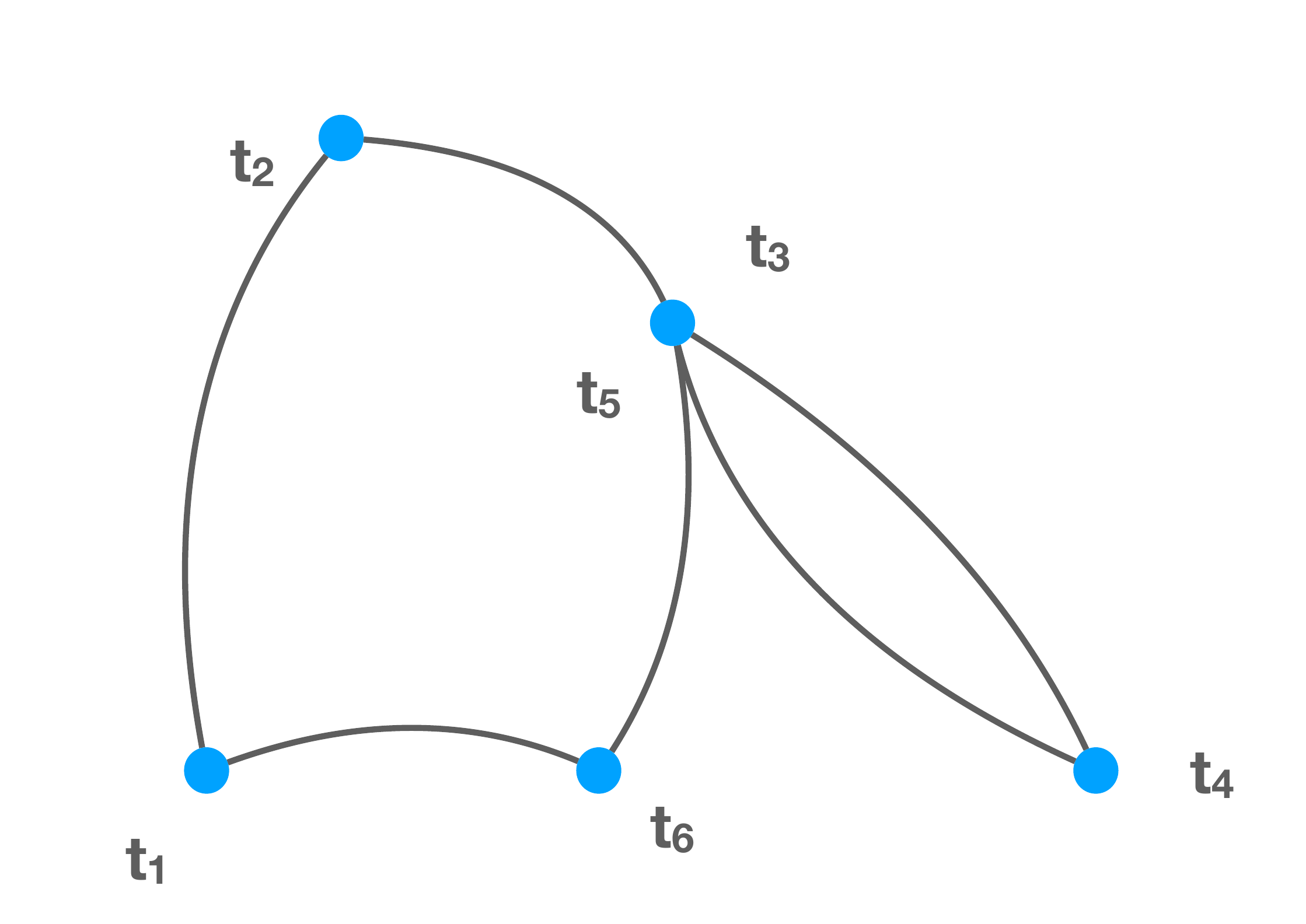}	
\caption{Eurlerian graph $\Gcal(\ubar{t})$.}
\label{fig:EulerianWig}
\end{figure}

To be precise, we obtain the (multi-)graph $\Gcal(\ubar{t}) = (V(\ubar{t}),E(\ubar{t}),\phi_{\ubar{t}})$, with vertex set $V(\ubar{t}) = \{t_1,\ldots,t_k\}$, edge set $E(\ubar{t}) = \{e_1,\ldots,e_k\}$ and incidence function
$\phi_{\ubar{t}}(e_i)=\{t_i,t_{i+1}\}$, where $k+1\equiv 1$.
Each tuple $\ubar{t}$ also denotes a Eulerian cycle of length $k$ through its graph $\Gcal(\ubar{t})$ by 
\begin{equation}
\label{eq:EuleriancycleWIG}	
t_1,e_1,t_2,e_2,t_3,\ldots,t_{k-1},e_{k-1},t_k,e_k,t_1.
\end{equation}

Note that $\Gcal(\ubar{t})$ may contain loops and multi-edges.
 The language of graph theory allows us to express $\integrala{\sigma_n}{x^k}$ in a different way. Recall
\begin{equation}
\label{eq:naivemomentWIG}	
\integrala{\mu_n}{x^k}\ =\ \frac{1}{n^{1+\frac{k}{2}}} \sum_{\ubar{t}\in\oneto{n}^k} X_n(\ubar{t})
\end{equation}
with
\begin{equation}
\label{eq:exmomentsummandWIG}
 X_n(\ubar{t}) = X_n(t_1,t_2)X_n(t_2,t_3)\cdots X_n(t_{k-1},t_k)X_n(t_k,t_1).
\end{equation}
For any tuple $\ubar{t}\in \oneto{n}^k$, we define its profile
\[
\rho(\ubar{t}) = (\rho_1(\ubar{t}),\ldots,\rho_{k}(\ubar{t})),
\]
where for all $\ell\in[k]$:
\[
\rho_{\ell}(\ubar{t}) \defeq \#\{\phi_{\ubar{t}}(e)\ |\ e\in E(\ubar{t})\ \text{is an $\ell$-fold edge}\}.
\]
Here, an $\ell$-fold edge in $E(\ubar{t})$ is any element $e\in E(\ubar{t})$ for which there are exactly $\ell-1$ distinct other elements $e'_2,\ldots,e'_{\ell} \in E(\ubar{t})$ so that $\phi_{\ubar{t}}(e)=\phi_{\ubar{t}}(e'_j)$ for $j\in\{2,\ldots,\ell\}$.

Then for all $\ell\in[k]$, the Eulerian cycle $\ubar{t}$ traverses exactly $\phi_{\ell}(\ubar{t})$ distinct $\ell$-fold edges. As a result, the following trivial but useful equality holds:
\begin{equation}
\label{eq:alledgessumWIG}
k = \sum_{\ell=1}^{k}\ell\cdot\rho_{\ell}(\ubar{t}).
\end{equation}
Now for all $k\in\N$ we define the following set of profiles:
\[
\Pi(k) = \left\{\rho \in\{0,\ldots,k\}^{k}\ |\ \rho \ \text{profile of some } \ubar{t}\in\oneto{n}^k\right\}.
\]
Now we achieve a finite decomposition
\begin{equation}
\label{eq:graphmomentWIG}	
\integrala{\sigma_n}{x^k} =\sum_{\rho\in\Pi(k)}\frac{1}{n^{1+\frac{k}{2}}}\sum_{\ubar{t}\in\Tcal^{n}(\rho)}  X_n(\ubar{t}),
\end{equation}
where
\[
\Tcal^{n}(\rho) \defeq \left\{\ubar{t}\in\oneto{n}^k\ |\ \rho(\ubar{t})=\rho \right\}.
\]
The transition from \eqref{eq:naivemomentWIG} to \eqref{eq:graphmomentWIG} allows us to identify exactly which components of the random moment contribute to the limit.

The next fundamental lemma will give an upper bound on the number of tuples  $\ubar{t}$ with at most $\ell\in\oneto{k}$ vertices. Notationally, we set $V(\ubar{u})\defeq\{u_1,\ldots,u_k\}$ for any $\ubar{u}\in\N^k$, even if we do not interpret $\ubar{u}$ as a graph. Further, if $M$ is a set, $\# M\label{sym:numberofelements} \in \N\cup\{\infty\}$ denotes the number of elements in $M$

\begin{lemma}\label{lem:maxnodestuplesWIG}
Let $n,k\in\N$ and $\ell\in\{1,2,\ldots,k\}$ be arbitrary. Then
\[
\# \left\{\ubar{t}\in\oneto{n}^k\, |\, \#V(\ubar{t})\leq \ell\right\} \leq k^k\cdot n^{\ell}.
\]
\end{lemma}
\begin{proof}
We first pick a coloring $\ubar{c}\in\oneto{k}^k$ with at most $\ell$ colors for which we have at most $k^k$ choices by Lemma~\ref{lem:coloring} $i)$. Since $\ubar{c}$ has at most $\ell$ colors, the number of tuples $\ubar{t}$ matching the coloring is bounded by $(n)_{\ell}$ by Lemma~\ref{lem:coloring} $ii)$. Therefore, we have at most $k^k(n)_{\ell}$ choices to pick an element from $\left\{\ubar{t}\in\oneto{n}^k\, |\, \#V(\ubar{t})\leq \ell\right\}$.
\end{proof}

\subsection*{Step 1: Convergence of expected moments}
We proceed to analyze the expectation of
\begin{equation}
\label{eq:randmomWIG}
\integrala{\sigma_n}{x^k} =\sum_{\rho\in\Pi(k)}\frac{1}{n^{1+\frac{k}{2}}}\sum_{\ubar{t}\in\Tcal^n(\rho)}  X_n(\ubar{t}).
\end{equation}
To this end, it suffices to analyze the expectation of each of the finitely many terms
\begin{equation}
\label{eq:randmomWIGsummand}
\frac{1}{n^{1+\frac{k}{2}}}\sum_{\ubar{t}\in\Tcal^n(\rho)}  X_n(\ubar{t})
\end{equation}
for $\rho\in\Pi(k)$ separately. We make two trivial observations: If $\rho\in\Pi(k)$ with $\rho_1>0$, then for all $\ubar{t}\in\Tcal^n(\rho)$ it holds $\E X_n(\ubar{t})=0$ due to independence and centeredness. Further, since $(X_n)_n$ is a Wigner scheme as in Definition~\ref{def:Wignerscheme}, we can always apply the trivial bound
\begin{equation}
\label{eq:trivialboundWIG}
\abs{\E X_n(\ubar{t})} \leq L_k	
\end{equation}
for any $\ubar{t}\in\oneto{n}^k$, where we also used Lemma~\ref{lem:holder}.

For the bounds on $\#\Tcal^n(\rho)$, we formulate the next lemma, which we take from \autocite{FleermannDiss}.

\begin{lemma}\label{lem:nodeandtuplecountWIG}
Let $k \in\N$ be arbitrary. Then it holds: 
\begin{enumerate}[i)]
\item $\#\Pi(k)\leq 4^k.$
\item Let $n\in\N$ and $\rho\in\Pi(k)$ be arbitrary, then
\begin{enumerate}[a)]
\item For any $\ubar{t}\in\Tcal^n(\rho)$ it holds
\[
\# V(\ubar{t})\leq 1 + \rho_1 + \ldots + \rho_{k} - L(\ubar{t}),
\]
where $L(\ubar{t})$ denotes the number of loops in $\ubar{t}$. In particular,
\[
\#\Tcal^n(\rho)\leq k^k \cdot  n^{1 + \rho_1 + \ldots + \rho_{k}}.
\]
\item If $\rho$ contains an odd edge, then for any $\ubar{t}\in\Tcal^{n}(\rho)$ it holds
\[
\# V(\ubar{t})\leq \rho_1 + \ldots + \rho_{k}.
\]
In particular,
\[
\#\Tcal^n(\rho)\leq k^k \cdot n^{\rho_1 + \ldots + \rho_{k}}.
\]
\end{enumerate}
\end{enumerate}
\end{lemma}

\begin{proof}
\underline{i)}  Each $\rho\in\Pi(k)$ is a $k$-tuple in which for all $\ell\in\{1,\ldots,k\}$ the entry $\rho_{\ell}$ lies in the set $\{0,1,\ldots,\lfloor k/\ell\rfloor\}$, which follows directly from \eqref{eq:alledgessumWIG}. Therefore,
\[
\#\Pi(k) \leq  \prod_{\ell=1}^{k}\left(\frac{k}{\ell}+1\right) = \frac{(2k)!}{k!\cdot k!} = \binom{2k}{k}  \lesssim \frac{4^k}{\sqrt{2k\pi}} \leq 4^k,
\]
where the fourth step is a well-known fact about the central binomial coefficient.\newline
\underline{ii)} It suffices to establish the upper bounds for $\#V(\ubar{t})$, since the bounds on $\#\Tcal^{n}(\rho)$ then follow directly with Lemma~\ref{lem:maxnodestuplesWIG}. Now to prove upper bounds for $\#V(\ubar{t})$, the idea is to travel the Eulerian cycle generated by $\ubar{t}:$
\begin{equation}
\label{eq:walkWIG}
t_1,e_1,t_2,e_2,t_2,e_3,t_3,\ldots,t_k, e_{k}, t_1
\end{equation}
by picking an initial node $t_i$ and then traversing the edges in increasing cyclic order until reaching the starting point again. On the way, we count the number of different vertices that were discovered. Whenever we pass an $\ell$-fold edge, only the first instance of that edge may discover a new vertex, and only if the edge is not a loop.\newline
\underline{a)}  We write $L(\ubar{t}) \defeq L_1(\ubar{t}) +\ldots + L_k(\ubar{t})$ where $L_i(\ubar{t})$ denotes the number of different $i$-fold loops in $\ubar{t}$. We start our tour at  $t_1$ and observe this very vertex. Then, as we travel along the cycle, for each $\ell\in\{1,\ldots,k\}$ we will pass $\ell\cdot (\rho_{\ell}-L_{\ell}(\ubar{t}))$ \emph{proper} $\ell$-fold edges out of which only the first instance may discover a new node, and there are $\rho_{\ell}-L_{\ell}(\ubar{t})$ of these first instances. Considering the initial node, we arrive at $\# V(\ubar{t})\leq 1 + \rho_1 - L_{1}(\ubar{t}) + \ldots + \rho_{k} - L_k(\ubar{t})$, which yields the desired inequality.\newline
\underline{b)} In presence of an odd edge, we can start the tour at a specific vertex such that the odd edge cannot contribute to the newly discovered vertices. To this end, fix an arbitrary $\ell$-fold edge in  $\ubar{t}$ with $\ell$ odd. Let $e_{i_1},\ldots,e_{i_{\ell}}$, $i_1 < \ldots < i_{\ell}$, be the instances of the $\ell$-fold edges in question in the cycle \eqref{eq:walkWIG}. Since $\ell$ is odd, we must find a $k\in\{1,\ldots,\ell\}$ such that $e_{i_k}$ and $e_{i_{k+1}}$ are traversed in the same direction, since we are on a cycle.  We then start our tour at $t_{i_k}$ and observe this vertex. However, now none of the edges $e_{i_1},\ldots,e_{i_{\ell}}$ may discover a new vertex, since if our $\ell$-fold edge is not a loop, the vertex $t_{i_{k+1}}$ must have been already discovered by some other edge. Therefore, the roundtrip leads to the discovery of at most $\rho_1 + \dots + (\rho_{\ell} - 1) + \ldots + \rho_{k}$ new nodes in addition to the first node.
\end{proof}

We proceed to analyze \eqref{eq:randmomWIGsummand} for all possible types of $\rho\in\Pi(k)$.\newline
\underline{Case 1: $\rho_1=0$ and $\rho_{\ell}>0$ for some $\ell\geq 3$.}\newline
Using Lemma~\ref{lem:nodeandtuplecountWIG} we obtain
\[
\#\Tcal^n(\rho)\leq
\left\{
\begin{array}{c}
k^k \cdot n^{\rho_1 + \ldots + \rho_{k}}\\
k^k \cdot n^{1 + \rho_1 + \ldots + \rho_{k}} 	
\end{array}
\right\}
\leq k^k n^{\frac{k}{2}},
\]
where the upper case is valid in presence of an odd edge (then $\rho_1 + \ldots + \rho_k \leq (k-3)/2 + 1$), and the lower case is valid if no odd edges are present (then $1+\rho_1 + \ldots + \rho_k \leq 1 + (k-4)/2+1$). Therefore, by \eqref{eq:trivialboundWIG}, \eqref{eq:randmomWIGsummand} converges to zero in expectation.

\noindent\underline{Case 2: $\rho_{1}>0$.}\newline
Then by centeredness and independence, the expectation of the term in \eqref{eq:randmomWIGsummand} is zero.

\noindent\underline{Case 3: $\rho_{2}=k/2$.}\newline
Returning to the random moment in \eqref{eq:randmomWIG}, we have seen in Cases 1 and 2 that for all $\rho\in\Pi(k)$ with $\rho_{2}\neq k/2$,
\[
\frac{1}{n^{1+\frac{k}{2}}}\sum_{\ubar{t}\in\Tcal^n(\rho)}  X_n(\ubar{t}) \xrightarrow[n\to\infty]{} 0 \qquad \text{in expectation.}
\]
As a result, the only asymptotic contribution from the expectation in \eqref{eq:randmomWIG} may stem from cycles $\ubar{t}$ containing only double edges. Their analysis is the content of this Case 3. Setting $\rho^{(k)}$ as the profile in $\Pi(k)$ with $\rho^{(k)}_2=k/2$ and $\rho^{(k)}_{\ell}=0$ for all $\ell\neq 2$, then it is our goal to show 
\begin{equation}
\label{eq:goalCase3}
\frac{1}{n^{1+\frac{k}{2}}}\sum_{\ubar{t}\in\Tcal^n(\rho^{(k)})} X_n(\ubar{t})\ \xrightarrow[n\to\infty]{}\	 \Cat_{\frac{k}{2}} \qquad \text{in expectation}.
\end{equation}
To this end, we observe
\begin{equation}
\label{eq:summandCase3}
\frac{1}{n^{1+\frac{k}{2}}}\sum_{\ubar{t}\in\Tcal^n(\rho^{(k)})} \E X_n(\ubar{t}) = \frac{1}{n^{1+\frac{k}{2}}}\#\Tcal^n(\rho^{(k)}).
\end{equation}
Next, we note that any $\ubar{t}\in\Tcal^n(\rho^{(k)})$ has at most $k/2+1$ vertices, so we may subdivide this set further by defining
\begin{align*}
\Tcal^{n}_{\leq k/2}(\rho^{(k)}) &\defeq \left\{\ubar{t}\in\Tcal^n(\rho^{(k)}):\ \# V(\ubar{t}) \leq k/2 \right\},\\
\Tcal^n_{k/2+1}(\rho^{(k)}) &\defeq \left\{\ubar{t}\in\Tcal^n(\rho^{(k)}):\ \# V(\ubar{t}) = k/2+1 \right\}.
\end{align*}
Note that by Lemma~\ref{lem:maxnodestuplesWIG}, $\#\Tcal^n_{\leq k/2}(\rho^{(k)}) \leq k^k n^{k/2}$, so that \eqref{eq:summandCase3} can be refined to
\begin{equation}
\label{eq:onlydoublekplusoneWIG}
\frac{1}{n^{1+\frac{k}{2}}}\sum_{\ubar{t}\in\Tcal^n(\rho^{(k)})} \E X_n(\ubar{t})  = \frac{1}{n^{1+\frac{k}{2}}}\#\Tcal^n_{k/2+1}(\rho^{(k)})  \ + \ o(1), 
\end{equation}
It is thus our task to show
\begin{equation}
\label{eq:convergetoCatalan}	
\frac{1}{n^{1+\frac{k}{2}}}\#\Tcal^n_{k/2+1} \xrightarrow[n\to\infty]{} \Cat_{\frac{k}{2}}.
\end{equation}

The main tool is to count all possible colorings of tuples in $\Tcal^n_{k/2+1}(\rho^{(k)})$, and then apply Lemma~\ref{lem:coloring}. It turns out that these colorings can be associated with a path difference sequence (pds) of the following form, where we may focus on even $k$, since otherwise, the set $\Tcal^n_{k/2+1}(\rho^{(k)})$ is empty:
\begin{definition}
\label{def:WIG-path}
A \emph{Wigner path difference sequence} (Wigner-pds) of length $2k$ is a tuple $(D_1,D_2,\ldots,D_{2k})$ which satifies the following conditions:
\begin{enumerate}[1)]
\item For all $i\in\oneto{2k}$: $D_i\in\{-1,+1\}$
\item $\sum_{i\in\oneto{2k}}D_i = 0$,
\item $\forall\,\ell\in\oneto{2k}: \sum_{i=1}^{\ell}D_i\geq 0$.
\end{enumerate}
We denote by $\Wcal(2k)$ the set of all Wigner-pds of length $2k$.	
\end{definition}

\begin{lemma}
\label{lem:coloringsWIG}
For all $k\in\N$ we find $\#\Wcal(2k) = \frac{1}{k+1}\binom{2k}{k}=\Cat_k$.
\end{lemma}
\begin{proof}
We prove the lemma with a reflection principle. To this end, due property 2), a Wigner-pds must contain as many "$+1$"-entries as "$-1$"-entries. To arrange $k$ "$+1$"-entries and $k$ "$-1$"-entries, we have
\[
\binom{2k}{k}
\]
choices. But since these choices do not in general respect condition $3)$ we have to subtract the number of tuples $(D_1,\ldots,D_{2k})$ that lead to a violation of $3)$. We show that these violating tuples are in bijective correspondence to all $(D_1',\ldots,D_{2k}')$ with
\begin{enumerate}[1')]
\item $D'_i\in\{-1,+1\}$,
\item $\sum_{i\in\oneto{2k}}D_i' = -2$. 
\end{enumerate}
The number of these $(D_1',\ldots,D_k')$ is clearly given by 
\[
\binom{2k}{k+1}
\]
so that the number of $(D_1,\ldots,D_{2k})$ that \emph{do} satisfy 1), 2) \emph{and} 3) is given by
\[
\binom{2k}{k} - \binom{2k}{k+1} = \frac{1}{k+1}\binom{2k}{k}.
\]
For the bijection, let $(D_1,\ldots,D_{2k})$ be arbitrary with $k$ "$+1$"s and $k$ "$-1$"s so that 3) is violated. Then there is an index $t$ such that $\sum_{i=1}^t D_i = -1$ for the first time. Then $(D_{t+1},\ldots,D_{2k})$
 is a vector which contains one more "$+1$" than "$-1$" entry. We define the vector  $(D'_{t+1},\ldots,D'_{2k})\defeq (-D_{t+1},\ldots,-D_{2k})$. Then $(D'_{t+1},\ldots,D'_{2k})$ contains one more "$-1$" than "$+1$". Defining $(D_1',\ldots,D_t')\defeq(D_1,\ldots,D_t)$ we thus have created a vector $(D_1',\ldots,D_{2k}')$ satisfying 1') and 2'). On the other hand, any vector $(D_1',\ldots,D_{2k}')$ satisfying 1) and 2) has a first hitting time $t$ of $-1$. Applying exactly the same transformation as before, we will then obtain a vector $(D_1,\ldots,D_{2k})$ satisfying 1) and 2), but violating 3).
 \end{proof}

Now the clou is that all $D\in\Wcal(2k)$ can be associated canonically with a specific Eulerian cycle $\ubar{t}(D)\in\Tcal^n_{k+1}(\rho^{(2k)})$. To see how this is done, let us first analyze simple properties a Eulerian cycle $\ubar{t}\in\Tcal^n_{k+1}(\rho^{(2k)})$. First, the graph $\Gcal(\ubar{t})$ is a \emph{double edged tree}, that is, it consists of $k$ distinct double edges and has $k+1$ vertices, therefore is a tree in the regular sense after eliminating one of each of the double edges (it also follows that all doubles edges are \emph{proper}). Thus, the Eulerian cycle $\ubar{t}$ crosses each edge twice, once in each direction, since a tree does not contain circles. We recall the representation of the cycle as in \eqref{eq:EuleriancycleWIG}. Now given a $D\in\Wcal(2k)$, we set $t_1=1$, and whenever $D_{\ell}=+1$, this means that a new vertex shall be discovered, so we set $t_{\ell +1}\defeq\max(t_1,\ldots,t_{\ell})+1$. On the other hand, if $D_{\ell}=-1$ then we shall backtrack, that is, $t_{\ell+1}$ shall be equal to one of the $t_1,\ldots,t_{\ell}$, and so it must be equal to the $t_i$ with $i\in\{1,\ldots,\ell\}$ from which $t_{\ell}$ was visited, since otherwise, the cycle $\ubar{t}$ would contain a circle. This completes the construction of $\ubar{t}(D)$. It is clear by construction that $\ubar{t}(D)\in\Tcal^{n}_{k+1}(\rho^{(2k)})$ . We observe that $\ubar{c}(\ubar{t}(D))=\ubar{t}(D)$, that is $\ubar{t}(D)$ is its own coloring, since vertex numbers were always chosen as small as possible. Now if $\ubar{t}'\sim\ubar{c}(\ubar{t}(D))$, we must have $\ubar{t}'\in\Tcal^{n}_{k+1}(\rho^{(2k)})$, since $\ubar{t}'$ is then only an injective relabeling of vertices in $\ubar{t}$.

We formulate the following Lemma from which \eqref{eq:convergetoCatalan} follows immediately.
\begin{lemma}
\label{lem:TnWIG}
The set $\Tcal^n_{k+1}(\rho^{(2k)})$ has a decomposition as follows:
\begin{equation}
\label{eq:decomposeTnWIG}
\Tcal^{n}_{k+1}(\rho^{(2k)}) = \dot{\bigcup_{D\in\Wcal(2k)}}\left\{\ubar{t}'\in \Tcal^{n}_{k+1}(\rho^{(2k)}) \ | \ \ubar{t}'  \sim \ubar{c}(\ubar{t}(D))\right\}
\end{equation}
In particular, 
\begin{equation}
\label{eq:countTnWIG}
\#\Tcal^n_{k+1}(\rho^{(2k)}) = \frac{1}{k+1}\binom{2k}{k}\cdot(n)_{k+1}\ .
\end{equation}
\end{lemma}
\begin{proof}
Before the statement of Lemma~\ref{lem:TnWIG}, we have already argued "$\supseteq$" in \eqref{eq:decomposeTnWIG}. To show "$\subseteq$", let $\ubar{t}'\in\Tcal^{n}_{k+1}(\rho^{(2k)})$ be arbitrary and  recall the representation of the cycle as in \eqref{eq:EuleriancycleWIG}. We encode this cycle into a Wigner-pds $D(\ubar{t}')$ and show that $\ubar{t}'\sim\ubar{c}(\ubar{t}(D(\ubar{t}'))$. To this end, start a tour at $t_1'$ and move along the cycle. For $\ell\in\{1,\ldots,2k\}$, if $e_{\ell}$ leads to a new vertex, we set $D_{\ell}=1$ and if $e_{\ell}$ backtracks to an old vertex, we set  $D_{\ell}=-1$. For example, we always have $D_{1}=1$, since each edge in $\ubar{t}$ is proper, and $D_k=-1$, since this edge leads back to the -- already seen --  vertex $t_1'$. 
Let us argue that the tuple $D(\ubar{t}')\defeq(D_1,D_2,\ldots,D_{2k})$ we just constructed satisfies conditions 1), 2) and 3) as above. Condition 1) is clearly satisfied. For condition 2), note that $\ubar{t}'$ has $k+1$ vertices, out of which $k$ -- all except the vertex $t_1'$ -- were considered new while traversing $\ubar{t}'$, so we must have $k$ "$+1$"-entries and $k$ "$-1$"-entries in $(D_1,D_2,\ldots,D_{2k})$. For condition 3) we realize that each vertex in $\ubar{t}'$ is visited exactly twice by the cycle $\ubar{t}'$, and that the first visit corresponds to a "$+1$"-entry while the second visit corresponds to a "$-1$"-entry in $(D_1,D_2,\ldots,D_{2k})$. Then 3) must be satisfied, since by nature of things, the "first" comes before the "second". The relation $\ubar{t}'\sim\ubar{c}(\ubar{t}(D(\ubar{t}'))$ follows with the construction of $\ubar{t}(D(\ubar{t}'))$ above the formulation of Lemma~\ref{lem:TnWIG}.

The equality \eqref{eq:countTnWIG} follows from Lemma~\ref{lem:coloringsWIG},  \eqref{eq:decomposeTnWIG} and Lemma~\ref{lem:coloring} $iii)$, since for all $D\in\Wcal(2k)$ we have $\#V(\ubar{t}(D))=k+1$ and all tuples matching the coloring $\ubar{c}(\ubar{t}(D))$ lie in $\Tcal^{n}_{k+1}(\rho^{(2k)})$.
\end{proof}

\subsection*{Step 2: Decay of central moments}

In Step 1, we have seen that for fixed $k\in\N$, the expectation of
\begin{equation}
\label{eq:randmomWIGStep2}
\integrala{\sigma_n}{x^k} =\sum_{\rho\in\Pi(k)}\frac{1}{n^{1+\frac{k}{2}}}\sum_{\ubar{t}\in\Tcal^n(\rho)}  X_n(\ubar{t})
\end{equation}
converges to the $k$-th moment of the semicircle distribution. In particular, we have seen that each of the finitely many summands 
\begin{equation}
\label{eq:randmomWIGsummandStep2}
\frac{1}{n^{1+\frac{k}{2}}}\sum_{\ubar{t}\in\Tcal^n(\rho)}  X_n(\ubar{t})
\end{equation}
converges to a constant in expectation. To show that the random moments in \eqref{eq:randmomWIGStep2} converge almost surely to the moments of the semicircle distribution, it thus suffices -- by Lemma~\ref{lem:convergencemodes} -- to show that for all $\rho\in\Pi(k)$, the variance of each term in \eqref{eq:randmomWIGsummandStep2} decays summably fast. The variance of \eqref{eq:randmomWIGsummandStep2} is given by 
\begin{equation}
\label{eq:WignerRhoVariance}
\frac{1}{n^{k+2}}\sum_{\ubar{t},\ubar{t}'\in\Tcal^n(\rho^{(k)})} \left[\E X_n(\ubar{t})X_n(\ubar{t}')- \E X_n(\ubar{t})\E X_n(\ubar{t}')\right].
\end{equation}

We observe that for all $\ubar{t},\ubar{t}'\in\Tcal^n(\rho^{(k)})$ which are edge-disjoint, the corresponding summand in \eqref{eq:WignerRhoVariance} vanishes. Thus it suffices to consider those $\ubar{t},\ubar{t}'\in\Tcal^n(\rho)$ which have at least one edge in common. To this end, 
denote for all $\ell\in\oneto{k}$: 
\[
\Tcal_{c(\ell)}^n(\rho)\defeq \left\{(\ubar{t},\ubar{t}') \in(\Tcal^n(\rho))^2 \,|\, \text{$\ubar{t}$ and $\ubar{t}'$ have exactly $\ell$ edges in common}\right\}.
\]
Our goal now is to evaluate for each $\ell\in\oneto{k}$ the term
\begin{equation}
\label{eq:tozeroas}
\frac{1}{n^{k+2}}\sum_{(\ubar{t},\ubar{t}')\in\Tcal_{c(\ell)}^n(\rho)} \left[\E X_n(\ubar{t})X_n(\ubar{t}')- \E X_n(\ubar{t})\E X_n(\ubar{t}')\right].
\end{equation}
To this end, we need to establish bounds on $\#\Tcal_{c(\ell)}^n(\rho)$.

\begin{lemma}
\label{lem:doubleboundWIG}
Let $\rho\in\Pi(k)$ and $\ell\in\oneto{k}$, then the following statements hold:
\begin{enumerate}[i)]
\item For all $\ubar{t},\ubar{t}'\in\Tcal^n(\rho)$ with at least $\ell$ common edges, it holds
\[
\#(V(\ubar{t})\cup V(\ubar{t}')) \leq 1 + 2\sum_{i=1}^{k} \rho_i - \ell
\]
In particular,
\[
\#\Tcal_{c(\ell)}^n(\rho) \leq (2k)^{2k} n^{1 + 2\sum_{i=1}^{k} \rho_i - \ell}
\]
\item If there is an $m\in\oneto{k}$ odd with $\rho_{m}\geq 1$, then for all $\ubar{t},\ubar{t}'\in\Tcal^n(\rho)$ with at least $\ell$ common edges, it holds
\[
\#(V(\ubar{t})\cup V(\ubar{t}')) \leq 2\sum_{i=1}^{k} \rho_i - \ell.
\]
In particular,
\[
\#\Tcal_{c(\ell)}^n(\rho) \leq (2k)^{2k} n^{2\sum_{i=1}^{k} \rho_i - \ell}.
\]
\end{enumerate}
\end{lemma}
\begin{proof}
For statement $ii)$ we assume w.l.o.g.\ that $\ubar{t}$ has an odd edge. Since the graphs spanned by $\ubar{t}$ and $\ubar{t}'$ share $\ell\geq 1$ common edges, we may take a tour around the joint Eulerian cycle, starting before a common edge, traveling first all edges of $\ubar{t}$ and then all edges of $\ubar{t}'$. While walking the edges of $\ubar{t}$, we can see at most $\rho_1 + \ldots + \rho_{k}$ different nodes by Lemma~\ref{lem:nodeandtuplecountWIG}. Next, traveling all edges of $\ubar{t}'$, at most all the single edges and first instances of $m$-fold edges with $m\in\{2,\ldots,k\}$ of $\ubar{t}'$ may discover a new node, but only if they have not been traversed before during the walk along $\ubar{t}$. Since we have $\ell$ common edges, we can see at most $\rho'_1+\ldots + \rho'_{k}-\ell$ new nodes. We established the bounds on the number of vertices in $ii)$. The second statement in $ii)$ follows immediately with Lemma~\ref{lem:maxnodestuplesWIG} $i)$ by concatenating $(\ubar{t},\ubar{t}')\in\oneto{n}^{2k}$. For statement $i)$ we proceed exactly in the same manner: Traveling $\ubar{t}$ we can see at most $1 + \rho_1 + \rho_2 + \ldots + \rho_{k}$ nodes by Lemma~\ref{lem:nodeandtuplecountWIG}, then traveling $\ubar{t}'$ we can see at most $\rho'_1+\ldots + \rho'_{k}-\ell$ new nodes. Now apply Lemma~\ref{lem:maxnodestuplesWIG} again.
\end{proof}

\noindent\underline{Case 1: $\rho_1\geq 1$}\newline
In this case, the term in \eqref{eq:tozeroas} simplifies and we must argue that for each $\ell\in\oneto{k}$,
\begin{equation}
\label{eq:RhoSingleVariance}
\frac{1}{n^{k+2}}\sum_{(\ubar{t},\ubar{t}')\in\Tcal_{c(\ell)}^n(\rho)} \E X_n(\ubar{t})X_n(\ubar{t}')
\end{equation}
decays summably fast to zero. But we note that if  $\ubar{t}$ and $\ubar{t}'$ have $1\leq\ell<\rho_1$ common edges, $\E X_n(\ubar{t})X_n(\ubar{t}')$ vanishes, since not all single edges can be eliminated due to overlapping. Thus, it suffices to consider those $\ubar{t},\ubar{t}'\in\Tcal^n(\rho)$ which have $\ell\geq \rho_1$ edges in common. Now if $\rho\in\Pi(k)$ with $\ell\geq\rho_1\geq 1$, then
\[
2\sum_{i=1}^{k} \rho_i - \ell \leq 2\left(\rho_1 + \frac{k-\rho_1}{2}\right) -\rho_1  \leq k.  
\]
so Lemma~\ref{lem:doubleboundWIG} $ii)$ yields
\[
\#\Tcal_{c(\ell)}^n(\rho) \leq (2k)^{2k} n^{2\sum_{i=1}^{k} \rho_i - \ell}\leq (2k)^{2k} n^k,
\]
Since every summand in \eqref{eq:RhoSingleVariance} is bounded by $L_{2k}$, it follows that \eqref{eq:RhoSingleVariance} is $O(n^{-2})$, thus converges to zero summably fast.

\noindent\underline{Case 2: $\rho_1=0$}\newline
In this case, each summand in \eqref{eq:tozeroas} is bounded by $L_{2k} + L_k^2$. Further, we obtain for all $\rho\in\Pi(k)$ with $\rho_1=0$ and $\ell\geq 1$ that
\[
1+ 2\sum_{i=1}^k \rho_i - \ell \leq 1 + 2 \cdot\frac{k}{2} -1 = k,
\]
so that by Lemma~\ref{lem:doubleboundWIG} $i)$ we find
\[
\#\Tcal_{c(\ell)}^n(\rho) \leq (2k)^{2k} n^{1+2\sum_{i=1}^{k} \rho_i - \ell}\leq (2k)^{2k} n^k,
\]
so that the sum in \eqref{eq:tozeroas} is $O(n^{-2})$, hence converges to zero summably fast.

\section{The Marchenko-Pastur Law}
\label{sec:MPByMoments}
Let $(X_n)_n$ be an MP scheme as in Definition~\ref{def:MPscheme}, $V_n \defeq n^{-1}X_nX_n^T$, and $\mu_n$ be the ESDs of $V_n$. Denote $y\defeq \lim_n p/n$. In order to show $\mu_n\to\mu^y$ weakly almost surely, we follow the general strategy as outlined in Section~\ref{sec:genstrat-MOMENTS}. To utilize this method, we need the moments of $\mu_n$ and $\mu^{y}$. By Lemma~\ref{lem:momentsMPD}, the moments of $\mu^{y}$ are given by
\begin{equation}
\label{eq:momentsMP}	
\forall\, k\in\N: \integrala{\mu^y}{x^k} = \sum_{r=0}^{k-1} \frac{y^r}{r+1}\binom{k}{r}\binom{k-1}{r},
\end{equation} 
whereas we may calculate the moments of $\mu_n$ by (cf. Corollary~\ref{cor:momenttosum})
\begin{align}
&\integrala{\mu_n}{x^k}\ =\ \integrala{\frac{1}{p}\sum_{s=1}^p\delta_{\lambda_s}}{x^k}\ =\ \frac{1}{p}\sum_{l\in\oneto{p}}\lambda_l^k
=\frac{1}{p}\tr[V_n^k]\ =\ \frac{1}{p}\tr\left[\left(\frac{1}{n}X_nX_n^T\right)^k\right]\notag\\
&=\frac{1}{pn^k}\sum_{s\in\oneto{p}} (X_nX_n^T)^k(s,s)
\ =\ \frac{1}{pn^k}\sum_{s_1,\ldots,s_k\in\oneto{p}}(X_n X_n^T)(s_1,s_2)(X_nX_n^T)(s_2,s_3)\ldots(X_nX_n^T)(s_k,s_1)\notag\\
&=\frac{1}{pn^k}\sum_{s_1,\ldots,s_k\in\oneto{p}}\sum_{t_1,\ldots,t_k\in\oneto{n}} X_n(s_1,t_1)X_n(s_2,t_1)X_n(s_2,t_2)X_n(s_3,t_2)\ldots X_n(s_k,t_k)X_n(s_1,t_k)\notag\\
&=\frac{1}{pn^k} \sum_{\ubar{s}\in\oneto{p}^k}\sum_{\ubar{t}\in\oneto{n}^k} X_n(\ubar{s},\ubar{t}),\label{eq:elaboratesumMP}
\end{align}
where for all $\ubar{s}\in\oneto{p}^k$ and $\ubar{t}\in\oneto{n}^k$ we define
\begin{equation}
\label{eq:STwalk}
X_n(\ubar{s},\ubar{t}) \defeq X_n(s_1,t_1)X_n(s_2,t_1)X_n(s_2,t_2)X_n(s_3,t_2)\ldots X_n(s_k,t_k)X_n(s_1,t_k).
\end{equation}

\subsection*{Combinatorial Preparations and Graph Theory}

As we saw above in \eqref{eq:elaboratesumMP}, the random moments $\integrala{\mu_n}{x^k}$ expand into elaborate sums. In order to be able to analyze these sums, we sort them with the language of graph theory and then establish basic combinatorial facts.

Recall \eqref{eq:STwalk}, then we adopt the view that each pair $(\ubar{s},\ubar{t})\in\oneto{p}^k\times\oneto{n}^k$ spans a Eulerian bipartite graph as follows:

\begin{figure}[htbp]
\centering
\includegraphics[clip, trim=10cm 12cm 10cm 4cm, width=\textwidth]{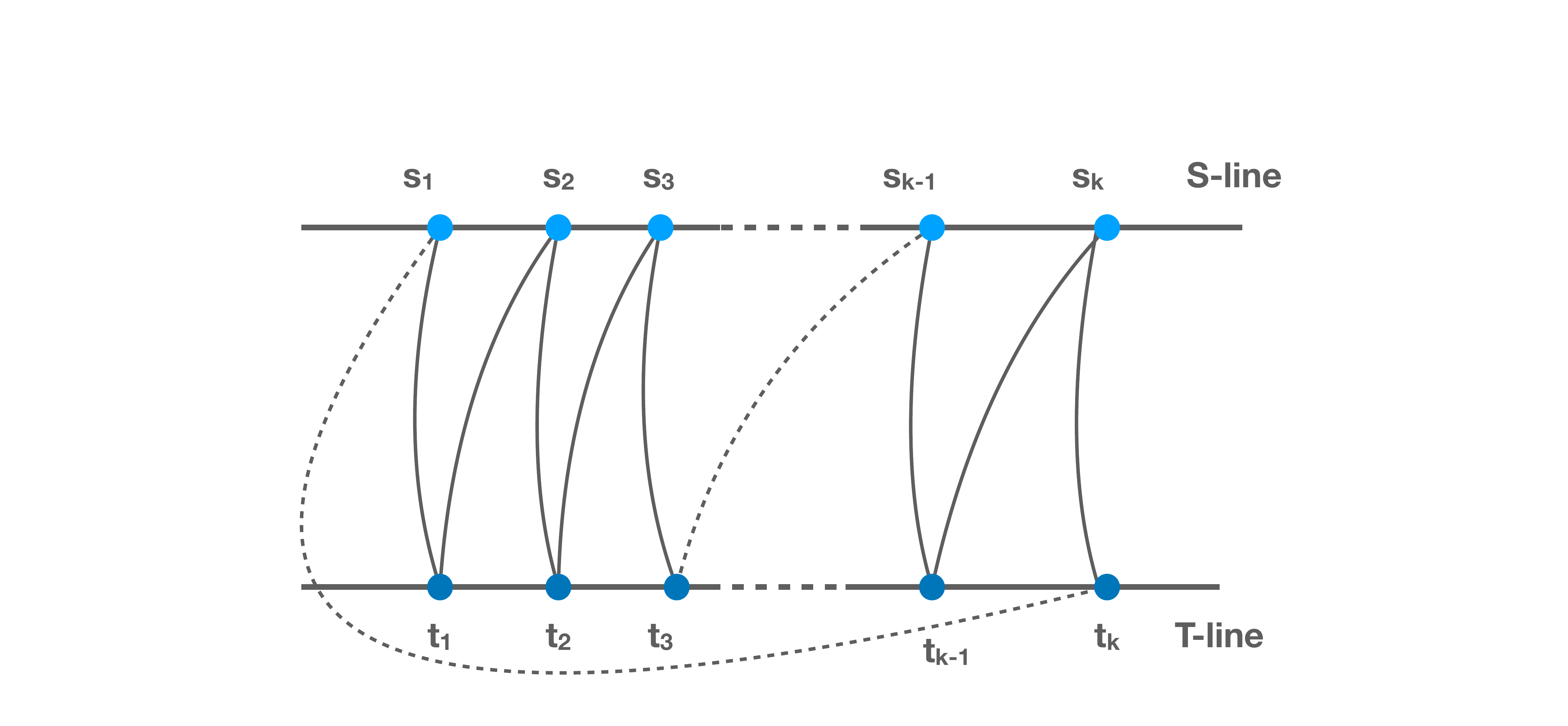}	
\caption{Eurlerian bipartite graph $\Gcal(\ubar{s},\ubar{t})$.}
\label{fig:Eulerian}
\end{figure}

Here, elements in the set $\{s_1,\ldots,s_k\}$ resp.\ $\{t_1,\ldots,t_k\}$ are called S-nodes resp.\ T-nodes. S and T-nodes are considered different even if their value is the same and are thus placed on separate lines -- called S-line and T-line -- which are drawn horizontally beneath each other. Then we draw an undirected edge $\{s_i,t_j\}$ between $s_i$ and $t_j$, $i\in\oneto{p}$, $j\in\oneto{n}$, whenever $(s_i,t_j)$ or $(t_j,s_i)$ appears in \eqref{eq:STwalk}, where we allow for multi-edges. 
This yields the (multi-)graph $\Gcal(\ubar{s},\ubar{t}) = (V(\ubar{s},\ubar{t}),E(\ubar{s},\ubar{t}),\phi_{\ubar{s},\ubar{t}})$, where
\begin{align*}
V(\ubar{s},\ubar{t}) \ &=\ \{s_1,\ldots,s_k\}\ \dot{\cup}\ \{t_1,\ldots,t_k\} \qquad \text{(disjoint union)}\\
E(\ubar{s},\ubar{t}) &= \{d_1,\ldots,d_k\}\ \dot{\cup}\ \{u_1,\ldots,u_k\}\qquad \text{(down edges, up edges)}	\\
&= \{e_1,e_2,\ldots, e_{2k}\} \qquad  (e_{2l-1} = d_l,\ e_{2l}=u_l,\ l=1,\ldots, k)\\
\phi_{\ubar{s},\ubar{t}}(d_i)&=\{s_i,t_i\}\\
\phi_{\ubar{s},\ubar{t}}(u_i)&=\{s_{i+1},t_i\}
\end{align*}
Each $(\ubar{s},\ubar{t})$ also denotes a Eulerian cycle of length $2k$ through its graph $\Gcal(\ubar{s},\ubar{t})$ by 
\begin{equation}
\label{eq:EuleriancycleMP}	
s_1,d_1,t_1,u_1,s_2,d_2,t_2,\ldots,u_{k-1},s_k,d_k,t_k,u_k,s_1
\end{equation}
Figure~\ref{fig:Eulerian} contains a visualisation of the graph $\Gcal(\ubar{s},\ubar{t})$. Note that by construction, $\Gcal(\ubar{s},\ubar{t})$ contains no loops, but may contain multi-edges.
 The language of graph theory allows us to express $\integrala{\mu_n}{x^k}$ in a different fashion. Recall
\begin{equation}
\label{eq:naivemoment}	
\integrala{\mu_n}{x^k}\ =\ \frac{1}{pn^k} \sum_{\ubar{s}\in\oneto{p}^k}\sum_{\ubar{t}\in\oneto{n}^k} X_n(\ubar{s},\ubar{t})
\end{equation}
with
\begin{equation}
\label{eq:exmomentsummandMP}
 X_n(\ubar{s},\ubar{t}) = X_n(s_1,t_1)X_n(s_2,t_1)X_n(s_2,t_2)X_n(s_3,t_2)\ldots X_n(s_k,t_k)X_n(s_1,t_k).
\end{equation}
For any pair of tuples $(\ubar{s},\ubar{t})\in \oneto{p}^k\times \oneto{n}^k$, we define its profile
\[
\rho(\ubar{s},\ubar{t}) = (\rho_1(\ubar{s},\ubar{t}),\ldots,\rho_{2k}(\ubar{s},\ubar{t})),
\]
where for all $\ell\in[2k]$:
\[
\rho_{\ell}(\ubar{s},\ubar{t}) = \#\{\phi_{\ubar{s},\ubar{t}}(e)\ |\ e\in E(\ubar{s},\ubar{t})\ \text{is an $\ell$-fold edge}\}.
\]
Here, an $\ell$-fold edge in $E(\ubar{s},\ubar{t})$ is any element $e\in E(\ubar{s},\ubar{t})$ for which there are exactly $\ell-1$ distinct other elements $e'_2,\ldots,e'_{\ell} \in E(\ubar{s},\ubar{t})$ so that $\phi_{\ubar{s},\ubar{t}}(e)=\phi_{\ubar{s},\ubar{t}}(e'_j)$ for $j\in\{2,\ldots,\ell\}$.

Then for all $\ell\in[2k]$, the Eulerian circuit $(\ubar{s},\ubar{t})$ traverses exactly $\phi_{\ell}(\ubar{s},\ubar{t})$ distinct $\ell$-fold edges. As a result, the following trivial but useful equality holds:
\begin{equation}
\label{eq:alledgessumMP}
2k = \sum_{\ell=1}^{2k}\ell\cdot\rho_{\ell}(\ubar{s},\ubar{t}).
\end{equation}
Now for all $k\in\{1,\ldots,2k\}$ we define the following set of profiles:
\[
\Pi(2k) = \left\{\rho \in\{0,\ldots,2k\}^{2k}\ |\ \rho \ \text{profile of some } (\ubar{s},\ubar{t})\in\oneto{p}^k\times\oneto{n}^k\right\}.
\]
Now we construct the finite decomposition
\begin{equation}
\label{eq:graphmomentMP}	
\integrala{\mu_n}{x^k} =\sum_{\rho\in\Pi(2k)}\frac{1}{pn^k}\sum_{(\ubar{s},\ubar{t})\in\Tcal^{p,n}(\rho)}  X_n(\ubar{s},\ubar{t}),
\end{equation}
where
\[
\Tcal^{p,n}(\rho) \defeq \left\{(\ubar{s},\ubar{t})\in\oneto{p}^k\times\oneto{n}^k\ |\ \rho(\ubar{s},\ubar{t})=\rho \right\}.
\]
The transition from \eqref{eq:naivemoment} to \eqref{eq:graphmomentMP} allows us to analyze the contribution of paths that match certain profiles, and to identify the profiles the paths of which contribute to the limit.

The next fundamental lemma will give an upper bound on the number of tuple pairs $(\ubar{s},\ubar{t})$ with at most $\ell\in\oneto{2k}$ vertices. Note that there are always at least two vertices present, since S-nodes and T-nodes are disjoint. Notationally, we set $V(\ubar{u})\defeq\{u_1,\ldots,u_k\}$ for any $\ubar{u}\in\N^k$ and $V(\ubar{u},\ubar{v})\defeq\{u_1,\ldots,u_k\}\,\dot{\cup}\,\{v_1,\ldots,v_k\}$ for any $\ubar{u}$, $\ubar{v}\in\N^k$, even if we do not view $(\ubar{u},\ubar{v})$ as a graph.

\begin{lemma}\label{lem:maxnodestuplesMP}
Let $p,n,k\in\N$, $a,b\in\{1,\ldots,k\}$ and $\ell\in\{2,3,\ldots,2k\}$ be arbitrary. Then
\begin{align*}
i)&\quad 
\# \left\{(\ubar{s},\ubar{t})\in\oneto{p}^k\times\oneto{n}^k\, |\, \#V(\ubar{s})=a, \#V(\ubar{t})=b\right\} \leq k^{2k}\cdot p^a n^b\\
ii)&\quad 
\# \left\{(\ubar{s},\ubar{t})\in\oneto{p}^k\times\oneto{n}^k\, |\, \#V(\ubar{s},\ubar{t})\leq \ell\right\} \leq k^{2k+2}\cdot (p\vee n)^{\ell}.
\end{align*}
\end{lemma}
\begin{proof}
For i) we first fix the colorings for $\ubar{s}$ with $a$ colors and $\ubar{t}$ with $b$ colors, for which be have at most $k^{2k}$ choices (Lemma~\ref{lem:coloring}). After fixing the colorings, we are left with at most $p^a$ choices for the tuple $\ubar{s}\in\oneto{p}^k$ and at most $n^b$ choices for the tuple $\ubar{t}\in\oneto{n}^k$, which yields the desired inequality. 
For $ii)$ we first decide on the number $a\leq k$ of different vertices in $\ubar{s}$ and the number $b\leq k$ of different vertices in $\ubar{t}$ such that $a+b\leq \ell$. This choice of $(a,b)$ admits at most $k^2$ choices. Then with $i)$, the statement follows.
\end{proof}

\subsection*{Step 1: Convergence of expected moments}
We proceed to analyze the expectation of
\begin{equation}
\label{eq:randmomMP}
\integrala{\mu_n}{x^k} =\sum_{\rho\in\Pi(2k)}\frac{1}{pn^k}\sum_{(\ubar{s},\ubar{t})\in\Tcal^{p,n}(\rho)} X_n(\ubar{s},\ubar{t}).
\end{equation}
To this end, it suffices to analyze the expectation of each of the finitely many terms
\begin{equation}
\label{eq:randomMPsummand}
\frac{1}{pn^k}\sum_{(\ubar{s},\ubar{t})\in\Tcal^{p,n}(\rho)} X_n(\ubar{s},\ubar{t})	
\end{equation}
for $\rho\in\Pi(2k)$ separately. As a first observation, note that if $\rho_1\geq 1$, we have $\E X_n(\ubar{s},\ubar{t}) = 0$ for all $(\ubar{s},\ubar{t})\in\Tcal^{p,n}(\rho)$ due to independence and centeredness. Further, since $(X_n)_n$ is an MP-scheme as in Definition~\ref{def:MPscheme}, we can always apply the trivial bound
\begin{equation}
\label{eq:trivialboundMP}
\abs{\E X_n(\ubar{s},\ubar{t})} \leq L_{2k},	
\end{equation}
where we also used Lemma~\ref{lem:holder}. 

For the bounds on $\#\Tcal^{p,n}(\rho)$ we formulate the next lemma. It is a modification of similar lemmas obtained in \autocite{FleermannDiss}.

\begin{lemma}\label{lem:nodeandtuplecountMP}
Let $k \in\N$ be arbitrary. Then it holds: 
\begin{enumerate}[i)]
\item $\#\Pi(2k)\leq 16^k.$
\item Let $p,n\in\N$ and $\rho\in\Pi(2k)$ be arbitrary, then
\begin{enumerate}[a)]
\item For any $(\ubar{s},\ubar{t})\in\Tcal^{p,n}(\rho)$ we obtain
\[
\# V(\ubar{s},\ubar{t})\leq 1 + \rho_1 + \ldots + \rho_{2k}.
\]
In particular,
\[
\#\Tcal^{p,n}(\rho)\leq k^{2k+2} \cdot (p\vee n)^{1 + \rho_1 + \ldots + \rho_{2k}}.
\]
\item If $\rho$ contains an odd edge, then for any $(\ubar{s},\ubar{t})\in\Tcal^{p,n}(\rho)$ we obtain
\[
\# V(\ubar{s},\ubar{t})\leq \rho_1 + \ldots + \rho_{2k}.
\]
In particular,
\[
\#\Tcal^{p,n}(\rho)\leq k^{2k+2} \cdot (p\vee n)^{\rho_1 + \ldots + \rho_{2k}}.
\]
\end{enumerate}
\end{enumerate}
\end{lemma}

\begin{proof}
\underline{i)}  Each $\rho\in\Pi(2k)$ is a $2k$-tuple in which for all $\ell\in\{1,\ldots,2k\}$ the entry $\rho_{\ell}$ lies in the set $\{0,1,\ldots,\lfloor 2k/\ell\rfloor\}$, which follows directly from \eqref{eq:alledgessumMP}. Therefore,
\[
\#\Pi(2k) \leq  \prod_{\ell=1}^{2k}\left(\frac{2k}{\ell}+1\right) = \frac{(4k)!}{(2k)!\cdot (2k)!} = \binom{2 (2k)}{2k}  \lesssim \frac{4^{2k}}{\sqrt{2k\pi}} \leq 16^k,
\]
where the fourth step is a well-known fact about the central binomial coefficient.\newline
\underline{ii)} It suffices to establish the upper bounds for $\#V(\ubar{s},\ubar{t})$, since the bounds on $\#\Tcal^{p,n}(\rho)$ then follow directly with Lemma~\ref{lem:maxnodestuplesMP} ii). Now to prove upper bounds for $\#V(\ubar{s},\ubar{t})$, the idea is to travel the Eulerian cycle generated by $(\ubar{s},\ubar{t}):$
\begin{equation}
\label{eq:walkMP}
s_1,e_1,t_1,e_2,s_2,e_3,t_2,\ldots,t_k, e_{2k}, s_1
\end{equation}
by picking an initial node $s_i$ or $t_i$ and then traversing the edges in increasing cyclic order until reaching the starting point again. On the way, we count the number of different nodes that were discovered. Whenever we pass an $\ell$-fold edge, only the first instance of that edge may discover a new vertex. \newline
\underline{a)}  We start our tour at  $s_1$ and observe this very vertex. Then, as we travel along the cycle, for each $\ell\in\{1,\ldots,2k\}$ we will pass $\ell\cdot \rho_{\ell}$ $\ell$-fold edges out of which only the first instance may discover a new node, and there are $\rho_{\ell}$ of these first instances. Considering the initial node, we arrive at $\# V(\ubar{s},\ubar{t})\leq 1 + \rho_1 + \ldots + \rho_{2k}$, which yields the desired inequality.\newline
\underline{b)} In presence of an odd edge, we can start the tour at a specific vertex such that the odd edge cannot contribute to the newly discovered vertices. To this end, fix an $\ell$-fold edge in  $(\ubar{s},\ubar{t})$ with $\ell$ odd. Let $e_{i_1},\ldots,e_{i_{\ell}}$, $i_1 < \ldots < i_{\ell}$, be the instances of the $\ell$-fold edge in question in the cycle \eqref{eq:walkMP}. Since $\ell$ is odd, we must find a $k\in\{1,\ldots,\ell\}$ such that $e_{i_k}$ and $e_{i_{k+1}}$ are both up edges or both down edges (where $\ell+1\equiv 1$), since we are on a cycle. W.l.o.g.\ $e_{i_k}$ is a down edge, thus leading to $t_{i_k}$. We start our tour at $t_{i_k}$ and observe this vertex. However, now none of the edges $e_{i_1},\ldots,e_{i_{\ell}}$ may discover a new vertex, since the vertex $s_{i_{k+1}}$ must be discovered by some other edge. Therefore, the roundtrip leads to the discovery of at most $\rho_1 + \dots + (\rho_{\ell} - 1) + \ldots + \rho_{2k}$ new nodes in addition to the first node.
\end{proof}
We proceed to analyze \eqref{eq:randmomMP} for all possible types of $\rho\in\Pi(2k)$:

\noindent\underline{Case 1: $\rho_1=0$ and $\rho_{\ell}>0$ for some $\ell\geq 3$.}\newline
Using Lemma~\ref{lem:nodeandtuplecountMP} we obtain
\[
\#\Tcal^{p,n}(\rho)\leq k^{2k+2} \cdot (p\vee n)^{1 + \rho_1 + \ldots + \rho_{2k}}\leq k^{2k+2} (p\vee n)^k,
\]
since with $\rho_{\ell}>0$ for some $\ell\geq 3$ it follows
\[
1 + \rho_1 + \ldots + \rho_{2k} \leq  
\left\{
\begin{array}{c}
 1 + \frac{2k-6}{2} + 2\\
 1 + \frac{2k-4}{2} + 1
\end{array}
\right\} = k,
\]
where the upper case is valid in presence of an odd edge (so we find at least a second odd edge), and the lower case is valid if no odd edges are present. Therefore, by \eqref{eq:trivialboundMP}, \eqref{eq:randomMPsummand} converges to zero in expectation.

\noindent\underline{Case 2: $\rho_1 >0$.}\newline
Then by centeredness and independence, the expectation in \eqref{eq:randomMPsummand} is zero. 

\noindent\underline{Case 3: $\rho_{2}=k/2$.}\newline
Returning to the random moment in \eqref{eq:randmomMP}, we have seen in Cases 1 and 2 that for all $\rho\in\Pi(k)$ with $\rho_{2}\neq k$,
\[
\frac{1}{pn^k}\sum_{(\ubar{s},\ubar{t})\in\Tcal^{p,n}(\rho)} X_n(\ubar{s},\ubar{t}) \xrightarrow[n\to\infty]{} 0 \qquad \text{in expectation.}
\]
As a result, the only asymptotic contribution in \eqref{eq:randmomWIG} may stem from cycles $(\ubar{s},\ubar{t})$ containing only double edges. Their analysis is the content of this Case 3. Setting $\rho^{(k)}$ as the profile in $\Pi(2k)$ with $\rho^{(k)}_2=k$ and $\rho^{(k)}_{\ell}=0$ for all $\ell\neq 2$, then it is our goal to show (cf.\ \eqref{eq:momentsMP})
\begin{equation}
\label{eq:goalCase3MP}
\frac{1}{pn^k}\sum_{(\ubar{s},\ubar{t})\in\Tcal^{p,n}(\rho^{(k)})} X_n(\ubar{s},\ubar{t})\ \xrightarrow[n\to\infty]{}\ \sum_{r=0}^{k-1} \frac{y^r}{r+1}\binom{k}{r}\binom{k-1}{r} \qquad \text{in expectation.}
\end{equation}
To this end, we observe

\begin{equation}
\label{eq:summandCase3MP}
\frac{1}{pn^k}\sum_{(\ubar{s},\ubar{t})\in\Tcal^{p,n}(\rho^{(k)})} \E X_n(\ubar{s},\ubar{t}) = \frac{1}{pn^k}\#\Tcal^{p,n}(\rho^{(k)}).
\end{equation}
We note that any $(\ubar{s},\ubar{t})\in\Tcal^{p,n}(\rho^{(k)})$ has at most $k+1$ vertices, so we may subdivide this set further: We define
\begin{align*}
\Tcal^{p,n}_{\leq k}(\rho^{(k)}) &\defeq \left\{(\ubar{s},\ubar{t})\in\Tcal^{p,n}(\rho^{(k)}):\ \# V(\ubar{s},\ubar{t}) \leq k \right\},\\
\Tcal^{p,n}_{k+1}(\rho^{(k)}) &\defeq \left\{(\ubar{s},\ubar{t})\in\Tcal^{p,n}(\rho^{(k)}):\ \# V(\ubar{s},\ubar{t}) = k+1 \right\}.
\end{align*}
and note that by Lemma~\ref{lem:maxnodestuplesMP}, $\#\Tcal^{p,n}_{\leq k}(\rho^{(k)}) \leq k^{2k+2}(p\vee n)^{k}$, so that \eqref{eq:summandCase3MP} can be refined to
\begin{equation}
\label{eq:onlydoublekplusoneMP}
\frac{1}{pn^k}\sum_{(\ubar{s},\ubar{t})\in\Tcal^{p,n}(\rho^{(k)})} \E X_n(\ubar{s},\ubar{t})  = \frac{1}{pn^k}\#\Tcal^{p,n}_{k+1}(\rho^{(k)})  \ + \ o(1), 
\end{equation}

It is thus our task to show
\begin{equation}
\label{eq:convergetoMPmoment}	
\frac{1}{pn^k}\#\Tcal^{p,n}_{k+1}(\rho^{(k)})\ \xrightarrow[n\to\infty]{}\ \sum_{r=0}^{k-1} \frac{y^r}{r+1}\binom{k}{r}\binom{k-1}{r}.
\end{equation}

To this end, for all $(\ubar{s},\ubar{t})\in\Tcal_{k+1}^{p,n}(\rho^{(k)})$ we track the number of vertices in $\ubar{s}$ and the number vertices in $\ubar{t}$ that the cycle visits. Thus, for all $a,b\in\N$ with $a+b= k+1$ we define
\[
\Tcal^{p,n}_{a,b}(\rho^{(k)}) \defeq \left\{(\ubar{s},\ubar{t})\in\oneto{p}^k\times\oneto{n}^k\ |\ \rho(\ubar{s},\ubar{t})=\rho^{(k)}, \# V(\ubar{s})=a, \# V(\ubar{t})=b \right\}.
\]
Then we obtain a \emph{partition}
\[
\Tcal^{p,n}_{k+1}(\rho^{(k)}) = \bigcup_{r=0}^{k-1}  \Tcal^{p,n}_{r+1,k-r}(\rho^{(k)}).
\]
As a result, to show \eqref{eq:convergetoMPmoment} it suffices to show that for all $r\in\{0,\ldots,k-1\}$,
\begin{equation}
\label{eq:convergetoMPmomentFinal}	
\frac{1}{pn^k}\#\Tcal^{p,n}_{r+1,k-r}(\rho^{(k)})\ \xrightarrow[n\to\infty]{}\  \frac{y^r}{r+1}\binom{k}{r}\binom{k-1}{r}.
\end{equation}

It remains to evaluate $\# \Tcal^{p,n}_{r+1,k-r}(\rho^{(k)})$ for all $r\in\{0,\ldots, k-1\}$. This is done by identifying the number of different color structures that an $(\ubar{s},\ubar{t})\in\Tcal^{p,n}_{r+1,k-r}(\rho^{(k)})$ may assume and then by multiplying this number with the number of possible colorings, which is a trivial task. The main tool to count all possible color structures is to associate with each $(\ubar{s},\ubar{t})\in\Tcal^{p,n}_{r+1,k-r}(\rho^{(k)})$ a path difference sequence (pds) of the following form:

\begin{definition}
\label{def:MP-path}
A Marcenko-Pastur path difference sequence (MP-pds) of length $2k$ and weight $r\in\{0,\ldots, k-1\}$ is a tuple $(D_1,U_1,D_2,U_2,\ldots,D_k,U_k)=(M_1,\ldots,M_{2k})$ which satifies the following conditions:
\begin{enumerate}[1)]
\item $D_i\in\{-1,0\}$ and $U_i\in\{0,1\}$.
\item $\sum_{i\in\oneto{k}}U_i = r$  and  $\sum_{i\in\oneto{k}}D_i = -r$.
\item $\forall\,\ell\in\{1,\ldots, 2k\}: \sum_{i=1}^{\ell}M_i\geq 0$.
\end{enumerate}
We denote by $\Mcal(k,r)$ the set of all MP-pds of length $2k$ and weight $r$.	
\end{definition}

\begin{lemma}
\label{lem:colorstructuresMP}
For all $k\in\N$ and $r\in\{0,\ldots,k-1\}$ we find $\#\Mcal(k,r) = \frac{1}{r+1}\binom{k-1}{r}\binom{k}{r}$	.
\end{lemma}
\begin{proof}
We assume $r\geq 1$ since for $r=0$ the statement is clear. We prove the lemma with a reflection principle. First note that $M_1=D_1=0$ and $M_{2k}=U_k=0$ so that we are interested in all sequences $(M_2,\ldots,M_{2k-1})$ where
\begin{enumerate}[1)]
\item $M_i\in\{-1,0\}$ for $i$ odd and $M_i\in\{0,1\}$ for $i$ even.
\item $\sum_{i \text{ odd}}M_i = -r$  and  $\sum_{i \text{ even}}D_i = r$.
\item $\forall\,\ell\in\{2,\ldots, 2k-1\}: \sum_{i=2}^{\ell}M_i\geq 0$.
\end{enumerate}
To this end, we have 
\[
\binom{k-1}{r}\cdot\binom{k-1}{r}
\]
choices to allocate $r$ "$+1$"s to $k-1$ places $r$ "$-1$"s to $k-1$ places. But since these choices do not in general respect condition $3)$ we have to subtract the number of tuples $(M_2,\ldots,M_{2k-1})$ that lead to a violation of $3)$. We show that these violating tuples are in bijective correspondence to all $(M_2',\ldots,M_{2k-1}')$ with
\begin{enumerate}[1')]
\item $M'_i\in\{-1,0\}$ for $i$ odd and $M_i'\in\{0,1\}$ for $i$ even.
\item $\sum_{i \text{ odd}}M_i' = -(r+1)$  and  $\sum_{i \text{ even}}M_i' = r-1$.
\end{enumerate}
The number of these $(M_2',\ldots,M_{2k-1}')$ is clearly given by 
\[
\binom{k-1}{r+1}\cdot\binom{k-1}{r-1}
\]
so that the number of $(M_2,\ldots,M_{2k-1})$ that \emph{do} satisfy 1), 2) \emph{and} 3) is given by
\[
\binom{k-1}{r}\cdot\binom{k-1}{r} - \binom{k-1}{r+1}\cdot\binom{k-1}{r-1} = \frac{1}{r+1}\binom{k-1}{r}\binom{k}{r}
\]
For the bijection, let $(M_2,\ldots,M_{2k-1})$ be arbitrary with $r$ "$+1$"s and $r$ "$-1$"s so that 3) is violated. Then there is an odd index $t$ such that $\sum_{i=2}^t M_i = -1$ for the first time. Then $(M_{t+1},\ldots,M_{2k-1})$
 is a vector of even length which contains one more "$+1$" than "$-1$" entry. We will transform this vector to a vector  $(M'_{t+1},\ldots,M'_{2k-1})$ by transforming  the pairs $(M_{t+1},M_{t+2}),\ldots,(M_{2k-2},M_{2k-1})$ as follows: If the pair is $(+1,-1)$ or $(0,0)$, we leave it unchanged. A pair $(1,0)$ will be changed to $(0,-1)$ and a pair $(0,-1)$ will be changed to $(1,0)$. Then $(M'_{t+1},\ldots,M'_{2k-1})$ contains one more "$-1$" than "$+1$". Defining $(M_2',\ldots,M_t')\defeq(M_2,\ldots,M_t)$ we thus have created a vector $(M_2',\ldots,M_{2k-1}')$ satisfying 1') and 2'). On the other hand, any vector $(M_2',\ldots,M_{2k-1}')$ satisfying 1) and 2) has a first hitting time $t$ of $-1$. Applying exactly the same transformation as before, we will then obtain a vector $(M_2,\ldots,M_{2k-1})$ satisfying 1) and 2), but violating 3).
 \end{proof}

Now to each $(\ubar{s},\ubar{t})\in\Tcal^{p,n}_{r+1,k-r}(\rho^{(k)})$ we can associate an $M\in\Mcal(k,r)$, and this association completely determines the color structure of $(\ubar{s},\ubar{t})$. To see how this is done, let us first analyze simple properties of a Eulerian cycle $(\ubar{s},\ubar{t})\in\Tcal^{p,n}_{r+1,k-r}(\rho^{(k)})$. First, the graph $\Gcal(\ubar{s},\ubar{t})$ is a \emph{double edged tree}, that is, it consists of $k$ distinct double edges and has $k+1$ vertices, therefore is a tree in the regular sense after eliminating one of each of the double edges. Thus, the Eulerian cycle $(\ubar{s},\ubar{t})$ crosses each edge twice, once in each direction, since a tree does not have circles. Further, $(\ubar{s},\ubar{t})$ starts at the S-vertex $s_1$ and then alternates between S- and T-vertices until reaching $s_1$ again. We recall the representation of the cycle as in \eqref{eq:EuleriancycleMP}. Now we will record two crucial pieces of information into the MP-pds $M$. We start a tour at $s_1$ and move along the cycle. Whenever a down edge $d_{\ell}$ leaves the S-vertex $s_{\ell}$ for the last time along the walk, we set $D_{\ell}=-1$, otherwise $D_{\ell}=0$. For example, we always have $D_{1}=0$, since $s_1$ is the last stop of the cycle. Additionally, whenever an up edge $u_{\ell}$ visits a new S-vertex $s_{\ell+1}$, which has not been visited before, we set $U_{\ell}=1$ and otherwise $U_{\ell}=0$. For example, we will always have $U_k=0$, since this edge leads to the starting point $s_1$ again. 

Let us argue that the tuple $(D_1,U_1,D_2,\ldots,U_k)$ we just constructed satisfies conditions 1), 2) and 3) as above. Condition 1) is clearly satisfied. For condition 2), note that $(\ubar{s},\ubar{t})$ has $r+1$ S-nodes, out of which $r$ -- all except the vertex $s_1$ -- were considered new, so that $\sum U_i =r$. Since the last edge $u_k$ leads back to the vertex $s_1$, we must have left each of the $r$ new S-vertices for a last time while on the cycle, so $\sum D_i = -r$. For condition 3) we realize only the $r$ new nodes are left for a last time along the cycle, and before they can be left a last time (leading to a summand $-1$) they must have been discovered (leading to a summand $+1$). Thus, condition 3) holds. 

Let us now see that each MP-pds $M\in\Mcal(k,r)$ completely determines the color structure of an $(\ubar{s},\ubar{t})\in\Tcal^{p,n}_{r+1,k-r}(\rho^{(k)})$ by constructing a canonical $(\ubar{s},\ubar{t})$ (that is, one with lowest vertex numbers possible) from $M$, and showing that we have only one choice for this construction. We set $s_1=1=t_1$. Then whenever $U_{\ell}=+1$, this means that a new S-node is discovered, so we set $s_{\ell +1}\defeq\max(s_1,\ldots,s_{\ell})+1$. On the other hand, if $U_{\ell}=0$ then this means that $s_{\ell+1}$ shall be equal to one of the $s_1,\ldots,s_{\ell}$, and so it must be equal to the $s_i$ with $i\in\{1,\ldots,\ell\}$ maximal from which $t_{\ell}$ was visited, since otherwise, the cycle $(\ubar{s},\ubar{t})$ would contain a circle. 

Now for $\ell\geq 2$, whenever $D_{\ell}=0$, this means that $s_{\ell}$ is \emph{not} visited the last time. But then $t_{\ell}$ must be different from $t_1,\ldots,t_{\ell-1}$ since otherwise the cycle $(\ubar{s},\ubar{t})$ would contain a circle. Therefore, for $\ell\geq 2$, if $D_{\ell}=0$ we set $t_{\ell}\defeq\max(t_1,\ldots,t_{\ell-1}) + 1$. Otherwise, if $D_{\ell}=-1$, this means that $s_{\ell}$ was visited for the last time by the cycle. But then $t_{\ell}$ must be equal to some element in $\{t_1,\ldots,t_{\ell-1}\}$, since if $t_{\ell}$ were new, the edge $\{s_{\ell},t_{\ell}\}$ would be new and there would then have to be a second edge traveling back from $t_{\ell}$ to $s_{\ell}$, which would entail yet another visit of $s_{\ell}$. So if $D_{\ell}=-1$, $t_{\ell}$ must be equal to some vertex in $\{t_1,\ldots,t_{\ell-1}\}$, and then it must be equal to the last vertex with the highest index number in the set, from which $s_{\ell}$ was visited, since otherwise, again, the cycle $\{s_{\ell},t_{\ell}\}$ would contain a circle. 

As we saw, an $(\ubar{s},\ubar{t})\in\Tcal^{p,n}_{r+1,k-r}(\rho^{(k)})$ is compatible with exactly one $M\in\Mcal(k,r)$, and we then write $(\ubar{s},\ubar{t})\sim M$. On the other hand, given an $M\in\Mcal(k,r)$ we could create exactly one canonical $(\ubar{s}^*,\ubar{t}^*)\in\Tcal^{p,n}_{r+1,k-r}(\rho^{(k)})$ compatible with $M$, which determines the color structure. All other $(\ubar{s},\ubar{t})$ compatible with $M$ are then obtained by picking different vertex names for the $r+1$ vertices in $\ubar{s}$ and $k-r$ vertices in $\ubar{t}$, which yields a total of $(p)_{r+1}\cdot (n)_{k-r}$ tuples in $\Tcal^{p,n}_{r+1,k-r}(\rho^{(k)})$ compatible with each $M\in\Mcal(k,r)$, where for any $\ell\leq m\in\N$, we set $(m)_{\ell}\defeq m\cdot (m-1)\cdots (m-\ell+1)$. This analysis yields the following lemma:

\begin{lemma}
\label{lem:TnMP}
The set $\Tcal^{p,n}_{r+1,k-r}(\rho^{(k)})$ has a decomposition as follows:
\begin{equation}
\label{eq:partitionstructureMP}
\Tcal^{p,n}_{r+1,k-r}(\rho^{(k)}) = \dot{\bigcup_{M\in\Mcal(k,r)}}\left\{(\ubar{s},\ubar{t})\in \Tcal^{p,n}_{r+1,k-r} \ | \ (\ubar{s},\ubar{t})\sim M\right\}
\end{equation}
Further, for all $M\in\Mcal(k,r)$,
\begin{equation}
\label{eq:compatiblecount}	
\#\left\{(\ubar{s},\ubar{t})\in \Tcal^{p,n}_{r+1,k-r} \ | \ (\ubar{s},\ubar{t})\sim M\right\} = (p)_{r+1}(n)_{k-r}\ ,
\end{equation}
such that by Lemma~\ref{lem:colorstructuresMP}, \eqref{eq:partitionstructureMP} and \eqref{eq:compatiblecount}, we obtain 
\[
\#\Tcal^{p,n}_{r+1,k-r}(\rho^{(k)}) = \frac{1}{r+1}\binom{k-1}{r}\binom{k}{r}\cdot(p)_{r+1}(n)_{k-r}\ .
\]
\end{lemma}
\begin{proof}
See the discussion before Lemma~\ref{lem:TnMP}.	
\end{proof}
Now since
\[
(p)_{r+1}(n)_{k-r} = p\cdot\underbrace{\frac{(p-1)_{r}}{n^{r}}}_{\to y^r}\cdot \underbrace{n^r\cdot (n)_{k-r}}_{\sim n^k}\ ,
\]
we find by Lemma~\ref{lem:TnMP} that
\[
\frac{1}{pn^k}\#\Tcal^{p,n}_{r+1,k-r}(\rho^{(k)})\ \xrightarrow[n\to\infty]{}\  \frac{y^r}{r+1}\binom{k}{r}\binom{k-1}{r},
\]
which is \eqref{eq:convergetoMPmomentFinal}. Therefore, we have shown \eqref{eq:convergetoMPmoment}	 which entails \eqref{eq:goalCase3MP}.

\subsection*{Step 2: Decay of central moments}

In Step 1 we have seen that for fixed $k\in\N$, the expectation of
\begin{equation}
\label{eq:randmomMPVAR}
\integrala{\mu_n}{x^k} =\sum_{\rho\in\Pi(2k)}\frac{1}{pn^k}\sum_{(\ubar{s},\ubar{t})\in\Tcal^{p,n}(\rho)} X_n(\ubar{s},\ubar{t})
\end{equation}
converges to the $k$-th moment of the MP distribution. In particular, we have seen that each of the finitely many summands
\begin{equation}
\label{eq:randomMPsummandVAR}
\frac{1}{pn^k}\sum_{(\ubar{s},\ubar{t})\in\Tcal^{p,n}(\rho)} X_n(\ubar{s},\ubar{t})	
\end{equation}
converges to a constant in expectation. To show that the random moments in \eqref{eq:randmomMPVAR} converge almost surely to the moments of the MP distribution, it thus suffices -- by Lemma~\ref{lem:convergencemodes} -- to show that for all $\rho\in\Pi(2k)$, the variance of \eqref{eq:randomMPsummandVAR} decays summably fast. The variance of \eqref{eq:randomMPsummandVAR} is given by
\begin{equation}
\label{eq:MPvariance}
\frac{1}{p^2n^{2k}}\sum_{(\ubar{s},\ubar{t}),(\ubar{s}',\ubar{t}')\in\Tcal^{p,n}(\rho)} \left[\E X_n(\ubar{s},\ubar{t})X_n(\ubar{s}',\ubar{t}')- \E X_n(\ubar{s},\ubar{t})\E X_n(\ubar{s}',\ubar{t}')\right].
\end{equation}

We see that whenever the Eulerian cycles $(\ubar{s},\ubar{t})$ and $(\ubar{s}',\ubar{t}')$ are edge-disjoint, the term in \eqref{eq:MPvariance} vanishes due to independence. Therefore, it suffices to consider those cycles $(\ubar{s},\ubar{t})$ and $(\ubar{s}',\ubar{t}')$ which have at least one edge in common. To this end, denote for all $\ell\in\{1,\ldots,2k\}$: 
\begin{multline*}
\Tcal_{c(\ell)}^{p,n}(\rho)\defeq \left\{((\ubar{s},\ubar{t}),(\ubar{s}',\ubar{t}')) \in (\Tcal^{p,n}(\rho))^2\,|\,\right. \\ \left.\text{$(\ubar{s},\ubar{t})$ and $(\ubar{s}',\ubar{t}')$ have exactly $\ell$ edges in common}\right\}.
\end{multline*}
Then it is now our goal to show that for each $\rho\in\Pi(2k)$,
\begin{equation}
\label{eq:MPvarianceSUF}
\frac{1}{p^2n^{2k}}\sum_{((\ubar{s},\ubar{t}),(\ubar{s}',\ubar{t}'))\in\Tcal_{c(\ell)}^{p,n}(\rho)} \left[\E X_n(\ubar{s},\ubar{t})X_n(\ubar{s}',\ubar{t}')- \E X_n(\ubar{s},\ubar{t})\E X_n(\ubar{s}',\ubar{t}')\right]
\end{equation}
converges to zero summably fast.
Before proceeding, we need to establish bounds on $\#\Tcal_{c(\ell)}^{p,n}(\rho)$.

\begin{lemma}
\label{lem:doubleboundMP}
Let $\rho\in\Pi(2k)$ and $\ell\in\oneto{2k}$, then the following statements hold:
\begin{enumerate}[i)]
\item For all $(\ubar{s},\ubar{t}),(\ubar{s}',\ubar{t}') \in\Tcal^{p,n}(\rho)$ with at least $\ell$ common edges, it holds
\[
\#(V(\ubar{s},\ubar{t})\cup V(\ubar{s}',\ubar{t}')) \leq 1 + 2\sum_{i=1}^{2k} \rho_i - \ell
\]
In particular,
\[
\#\Tcal_{c(\ell)}^{p,n}(\rho,\rho') \leq (2k)^{4k+2} (n\vee p)^{1 + 2 \sum_{i=1}^{2k} \rho_i  - \ell}
\]
\item If there is an $\ell\in\oneto{2k}$ odd with $\rho_{\ell}\geq 1$, then for all $(\ubar{s},\ubar{t}),(\ubar{s}',\ubar{t}')\in\Tcal^{p,n}(\rho)$ with at least $\ell$ common edges, it holds
\[
\#(V(\ubar{s},\ubar{t})\cup V(\ubar{s}',\ubar{t}')) \leq 2\sum_{i=1}^{2k} \rho_i - \ell.
\]
In particular,
\[
\#\Tcal_{c(\ell)}^{p,n}(\rho,\rho') \leq (2k)^{4k+2} (n\vee p)^{2\sum_{i=1}^{2k} \rho_i - \ell}.
\]
\end{enumerate}
\end{lemma}
\begin{proof}
For statement $ii)$ we assume w.l.o.g.\ that $(\ubar{s},\ubar{t})$ has an odd edge. Since the graphs spanned by $(\ubar{s},\ubar{t})$ and $(\ubar{s}',\ubar{t}')$ share $l\geq 1$ common edges, we may take a tour around the joint Eulerian cycle, starting before a common edge, traveling first all edges of $(\ubar{s},\ubar{t})$ and then all edges of $(\ubar{s}',\ubar{t}')$. While walking the edges of $(\ubar{s},\ubar{t})$, we can see at most $\rho_1 + \ldots + \rho_{2k}$ different nodes by Lemma~\ref{lem:nodeandtuplecountMP}. Next, traveling all edges of $(\ubar{s}',\ubar{t}')$, at most all the single edges and first instances of $m$-fold edges with $m\in\{2,\ldots,2k\}$ of $(\ubar{s}',\ubar{t}')$ may discover a new node, but only if they have not been traversed before during the walk along $(\ubar{s},\ubar{t})$. Since we have $\ell$ common edges, we can see at most $\rho'_1+\ldots + \rho'_{2k}-\ell$ new nodes. We established the bounds on the number of vertices in $ii)$. The second statement in $ii)$ follows immediately with Lemma~\ref{lem:maxnodestuplesMP} $ii)$ by concatenating $(\ubar{s},\ubar{s}')\in\oneto{p}^{2k}$ and $(\ubar{t},\ubar{t}')\in\oneto{n}^{2k}$. For statement $i)$ we proceed exactly in the same manner: Traveling $(\ubar{s},\ubar{t})$ we can see at most $1 + \rho_1 + \rho_2 + \ldots + \rho_{2k}$ nodes by Lemma~\ref{lem:nodeandtuplecountMP}, then traveling $(\ubar{s}',\ubar{t}')$ we can see at most $\rho'_1+\ldots + \rho'_{2k}-\ell$ new nodes. Now apply Lemma~\ref{lem:maxnodestuplesMP} $ii)$ again.
\end{proof}
Returning to \eqref{eq:MPvarianceSUF}, we distinguish the following cases:

\noindent\underline{Case 1: $\rho_1\geq 1$}\newline
In this case, the term in \eqref{eq:MPvarianceSUF} simplifies and we must argue that for each $\ell\in\oneto{2k}$,
\begin{equation}
\label{eq:tozeroasMP}
\frac{1}{p^2n^{2k}}\sum_{((\ubar{s},\ubar{t}),(\ubar{s}',\ubar{t}'))\in\Tcal_{c(\ell)}^{p,n}(\rho)} \E X_n(\ubar{s},\ubar{t})X_n(\ubar{s}',\ubar{t}')
\end{equation}
converges to zero summably fast. If $(\ubar{s},\ubar{t})$ and $(\ubar{s}',\ubar{t}')$ have $1\leq\ell<\rho_1$ common edges, $\E X_n(\ubar{s},\ubar{t})X_n(\ubar{s}',\ubar{t}')$ vanishes, since not all single edges can be eliminated due to overlapping. Thus, it suffices to consider those $(\ubar{s},\ubar{t}),(\ubar{s}',\ubar{t}')\in\Tcal^{p,n}(\rho)$ which have $\ell\geq \rho_1$ edges in common.
Then
\[
2\sum_{i=1}^{2k} \rho_i - \ell \leq 2\left(\rho_1 + \frac{2k-\rho_1}{2}\right) -\rho_1  \leq 2k.  
\]
so Lemma~\ref{lem:doubleboundMP} $ii)$ yields
\[
\#\Tcal_{c(\ell)}^{p,n}(\rho) \leq (2k)^{4k+2} (n\vee p)^{2\sum_{i=1}^{2k} \rho_i - \ell}\leq (2k)^{4k+2} (n\vee p)^{2k},
\]
Since every summand in \eqref{eq:tozeroasMP} is bounded by $L_{4k}$, it follows that \eqref{eq:tozeroasMP} is $O(n^{-2})$, thus converges to zero summably fast.

\noindent\underline{Case 2: $\rho_1 = 0$}\newline
In this case, each summand in \eqref{eq:MPvarianceSUF} is bounded by $L_{4k} + L_{2k}^2$. Further, we obtain for all $\rho\in\Pi(2k)$ with $\rho_1 = 0$ and $\ell\geq 1$ that
\[
1 + 2\sum_{i=1}^{2k} \rho_i - \ell \leq 1 + 2\cdot\frac{2k}{2} - 1 = 2k
\]

so that by Lemma~\ref{lem:doubleboundMP} $i)$, 
\[
\#\Tcal_{c(\ell)}^{p,n}(\rho) \leq (2k)^{4k+2} (n\vee p)^{1 + 2\sum_{i=1}^{2k} \rho_i - \ell} \leq (2k)^{4k+2} (n\vee p)^{2k},
\]
so that the sum in \eqref{eq:MPvarianceSUF} is $O(n^{-2})$, hence converges to zero summably fast.

\chapter{The Stieltjes Transform Method}
\label{chp:stieltjesprops}
\section{Motivation and Basic Properties}

In order to analyze properties of random variables and their distributions, it is a common technique to use transforms of these distributions which make analysis more accessible due to their favorable algebraic structure. For example, a common and short proof of the central limit theorem is conducted by using the Fourier transform of the random variables involved, owing to the property that Fourier transforms handle convolutions particularly well, and the central limit theorem is about a sum of random variables.

In random matrix theory, however, when analyzing empirical spectral distributions of diverse matrix ensembles, it is desirable to use a tool for analysis that relates the behavior of the empirical spectral distribution back to the level of the entries of the matrices. For example, using the method of moments, one sees in equation \eqref{eq:momenttosum} that the moments of the ESD $\sigma_n$ of a random matrix $X_n$ can be calculated through:
\[
\forall\,k\in\N:~\integrala{\sigma_n}{x^k} =\frac{1}{n}\tr(X_n^k)=\frac{1}{n} \sum_{i_1,\ldots,i_k=1}^n X_n(i_1,i_2)X_n(i_2,i_3)\cdots X_n(i_k,i_1).
\]
In other words, instead trying to work with an ESD directly, we can analyze its moments which allows us to work on the level of the matrix entries.

A tool that combines both worlds, that is, that provides the structure of a transform with favorable algebraic properties and that allows us to work on the level of the matrix entries is the so called Stieltjes transform: 

\begin{definition}
\label{def:stieltjestransform}
Let $\mu$ be a finite measure on $(\R,\Bcal)$. Then we define the Stieltjes transform $S_{\mu}$\label{sym:Stieltjestrans} of $\mu$ as the map
\map{S_{\mu}}{\C\backslash\R}{\C}{z}{\int_{\R} \frac{1}{x-z} \mu(\de{x})}.\label{sym:complexnumbers}
\end{definition}

We note that the Stieltjes transform is defined via a measure-theoretical integral over a complex-valued function. We assume the reader to be acquainted with measure-theoretical integration of real-valued functions on measure spaces and give a very short introduction to complex-valued integration in the form of one definition and two lemmata.

\begin{definition}
Let $(\Omega,\Acal,\mu)$ be a measure space, $f:(\Omega,\Acal)\to\C$ measurable, then $f$ is called $\mu$-\emph{integrable}, if the real-valued functions 
$\Re f$ and $\Im f$ both are $\mu$- integrable. In this case, we define 
\[
\int_\Omega f\de{\mu} \defeq \int_\Omega \Re{f}\de{\mu} + i \int_\Omega \Im{f}\de{\mu}.
\]
We will denote the space of $\C$-valued integrable functions as $\mc{L}_1(\mu,\C)$.
\end{definition}

It is worth noting the following lemma about the properties of the integral:
\begin{lemma}
Let $(\Omega,\Acal,\mu)$ be a measure space, then the following statements hold:
\begin{enumerate}
	\item The map $\mc{L}_1(\mu,\C)\to \C$, $f\mapsto \int f \de{\mu}$ is $\C$-linear.
	\item $\forall\, f\in \mc{L}_1(\mu,\C): \overline{\int f \de{\mu}} = \int \overline{f} \de{\mu}$.
	\item $\forall\, f\in \mc{L}_1(\mu,\C): \bigabs{\int f \de{\mu}} \leq \int \abs{f} \de{\mu}$.
\end{enumerate}
\end{lemma}
\begin{proof}
1) follows by elementary calculations and 2) holds by the definition of the integral. To see 3), let $z\in\C$ with $\abs{z}=1$, such that $z  \int f \de{\mu} = \bigabs{\int f \de{\mu}}$, then it follows 
\[
\bigabs{\int f \de{\mu}} = z  \int f \de{\mu} = \int{\Re(zf)}\de{\mu} + i \underbrace{\int\Im(zf)\de{\mu}}_{=0} \leq \int \abs{zf}\de\mu = \int \abs{f}\de\mu.
\]	
\end{proof}

\begin{lemma}[Lebesgue's Dominated Convergence Theorem]
Let $(\Omega,\Acal,\mu)$ be a measure space, $(f_n)_n, f: \Omega\to\C$ be measurable with $f_n\to f$ $\mu$-almost everywhere. 
If there exists a $\mu$-integrable $g:\Omega\to\R$ with $\abs{f_n}\leq g$ $\mu$-almost everywhere for all $n$, then $f$ is $\mu$-integrable and it holds
\[
\lim_{n\to\infty} \int \abs{f-f_n}\de\mu = 0,
\]
so that in particular
\[
\lim_{n\to\infty} \int{f_n}\de\mu = \int f \de\mu.
\]
\end{lemma}
\begin{proof}
 Certainly, $\abs{\Re f_n},\abs{\Im f_n} \leq \abs{f_n}\leq \abs{g}$ and $\Re f_n\to \Re f$,  $\Im f_n\to \Im f$ $\mu$-almost everywhere. Also, $\abs{f-f_n} \leq  \abs{\Re f-\Re f_n} + \abs{\Im f-\Im f_n}$. Now for real-valued measurable functions, the theorem is assumed to be known. See \autocite[142]{Klenke} for a reference.
 \end{proof}

The following lemma studies the Stieltjes transform $S_\mu(z)=\int_{\R} \frac{1}{x-z} \mu(\de{x})$. Note that we do not have to consider the trivial case where $\mu\equiv 0$, since in this case, $S_{\mu}\equiv 0$. Notationally, we set $\C_+\defeq\{z\in\C\,|\,\Im(z)>0\}$.\label{sym:complexnumbersposim}\label{sym:imaginarypart}
\begin{lemma}
\label{lem:stieltjesprops}
Let $\mu$ be a finite measure on $(\R,\Bcal)$ with $\mu(\R)>0$ and $S_{\mu}$ be its Stieltjes transform. Further, let $E\in\R$, $\eta \in\R\backslash\{0\}$ and $z\defeq E+i\eta$, then we obtain:

\begin{enumerate}[i)]
	\item For any $x\in\R$ we find: $\frac{1}{x-z}=\frac{x-E}{(x-E)^2+\eta^2} + i \frac{\eta}{(x-E)^2+\eta^2}$.
	\item $\Re S_{\mu}(z)=\int\frac{x-E}{(x-E)^2+\eta^2}\mu(\de x) \quad \text{and}\quad \Im S_{\mu}(z)=\int\frac{\eta}{(x-E)^2+\eta^2}\mu(\de x).$\label{sym:realpart}
	\item $\Im(z)\gtrless 0 \Leftrightarrow \Im S_{\mu}(z) \gtrless 0$.
	\item $S_\mu(\overline{z})= \overline{S_\mu(z)}$.
	\item $S_\mu$ is uniquely determined by its restriction $S_\mu: \C_+\to \C_+$. 
	\item $\abs{S_\mu (z)} \leq \frac{\mu(\R)}{\abs{\Im(z)}}$
	\item $S_\mu$ is holomorphic.
	\item In particular, $S_\mu$ is continuous, can be represented by a power series around any $z_0\in\C\backslash\R$, and is infinitely often differentiable.
\end{enumerate}

\end{lemma}
\begin{proof} 
Statement i) is obvious, ii) follows from i) by definition of the complex-valued integral, iii) follows directly from ii) and so does iv) in combination with the construction of the integral. Statement v) follows directly from iii) and iv), and vi) follows from
\[
\bigabs{\frac{1}{x-z}}=\frac{1}{\abs{x-z}}\leq \frac{1}{\abs{\Im(x-z)}}=\frac{1}{\abs{\Im(z)}}.
\]
To show statement vii), let $(z_n)_n$ and $z \in \C\backslash\R$ with $z_n\to z$, but $z_n\neq z$ be arbitrary, then:
	\begin{align*}
	&\frac{S_\mu(z_n)-S_\mu(z)}{z_n - z} = \frac{1}{z_n - z} \int \frac{1}{x - z_n} - \frac{1}{x - z}\mu(\de x)\\
	&=\frac{1}{z_n - z}  \int \frac{z_n-z}{(x - z_n)(x - z)} \mu(\de x) \xrightarrow[n\to\infty]{} \int \frac{1}{(x - z)^2} \mu(\de x) 
	\end{align*}
	by dominated convergence, since for some $C>0$ and all $n\in\N$,
	\[
	\bigabs{\frac{1}{(x - z_n)(x - z)}}\leq \frac{1}{\abs{\Im(z_n)}\abs{\Im(z)}}\leq C,
	\]
	for convergent sequences are bounded.
\end{proof}

\begin{theorem}[Retrieval of Measure]
\label{thm:invert}
For any bounded interval $I\subseteq\R$ with end points $\alpha<\beta$, we obtain the following:
\[
\mu((\alpha,\beta)) + \frac{1}{2}(\mu(\{\alpha\})+\mu(\{\beta\})) = \lim_{\eta\searrow 0} \frac{1}{\pi}\int_{I} \Im S_\mu(E+i\eta)\lebesgue(\de E).
\]
\end{theorem}
\begin{proof}
Let $I$ be an interval with end points $\alpha<\beta$ and $\eta>0$. Then we obtain via Fubini:
\begin{align*}
\frac{1}{\pi} \int_I \Im S_{\mu}(E+i\eta)\lebesgue(\text{d}E) 
&= \frac{1}{\pi} \int_{I}\int_{\R}\frac{\eta}{(x-E)^2+\eta^2}\mu(\text{d}x)\lebesgue(\text{d}E)\\
&= \frac{1}{\pi} \int_{\R}\int_{I}\frac{\eta}{(x-E)^2+\eta^2}\lebesgue(\text{d}E)\mu(\text{d}x)\\
&= \frac{1}{\pi} \int_{\R}\int_{\alpha}^{\beta}\frac{\eta}{(x-E)^2+\eta^2}\text{d}E\mu(\text{d}x).
\end{align*}
Now since
\begin{align*}
\int_{\alpha}^{\beta}\frac{\eta}{(x-E)^2+\eta^2}\text{d}E 
&= \frac{1}{\eta} \int_{\alpha}^{\beta} \frac{1}{(\frac{E-x}{\eta})^2+1}\text{d}E \\
&= \int_{\frac{\alpha-x}{\eta}}^{\frac{\beta-x}{\eta}}\frac{1}{E^2+1}\text{d}E\\
&= \arctan\left(\frac{\beta-x}{\eta}\right) - \arctan\left(\frac{\alpha-x}{\eta}\right),
\end{align*}
and $\arctan:\R\to(-\frac{\pi}{2},+\frac{\pi}{2})$ is strictly increasing with $\lim_{x\to\pm\infty} \arctan(x) = \pm \frac{\pi}{2}$, we obtain
\[
\lim_{\eta\searrow 0}\left[ \arctan\left(\frac{\beta-x}{\eta}\right) -\arctan\left(\frac{\alpha-x}{\eta}\right)\right] = 
\begin{cases}
\pi &\text{ if } x\in(\alpha,\beta)\\
0 & \text{ if } x \notin [\alpha,\beta]\\
\frac{\pi}{2} &\text{ if } x=\alpha \vee x=\beta.
\end{cases}
\]
Thus, by dominated convergence we find
\begin{align*}
\lim_{\eta\searrow 0}\frac{1}{\pi} \int_I \Im S_{\mu}(E+i\eta)\lebesgue(\text{d}E)  
&= \lim_{\eta\searrow 0}\frac{1}{\pi} \int_{\R}\arctan\left(\frac{\beta-x}{\eta}\right) - \arctan\left(\frac{\alpha-x}{\eta}\right)\mu(\text{d}x)\\
&= \int_{\R}\one_{(\alpha,\beta)}(x) + \frac{1}{2}\one_{\{\alpha,\beta\}}(x)\mu(\text{d}x)\\
&= \mu((\alpha,\beta)) + \frac{1}{2}(\mu(\{\alpha\})+\mu(\{\beta\})) 
\end{align*}

\end{proof}

The previous theorem and the following corollary are similar to Theorem 2.4.3 in \autocite{Anderson}. As usual, for a subset $I$ of a topological space, we denote by $\partial I$ its boundary, which is a concept we assume to be known to the reader.
\begin{corollary}
\label{cor:invert}
For any bounded interval $I\subseteq\R$ with $\mu(\partial I)$ = 0, we find:
\[
\mu(I) = \lim_{\eta\searrow 0} \frac{1}{\pi}\int_{I} \Im S_\mu(E+i\eta)\lebesgue(\de E).
\]
Thus, any finite measure $\mu$ on $(\R,\Bcal)$ is uniquely determined by $S_\mu$. In other words, $\mu\mapsto S_{\mu}$ is injective.
\end{corollary}
\begin{proof} The convergence statement follows from Theorem~\ref{thm:invert}.
Now if $\mu$ and $\nu$ are finite measures on $(\R,\Bcal)$ with $S_{\mu}=S_{\nu}$, denote their sets of atoms by $A_{\rho}\defeq\{x\in\R,\rho(\{x\})>0\}$ for $\rho\in\{\mu,\nu\}$. Then $A_{\mu}$ and $A_{\nu}$ are at most countable. Let $a<b$ be arbitrary real numbers, then there are sequences $(a_n)_n$ and $(b_n)_n$ in $\R\backslash(A_{\mu}\cup A_{\nu})$ with $a<a_n<b_n<b$ for all $n\in\N$ such that $a_n\searrow a$ and $b_n\nearrow b$. It follows
\[
\mu((a,b)) = \lim_{n\to\infty}\mu((a_n,b_n)) = \lim_{n\to\infty}\nu((a_n,b_n)) = \nu((a,b)),
\]
where we used continuity of measures in the first and last step and Theorem~\ref{thm:invert} in the middle step. Since $\mu$ and $\nu$ agree on all open bounded intervals, we conclude $\mu=\nu$.
\end{proof}

The last theorem and its corollary suggest that for any finite measure $\mu$ on $(\R,\Bcal)$ and $\eta > 0$ small, $E\mapsto \frac{1}{\pi}\Im S_{\mu}(E+i\eta)$ acts as a Lebesgue density for (a measure approximating) $\mu$. In particular, even measures that do not possess a Lebesgue density (for example, all empirical measures) can be approximated in this way by using the Stieltjes transform. In Section~\ref{sec:imaginarystieltjes} we will see how this can be made precise.

\section{The Stieltjes Transform and Weak Convergence}
\label{sec:stieltjesconv}

For any finite measure $\mu$, $S_{\mu}$ carries all the information of $\mu$ (cf. Corollary~\ref{cor:invert}). Therefore, it is not surprising that this tool can be used particularly well to analyze weak convergence of probability measures. The following theorem generalizes Theorem 2.4.4 in \autocite{Anderson}.

\begin{theorem}[Convergence Theorem]
\label{thm:stieltjesconvergence}
Let $Z\subseteq\C\backslash\R$ be a subset that has an accumulation point in  $\C\backslash\R$ (which is not necessarily an element of $Z$ itself). Then the following statements hold:
\begin{enumerate}
	\item Let $(\mu_n)_n$ in $\Mcal_1(\R)$, such that for all $z\in Z$ we find that $S(z)\defeq\lim_{n\to\infty} S_{\mu_n}(z)$ exists. Then there is a sub-probability measure $\mu$ with $\mu_n\to\mu$ vaguely and $S_\mu=S$.
	\item Let $(\mu_n)_n$ and $\mu$ in $\Mcal_1(\R)$, then we find:
	\[
	\mu_n \to\mu \text{ weakly}\ \Leftrightarrow\ S_{\mu_n}(z)\to S_\mu(z) \text{ for all } z\in Z.
	\]
	\item Let $(\mu_n)_n$ be random probability measures and $\mu$ be a deterministic probability measure, then:
	\begin{enumerate}
		\item[a)] $\mu_n \to\mu$ weakly in expectation $\Leftrightarrow \E S_{\mu_n}(z) \to S_\mu(z)$ for all $z\in Z$.
		\item[b)] $\mu_n \to\mu$ weakly in probability $\Leftrightarrow S_{\mu_n}(z) \to S_\mu(z)$ in probability for all $z\in Z$.
		\item[c)] $\mu_n \to\mu$ weakly almost surely $\Leftrightarrow \left[S_{\mu_n}(z) \to S_\mu(z)~\text{almost surely}\right]$ for all $z\in Z$.
	\end{enumerate}
\end{enumerate}
\end{theorem}
\begin{proof}
\underline{1.} Let $(\mu_n)_{n\in J}$ be an arbitrary subsequence of $(\mu_n)_{n\in \N}$. Due to Lemma~\ref{lem:convergentsubsequence}, there exists a subsequence $(\mu_n)_{n\in I}$, $I\subseteq J$, such that $\mu_n\to\mu$ vaguely for $n\in I$ and a sub-probability measure $\mu$. Since $x\mapsto \frac{1}{x-z}$ vanishes at $\pm \infty$, it follows $S_{\mu_n}(z)\to S_{\mu}(z)$ for $n\in I$ for all $z\in Z$ (cf.\ Lemma~\ref{lem:vagueczero}). Therefore, $S(z)=S_{\mu}(z)$ for all $z\in Z$. If $\nu$ is another subsequential limit of $(\mu_n)_{n\in J}$, we find by the same argument that $S_{\mu}(z) = S(z) = S_{\nu}(z)$ for all $z\in Z$. This implies $S_{\mu}=S_{\nu}$, since Stieltjes transforms are holomorphic. Therefore, $\mu=\nu$ by Theorem~\ref{thm:invert}. By Lemma~\ref{lem:subsequential}, we find $\mu_n\to\mu$ vaguely for $n\in\N$.\newline
\underline{2.} Since $x\mapsto \frac{1}{x-z}$ is continuous, "$\Rightarrow$" is obvious. To show "$\Leftarrow$", statement 1 yields that $\mu_n\to\mu$ vaguely, thus $\mu_n\to\mu$ weakly, since all measures involved are probability measures (cf.\ Lemma~\ref{lem:vaguetoweak}).\newline
\underline{3.a)} This follows directly from statement 2, considering
\[
\E S_{\mu_n}(z) = \E \int\frac{1}{x-z}\mu_n(\text{d}x)=\int \frac{1}{x-z} \E\mu_n(\text{d}x) = S_{\E\mu_n}(z),
\] 
where we used Theorem~\ref{thm:expectedmeasure}.\newline
\underline{3.c)} If $\mu_n\to\mu$ weakly on a measurable set $A$ with $\Prob(A)=1$, then we have on $A$ that for all $z\in Z$ we find $S_{\mu_n}(z)\to S_{\mu}(z)$ (by statement 2). This shows "$\Rightarrow$", and to show "$\Leftarrow$", fix a sequence $(z_k)_k$ in $Z$ that converges to some $z\in\C\backslash\R$. For all $k\in\N$ we find a measurable set $A_k$ with $\Prob(A_k)=1$ on which $S_{\mu_n}(z_k) \to S_{\mu}(z_k)$ as $n\to\infty$. Then $A\defeq\cap_{k\in\N} A_k$ is measurable with $\Prob(A)=1$, and on $A$ we find that for all $z\in Z'\defeq \{z_k|k\in\N\}$ we have $S_{\mu_n}(z)\to S_{\mu}(z)$. Since $Z'$ has an accumulation point in $\C\backslash\R$, we find on the set $A$ that $\mu_n\to\mu$ weakly by statement 2.\newline
\underline{3.b)} The direction "$\Rightarrow$" is trivial since $x\mapsto \frac{x-E}{(x-E)^2 + \eta^2}$ and $x\mapsto \frac{\eta}{(x-E)^2 + \eta^2}$ are bounded and continuous (cf.\ Theorem~\ref{thm:stochasticweakconvergence}). For "$\Leftarrow$" we let $f\in\Ccal_b(\R)$ be arbitrary. Then we need to show that $\integrala{\mu_n}{f}\to\integrala{\mu}{f}$ in probability. Let $J\subseteq \N$ be a subsequence, then by Lemma~\ref{lem:uniformsubsequence}, we find a subsequence $I\subseteq J$ and a measurable set $N$ with $\Prob(N) = 0$, such that for $(z_k)_k$ fixed as in the proof of 3.c):
\[
\forall\, \omega\in \Omega\backslash N:\, \forall\, k\in\N: S_{\mu_n(\omega)}(z_k) \xrightarrow[n\in I]{} S_{\mu(\omega)}(z_k).
\]
Therefore, it follows with statement 3.c) that
$\mu_n\xrightarrow[n\in I]{}\mu$ almost surely, so in particular $\integrala{\mu_n}{f} \xrightarrow[n\in I]{} \integrala{\mu}{f}$ almost surely. Then $\integrala{\mu_n}{f} \xrightarrow[n\in \N]{} \integrala{\mu}{f}$ in probability by Lemma~\ref{lem:subsequence}.
\end{proof}

We refer the reader to Remark~\ref{rem:stochasticweakconvergence} for an explanation on the use of brackets $[\ldots]$ in Theorem~\ref{thm:stieltjesconvergence} 3.\ $c)$.

\section{The Imaginary Part of the Stieltjes Transform}
\label{sec:imaginarystieltjes}

In Corollary~\ref{cor:invert} we saw that if $\mu\in\Mcal_1(\R)$, then for a small $\eta>0$, the function $E\mapsto \frac{1}{\pi}\Im S_{\mu}(E+i\eta)$ should be the Lebesgue density of a probability measure on $(\R,\Bcal)$ that approximates $\mu$ well. But so far, we do not even know whether $E\mapsto \frac{1}{\pi}\Im S_{\mu}(E+i\eta)$ yields a density of a \emph{probability measure} at all. How can this intuition be portrayed in the right context, and is there a connection to the weak convergence results of Section~\ref{sec:stieltjesconv}? This section aims to shed light onto these aspects. First, we will rigorously delve into convolution of probability measures, which will be based on \autocite{Alsmeyer}. Second, we will introduce kernel density estimators, which motivate further the use of the Stieltjes transform when analyzing ESDs of random matrices.
We begin by making the following definition:

\begin{definition}
\label{def:convolution}
Let $\mu$ and $\nu$ be probability measures on $(\R,\Bcal)$ and $f,g:\R\to\R$ Lebesgue-density functions (i.e.\ $h\geq 0$ and $\int h \de\lebesgue =1$, $h\in\{f,g\}$). 
\begin{enumerate}[i)]
\item The convolution of the probability measures $\mu$ and $\nu$ is defined as $\mu\ast\nu\defeq (\mu\otimes\nu)^{+}$\label{sym:measureconvmeasure}. Here, $\mu\otimes\nu$\label{sym:productmeasure} is the product measure on $(\R^2,\Bcal^2)$, $+:\R^2\to\R$ is the addition map, and $(\mu\otimes\nu)^{+}$ is the push-forward of the product measure under the addition map.
\item The convolution of the density $f$ and the probability measure $\nu$ is defined as the function $f\ast\nu:\R\to\R$\label{sym:functionconvmeasure} with
\[
\forall\,x\in\R:~ (f\ast\nu)(x)\defeq \int_{\R}f(x-y)\nu(\de y).
\]
\item The convolution of the densities $f$ and $g$ is the function $f\ast g:\R\to\R$\label{sym:functionconvfunction} with
\[
\forall\,x\in\R:~ (f\ast g)(x)\defeq \int_{\R}f(x-y)g(y)\lebesgue(\de y).
\]
\end{enumerate}
Note that in ii) and iii) above, the definitions of the convolution are to be understood for $\lebesgue$-almost all $x\in\R$, since the respective integrals are well-defined only for $\lebesgue$-almost all $x\in\R$, which can be observed via Fubini/Tonelli. The convolutions are understood to equal zero on the respective sets of measure zero. 
\end{definition}

We will now casually discuss Definition~\ref{def:convolution} and summarize the most important aspects of our findings in the next lemma. So let us assume we are in the situation of said definition. 

Let us first discuss point $i)$ of Definition~\ref{def:convolution}: Per construction, the convolution of two probability measures yields another probability measure on the real line, and if $B\in\Bcal$ is arbitrary, we find
\[
(\mu\ast\nu)(B) = (\mu\otimes\nu)\left(\left\{(x,y)\in\R^2: x+y\in B\right\}\right).
\]
Further, if $f:\R\to\R$ is $\mu\ast\nu$-integrable, then we obtain by transformation:
\[
\int_{\R}f \de (\mu\ast\nu)
= \int_{\R^2}(f\circ +) \de(\mu\otimes\nu)
= \int_{\R^2}f(x+y) (\mu\otimes\nu)(\de(x,y))
= \int_{\R}\int_{\R}f(x+y)\mu(\de x)\nu(\de y),
\]
so that in particular for an indicator function $f=\one_B$ for some $B\in\Bcal$:
\[
(\mu\ast\nu)(B)=\int_{\R}\one_B\de(\mu\ast\nu)=\int_{\R}\int_{\R}\one_B(x+y)\mu(\de x)\nu(\de y)\\=\int_{\R}\mu(B-y)\nu(\de y),
\]
where the fact that the first term is equal to the third shows in a particularly nice way that the convolution is commutative (via Fubini). Let us introduce a quick but enlightening example:
\begin{example}[Convolution with Dirac measures]
\label{ex:convolution}
For all $a\in\R$, denote by $\delta_a$ the Dirac measure in $a$ and by $T_a$ the translation by $a$, that is, $T_a:\R\to\R$, $T_a(x)=x+a$ for all $x\in\R$. Then we find for any probability measure $\mu\in\Mcal_1(\R)$:
\[
\mu\ast\delta_a=\mu^{T_a},\quad \text{in particular:}\quad \mu\ast\delta_0 =\mu,
\]
since $T_0=\text{id}_{\R}$. We conclude that $\delta_0$ is the neutral element of convolution (there is no other neutral element, since $\ast$ is commutative). To prove our claim, we calculate for an arbitrary $B\in\Bcal$:
\[
(\mu\ast\delta_a)(B) = \int_{\R} \mu(B-y)\delta_a(\de y) = \mu(B-a)=\mu^{T_a}(B),
\]
where we used that $T_a^{-1}=T_{-a}$.
\end{example}

Now, let us discuss point $ii)$ of Definition~\ref{def:convolution}: First of all, we point out that $f\ast\nu$ is a Lebesgue-density function, for it is non-negative, and via Fubini we obtain immediately that $\int f\ast\nu\de\lebesgue=1$. But even more holds: $f\ast\nu$ is the Lebesgue-density of the convolution $(f\lebesgue)\ast\nu$, so that the equality $(f\lebesgue)\ast\nu = (f\ast\mu)\lebesgue$ holds. In particular, this convolution is Lebesgue-continuous. To verify our statement, we calculate for an arbitrary $B\in\Bcal$:
\begin{align*}
[(f\lebesgue)\ast\nu](B)
			&= \int_{\R}(f\lebesgue)(B-y)\nu(\de y)\\
			&= \int_{\R}\int_{B-y}f(x)\lebesgue(\de x)\nu(\de y)\\
			&=\int_{\R}\int_{B}f(x-y)\lebesgue(\de x)\nu(\de y)\\
			&=\int_{B}\int_{\R}f(x-y)\nu(\de y)\lebesgue(\de x)\\
			&=\int_B [f\ast\nu](x)\lebesgue(\de x),
\end{align*}
where the third step follows from 
\[
\int_{B-y}f(x)\lebesgue(\de x) = \int_{T^{-1}_y(B)}(f\circ T^{-1}_y\circ T_y)(x)\lebesgue(\de x) = \int_{B}(f\circ T^{-1}_y)(x)\lebesgue^{T_y}(\de x),
\]
and the Lebesgue measure is translation invariant, thus $\lebesgue^{T_y}=\lebesgue$.

Lastly, let us discuss point $iii)$ of Definition~\ref{def:convolution}: Again by Fubini, we find immediately that $f\ast g$ is a Lebesgue-density function. Now since from the definition we have for all $x\in\R$ that $(f\ast g)(x)= (f\ast (g\lebesgue))(x)$, we find through our discussion of point $ii)$ that $f\ast g$ is the Lebesgue-density of the convolution $(f\lebesgue)\ast(g\lebesgue)$, so $(f\ast g)\lebesgue = (f\lebesgue)\ast(g\lebesgue)$. Let us summarize our findings in the following lemma:
\begin{lemma}
\label{lem:convolution}
In the situation of Definition~\ref{def:convolution}, we make the following observations (point $x$ here is with respect to point $x$ in Definition~\ref{def:convolution}, $x\in\{i),ii),iii)\}$):
\begin{enumerate}[i)]
\item The convolution is a commutative binary operation on the space of probability measures. The neutral element is given by $\delta_0$. Further, the following formula holds:
\[
\forall\,B\in\Bcal: (\mu\ast\nu)(B) = \int_{\R}\mu(B-y)\nu(\de y).
\]
\item $f\ast\nu$ is a Lebesgue-density for the convolution $(f\lebesgue)\ast\nu$, that is, $(f\lebesgue)\ast\nu = (f\ast\nu)\lebesgue$.
\item $f\ast g$ is a Lebesgue-density for the convolution $(f\lebesgue)\ast(g\lebesgue)$, that is, $(f\lebesgue)\ast(g\lebesgue) = (f\ast g)\lebesgue$.
\end{enumerate}
\end{lemma}
\begin{proof}
This follows from the discussion preceding the lemma.
\end{proof}
The following lemma will capture a very important property of the convolution:
\begin{lemma}
The convolution of probability measures on $(\R,\Bcal)$ is continous with respect to weak convergence. That is, if $(\mu_n)_n$, $(\nu_n)_n$, $\mu$ and $\nu$ are probability measures on $(\R,\Bcal)$ with $\mu_n\to\mu $ and $\nu_n\to\nu$ weakly, then $\mu_n\ast\nu_n\to\mu\ast\nu$ weakly.
\end{lemma}
\begin{proof}
With \autocite[23]{Billingsley} it follows that $\mu_n\otimes\nu_n\to\mu\otimes\nu$. Now if $f\in\Ccal_b(\R)$ is arbitrary, then we also have that $(x,y)\mapsto f(x+y)$ is a continuous and bounded function on $\R^2$, so
\[
\int_{\R} f \de(\mu_n\ast\nu_n) = \int_{\R^2} (f\circ+) \de (\mu_n\otimes\nu_n) \xrightarrow[n\to\infty]{} \int_{\R^2} (f\circ+) \de (\mu\otimes\nu) = \int_{\R} f \de(\mu\ast\nu).
\]
\end{proof}

Now, we will bring the Stieltjes transform into play:

\begin{definition}
\label{def:cauchykernel}
For all $\eta>0$, we define the Cauchy kernel $P_{\eta}:\R\to\R$\label{sym:Cauchykernel} as the function with
\[
\forall\,x\in\R: P_{\eta}(x) \defeq \frac{1}{\pi}\frac{\eta}{x^2+\eta^2}, 
\]
which is the $\lebesgue$-density function of the Cauchy distribution with scale parameter $\eta$.
\end{definition}

We will collect a quick lemma before proceeding:

\begin{lemma}
\label{lem:diracdelta}
As $\eta\searrow 0$, we find $(P_{\eta}\lebesgue)\to \delta_0$ weakly.	
\end{lemma}
\begin{proof}
The characteristic function of the measure $P_{\eta}\lebesgue$ is given by $t\mapsto e^{-\eta\abs{t}}$, see \autocite[330]{Klenke} and \autocite[208]{Rueschendorf}. Fixing $t\in\R$ and letting $\eta\to 0$ will yield the statement, since $e^0$ is the characteristic function of $\delta_0$.
\end{proof}
Now, as we see, for any probability measure $\mu$ on $(\R,B)$, we have 
\[
\frac{1}{\pi}\Im S_{\mu}(E+i\eta) = \int_{\R}\frac{1}{\pi}\frac{\eta}{(E-x)^2+\eta^2}\mu(\de x) = (P_{\eta}\ast \mu)(E)
\]
Therefore, $1/\pi\Im S_{\mu}(\cdot + i\eta)$ is the convolution of the density $P_{\eta}$ with $\mu$ and thus a Lebesgue-density for the probability measure $(P_\eta\lebesgue)\ast\mu$. In particular, as $\eta \searrow 0$ we have that 
\[
\frac{1}{\pi}\Im S_{\mu}(\cdot+i\eta)\lebesgue = (P_\eta\lebesgue)\ast\mu \longrightarrow \delta_0\ast\mu =\mu \quad \text{weakly.}
\]
This immediately proves Corollary~\ref{cor:invert} again (using the Portmanteau theorem). But due to continuity of the convolution, we can say much more:

Assume that $(\sigma_n)_n$ is a sequence of ESDs of Hermitian random matrices, so that $\sigma_n$ converges almost surely to the semicircle distribution $\sigma$. We assume this convergence takes place on a measurable set $A$ with $\Prob(A)=1$. Then we find on $A$ that the following commutative diagram holds, where all arrows indicate weak convergence:
\begin{center}
\begin{tikzpicture}
  \matrix (m) [matrix of math nodes,row sep=7em,column sep=6em,minimum width=2em]
  {
     (P_{\eta}\ast\sigma_n)\lebesgue & (P_{\eta}\ast\sigma)\lebesgue \\
     \delta_0\ast\sigma_n=\sigma_n & \sigma \\};
  \path[-stealth]
    (m-1-1) edge node [left] {$\eta\searrow 0$} (m-2-1)
            edge node [below] {$n\to\infty$} (m-1-2)
    (m-2-1.east|-m-2-2) edge node [below] {$n\to\infty$}
            (m-2-2)
    (m-1-2) edge node [right] {$\eta\searrow 0$} (m-2-2)
    (m-1-1) edge node [below left] {$\substack{n\to\infty\\ \eta\searrow 0}$} (m-2-2);
\end{tikzpicture}
\end{center}
\noindent
In particular, the diagonal arrow says that we obtain weak convergence  $(P_{\eta_n}\ast\sigma_n)\lebesgue\to\sigma$ as $n\to\infty$ for any sequence $\eta_n\searrow 0$. This is an interesting result, but it does not tell us if also densities align. More concretely, write $\sigma = f_{\sigma}\lebesgue$, then from $(P_{\eta}\ast\sigma_n)\lebesgue \to f_{\sigma}\lebesgue$ weakly we cannot infer that also $P_{\eta}\ast\sigma_n \to f_{\sigma}$ in some sense, for example in $\supnorm{\cdot}$ over a specified compact interval. This is desirable since it allows conclusion about \emph{local} estimation of $\sigma_n$ by $\sigma$.
If $\eta=\eta_n$ drops too quickly to zero as $n\to\infty$, then $(P_{\eta_n}\ast\sigma_n)$ will have steep peaks at each eigenvalue, thus will not approximate the density of the semicircle distribution uniformly. This "problem" is typical for kernel density estimators in general (see \autocite{Rueschendorf}, especially their Remark 11.2.10), which we will introduce next.

\begin{definition}
A kernel $K$ is a Lebesgue-probability-density function $\R\to\R$, that is, $K$ is non-negative and
\[
\int_{\R} K(y) \lebesgue(dy) =1.
\]
Further, if $K$ is a kernel and $h>0$, we define $K_h$ as the kernel with $K_h(x) = \frac{1}{h}K(\frac{x}{h})$ for all $x\in\R$ and call $K_h$ the kernel $K$ at bandwidth $h$. In particular, $K=K_1$.
\end{definition}

In above definition, it is clear that $K_h$ is a kernel if $K$ is a kernel and $h>0$. An example of a kernel is the Cauchy kernel $P_{1}$ from Definition~\ref{def:cauchykernel}, which yields the standard Cauchy distribution.  We have for all $x\in\R$ and $\eta>0$:
\[
P_1(x) = \frac{1}{\pi}\frac{1}{x^2+1}\quad \text{and}\quad P_{\eta}(x)= \frac{1}{\pi\eta}\frac{1}{\left(\frac{x}{\eta}\right)^2+1}=\frac{1}{\pi}\frac{\eta}{x^2+\eta^2}.
\]

 Now given a vector $v= (v_1,\ldots,v_n)$ of real-valued observations, we are interested in constructing a Lebesgue-density that describes the experiment of drawing uniformly at random from these observations, in other words that approximates the empirical probability measure
\begin{equation}
\label{eq:empmeasure}
\nu_N\defeq \frac{1}{n}\sum_{i=1}^n \delta_{v_i}.
\end{equation}
This can be done with help of a kernel $K$, which is oftentimes chosen to be unimodal and symmetric around $0$, just as the Cauchy kernel $P_1$.  

\begin{definition}
The \emph{kernel density estimator with kernel $K$ and bandwidth $h>0$} for an empirical measure $\nu$ as in \eqref{eq:empmeasure} is the Lebesgue-density given by the convolution $K_h\ast \nu$, thus
\map{K_h\ast \nu}{\R}{\R}{x}{(K_h\ast \nu)(x) = \frac{1}{n}\sum_{i=1}^n K_h(x-v_i)= \frac{1}{nh}\sum_{i=1}^n K_1\left(\frac{x-v_i}{h}\right)}
\end{definition}

Heuristically speaking, the concept works in the following way: The center of the kernel is placed upon each observation, whose influence (i.e. probability mass of $1/n$) is smoothed over its neighborhood. The size of this neighborhood is governed by the bandwidth $h$: A small $h$ will restrain the probability mass of $1/n$ to be closer to its observation, whereas a larger $h$ will result in a wider spread of probability mass. Therefore, a smaller $h$ will result in a peaky density function (with steep peaks at the observation), whereas a larger $h$ will result in a smoother density function.  

We now assume we are given an empirical spectral distribution $\sigma_N$ from a real symmetric $n\times n$ matrix $X_n$. The kernel density estimator at location $E\in\R$ for $\sigma_n$ with kernel $P_1$ at bandwidth $\eta >0$ is then given by 
\[
(P_{\eta}\ast\sigma_n)(E)= \frac{1}{n \eta} \sum_{i=1}^n \frac{1}{\pi}\frac{1}{\left(\frac{E-\lambda^{X_n}_i}{\eta}\right)^2 +1} = \frac{1}{\pi n}\sum_{i=1}^n \frac{\eta}{(E-\lambda^{X_n}_i)^2 + \eta^2} =\frac{1}{\pi}\Im S_{\sigma_n}(E+i\eta).
\]
This gives the imaginary part of the Stieltjes transform the new role of a kernel density estimator for the empirical spectral distribution. Let us conduct a simulation study for $n=100$. Let $A_{100}$ be a symmetric $100\times 100$ random matrix with independent Rademacher distributed variables in the upper half triangle, including the main diagonal. Let $X_{100} \defeq \frac{1}{\sqrt{100}}A_{100}$. Denote by $\sigma_{100}$ the empirical spectral distribution of $X_{100}$. Further, we define the bandwidths $\eta_1\defeq n^{-1/2} = 1/10$ and $\eta_2\defeq n^{-1}=1/100$.  With respect to the commutative diagram after Lemma~\ref{lem:diracdelta} and the discussion below it, let us analyze how well $P_{\eta_1} \ast \sigma_{100}$ and $P_{\eta_2}\ast\sigma_{100}$ can be approximated by the density of the semicircle distribution, $f_{\sigma}$, in Figures~\ref{fig:goodeta} and~\ref{fig:badeta}, which are based on the same simulation outcome.

\begin{figure}[htbp]
\centering
\includegraphics[clip, trim=0cm 1.5cm 0cm 1cm, width=\textwidth]{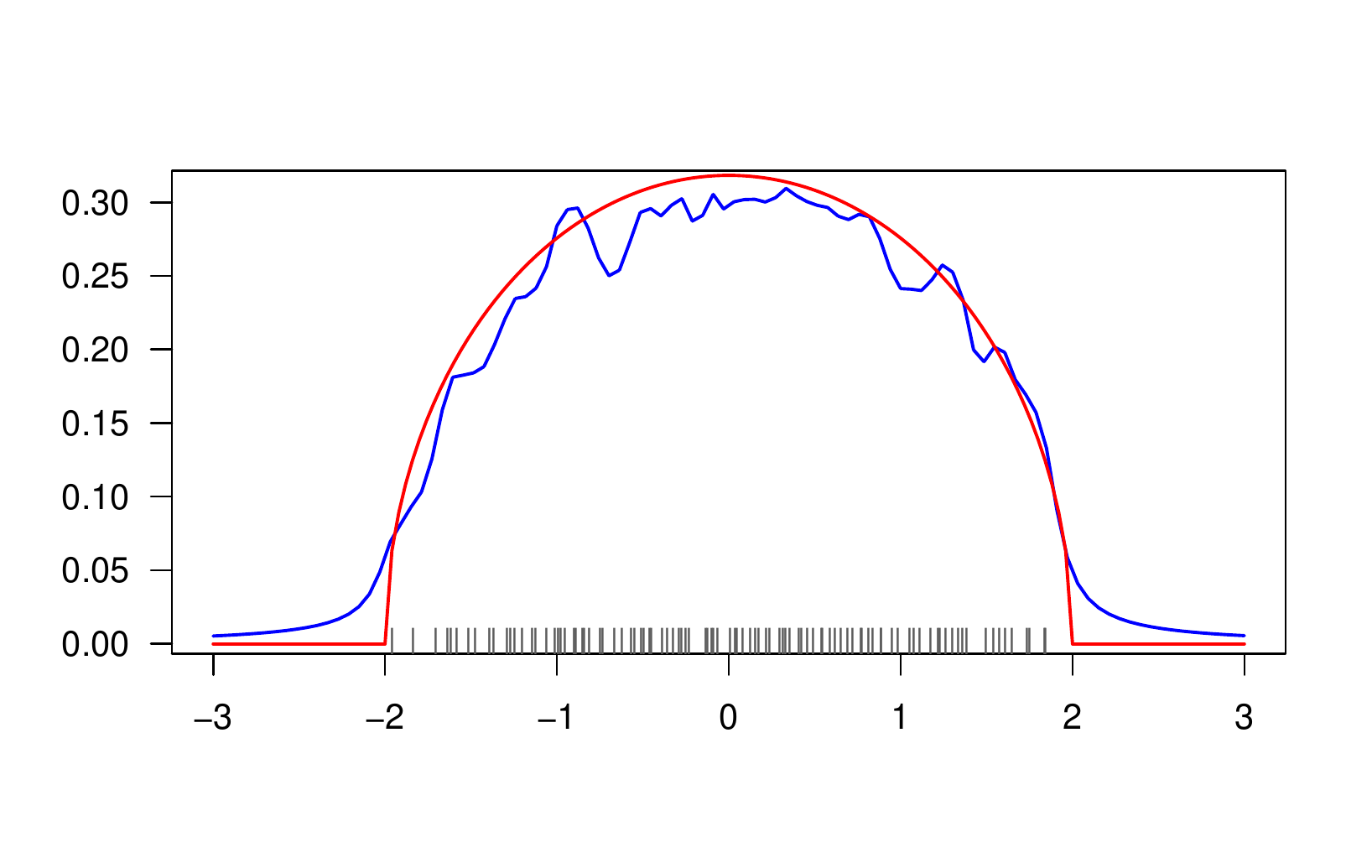}	
\caption{Red line: $f_{\sigma}$. Blue line: $\frac{1}{\pi}\Im S_{\sigma_{100}}(\cdot+i\eta_1)=P_{\eta_1} \ast \sigma_{100}$. Grey bars: eigenvalue locations.}
\label{fig:goodeta}
\end{figure}

\begin{figure}[htbp]
\centering
\includegraphics[clip, trim=0cm 1.5cm 0cm 1cm, width=\textwidth]{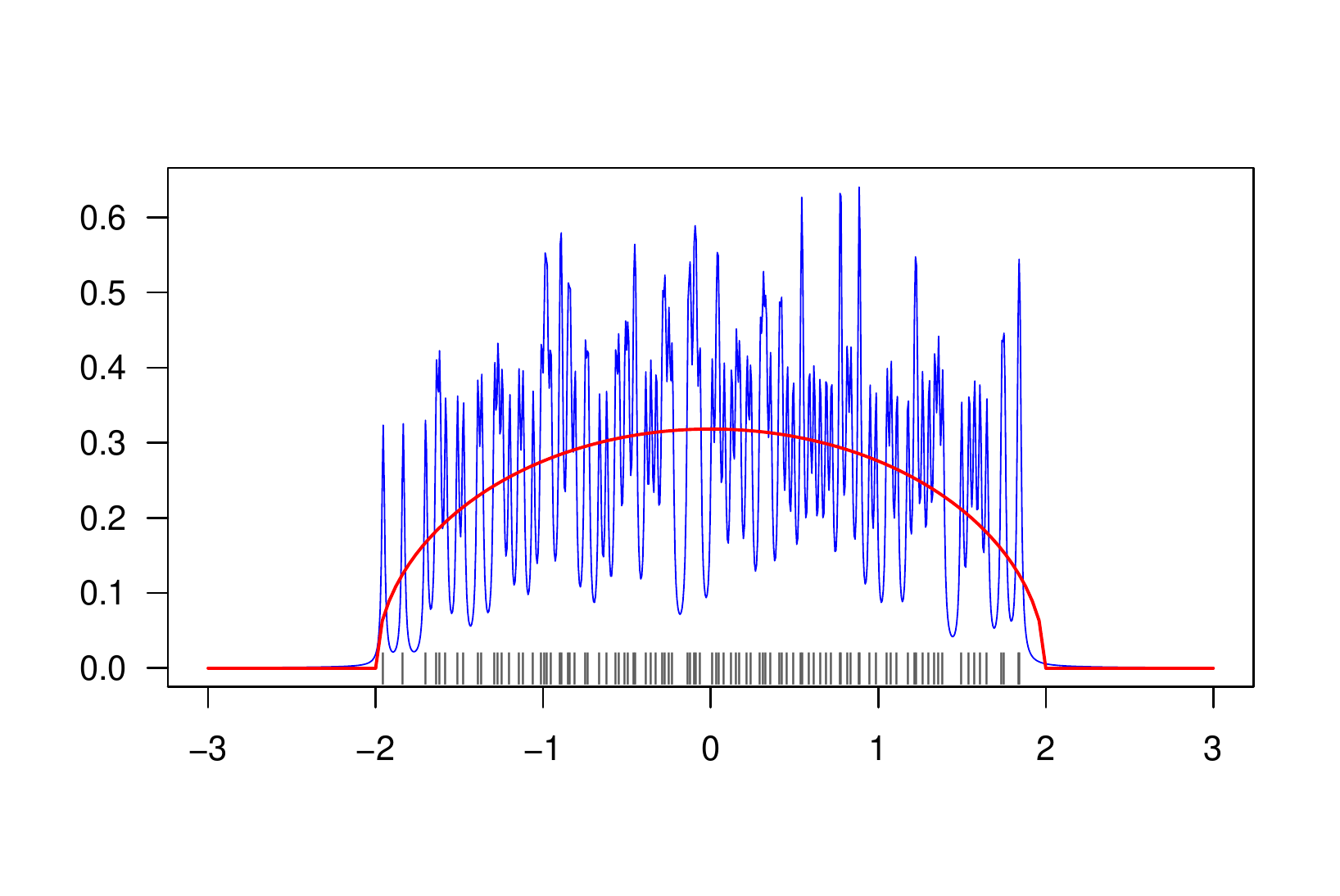}	
\caption{Red line: $f_{\sigma}$. Blue line: $\frac{1}{\pi}\Im S_{\sigma_{100}}(\cdot+i\eta_2)=P_{\eta_2} \ast \sigma_{100}$. Grey bars: Eigenvalue locations.}
\label{fig:badeta}
\end{figure}

As we see, considering that we are in the case of a very low $n=100$, we already obtain a decent approximation by the semicircle density in Figure~\ref{fig:goodeta}. Reducing the scale from $\eta_1$ to $\eta_2$ we obtain the result in Figure~\ref{fig:badeta}. There we observe that for the smaller bandwidth parameter $\eta_2$, we do not obtain a useful approximation by the semicircle density anymore. Indeed, the scale $n^{-1}$ is too fast to obtain uniform convergence of the estimated density to the target density, whereas a scale of $n^{\gamma-1}$ for any $\gamma\in(0,1)$ would be sufficient. Nevertheless, Figure~\ref{fig:badeta} displays nicely how the kernel density estimator works: A closer look -- in particular to the edges of the bulk -- shows how the probability mass of each individual eigenvalue is spread around its neighborhood.

\section{The Stieltjes Transform of ESDs of Hermitian Matrices}

As we motivated the Stieltjes transform in the beginning of this chapter, it is possible to relate the Stieltjes transform of an ESD of a random matrix to the entries of the random matrix. We will now see how this is done.
Notationally, as the Stieltjes transform of the semicircle distribution received the special letter $s\defeq S_{\sigma}$, the Stieltjes transform of an ESD $\sigma_n$ of an Hermitian $n\times n$ matrix $X_n$ is denoted by $s_n\defeq S_{\sigma_n}$\label{sym:StieltjesESD}. The following theorem summarizes the findings of this section (see also \autocite[470-472]{BaiSi}).

\begin{theorem}
\label{thm:StieltjesESD}
Let $X_n$ be an Hermitian $n\times n$ matrix with ESD $\sigma_n$.
\begin{enumerate}[i)]
\item For all $z\in\C\backslash\R$ we find:
\[
s_n(z)=S_{\sigma_n}(z)=\frac{1}{n}\tr (X_n-z)^{-1} = \frac{1}{n}\sum_{k=1}^n \frac{1}{X_n(k,k)-z-x^*_k (X_n^{(k)}-z)^{-1} x_k}.
\]
\item  For $z=E+i\eta$, where $E\in\R$ and $\eta>0$, we obtain for all $k\in\{1,\ldots,n\}$:
\[
\bigabs{\tr{(X_n-z)^{-1}}-\tr{(X_n^{(k)}-z)^{-1}}}\leq\frac{1}{\eta}.
\]
\end{enumerate}
Here, $X_n^{(k)}$\label{sym:princmin} denotes the $k$-th principal minor of $X_n$ (thus an $(n-1)\times(n-1)$ matrix) and $x_k$\label{sym:erasedcolumn} the $k$-th column of $X_n$ without the $k$-th entry (thus an $(n-1)$-vector).
\end{theorem}
\begin{proof}
\underline{$i)$} The first equality is just a notational convention and the last equality is the statement of Corollary~\ref{cor:inversediagonal} below. For the second equality, let $\lambda_1,\ldots,\lambda_n$ be the eigenvalues of $X_n$, then by the spectral theorem for normal operators, $\frac{1}{\lambda_1-z},\ldots,\frac{1}{\lambda_n-z}$ are the eigenvalues of $(X_n-z)^{-1}$. Since  for normal matrices, the trace yields the sum of the eigenvalues, we conclude 
\[
S_{\sigma_n}(z) = \int_{\R}\frac{1}{x-z}\sigma_n(\text{d}x) = \frac{1}{n}\sum_{i=1}^n \frac{1}{\lambda_i-z}=\frac{1}{n}\tr (X_n-z)^{-1}.
\]
\underline{$ii)$} This is the statement of Corollary~\ref{cor:tracedifference} below. 
\end{proof}

Note that Theorem~\ref{thm:StieltjesESD} $i)$ also allows us to work with the Stieltjes transform $S_{\E\sigma_n}$ of the expected ESD $\E\sigma_n$, since as in the proof of Theorem~\ref{thm:stieltjesconvergence} we have $S_{\E\sigma_n} = \E S_{\sigma_n} = \E s_n$.

The remainder of this section will be devoted to the proof of Theorem~\ref{thm:StieltjesESD}, for which we follow the roadmap as in \autocite{BaiSi}.
In the following Lemma, the Schur complement is defined and studied (see also \autocite{ZhangMatrixTheory}).

\begin{lemma}
\label{def:schurcomplement}
Let 
\[
A = 
\begin{pmatrix}
A_{11} & A_{12}\\
A_{21} & A_{22}	
\end{pmatrix}
\]
be a quadratic block matrix with $A_{11}$ invertible. Then the Schur complement of $A_{11}$ in $A$ is  defined as 
\[
B\defeq A_{22}- A_{21}A_{11}^{-1}A_{12}
\]
and has the following properties, where $I$ resp.\ $0$ are identity matrices resp.\ $0$-matrices of appropriate dimension:
\begin{enumerate}[i)]
\item We obtain the \emph{Schur complement formula}
\[
\begin{pmatrix}
I & 0\\
-A_{21}A_{11}^{-1} & I
\end{pmatrix}
\begin{pmatrix}
A_{11} & A_{12}\\
A_{21} & A_{22}	
\end{pmatrix}
\begin{pmatrix}
I & -A_{11}^{-1}A_{12}\\
0 & I	
\end{pmatrix}
=
\begin{pmatrix}
A_{11} & 0\\
0 & B	
\end{pmatrix}.
\]
\item We find the \emph{Schur complement determinant formula}
 \[
 \det(A) = \det(A_{11})\det(B) = \det(A_{11})\det(A_{22}- A_{21}A_{11}^{-1}A_{12})
 \]
\item If $A$ is invertible, so is $B=A_{22}- A_{21}A_{11}^{-1}A_{12}$.
\item In case $A$ is invertible, we find the \emph{Schur complement inversion formula}
\begin{align*}
A^{-1} &= 
\begin{pmatrix}
I & -A_{11}^{-1} A_{12}\\
0 & I
\end{pmatrix}
\begin{pmatrix}
A_{11}^{-1} & 0\\
0 & B^{-1}
\end{pmatrix}
\begin{pmatrix}
I & 0\\
-A_{21}A_{11}^{-1} & I
\end{pmatrix}\\
&=
\begin{pmatrix}
A_{11}^{-1} + A_{11}^{-1}A_{12}B^{-1}A_{21}A_{11}^{-1} & -A_{11}^{-1}A_{12}B^{-1}\\
-B^{-1}A_{21}A_{11}^{-1} & B^{-1}
\end{pmatrix}.
\end{align*}
\end{enumerate}
\end{lemma}
\begin{proof}
Statement $i)$ requires mere verification by multiplication of the matrices, $ii)$ follows directly from $i)$ and $iii)$ follows directly from $ii)$. The first equality of statement $iv)$ follows directly by inverting the Schur complement formula and multiplying from the left and right with the appropriate matrices. The second equality is again verified through simple multiplication of the matrices.
\end{proof}

\begin{lemma}
\label{lem:inversediagonal}
Let $A$ be an invertible $n\times n$ matrix. If $A^{(k)}$ is invertible for some $k\in\{1,\ldots,n\}$, then
\[
A^{-1}(k,k) = \frac{1}{A(k,k)-r_k A^{(k)-1}c_k},
\]
where $r_k$ is the $k$-th row of $A$ without the $k$-th entry and $c_k$ is the $k$-th column of $A$ without the $k$-th entry.
\end{lemma}
\begin{proof}
We first prove the statement for $k=n$. We write
\[
A =
\begin{pmatrix}
A^{(n)} & c_n\\
r_n & A(n,n)
\end{pmatrix}
\]
and set $B\defeq A(n,n)-r_n A^{(n)-1} c_n$. Then by the Schur complement inversion formula, $A^{-1}(n,n)=B^{-1}$, which shows the statement for $k=n$. Next, we assume $k<n$. Then define a permutation matrix column-wise as
\[
V \defeq (e_1|e_2|\ldots|\widehat{e_{k}}|\ldots|e_n|e_k)
\]
where the $e_i$ are the standard $n$-dimensional basis vectors, and the hat over $e_{k}$ indicates that this vector is left out. In other words, $V$ is obtained through the identity matrix by erasing its $k$-th column $e_k$ and appending it at the end of the matrix. We obtain immediately that $V^T=V=V^{-1}$. Then $AV$ is the matrix $A$ with erased and then appended $k$-th column and $VA$ is the matrix $A$ with erased and then appended $k$-th row. Therefore, $(VAV)^{(n)} = A^{(k)}$ and by the case above
\[
A^{-1}(k,k) = (V A^{-1} V)(n,n) = (VAV)^{-1}(n,n) = \frac{1}{(VAV)(n,n) - r'_n (VAV)^{(n)-1} c'_n},
\]
where $r'_n$ denotes the $n$-th row of $VAV$ and $c'_n$ denotes the $n$-th column of $VAV$, both without their $n$-th entry. But $r'_n = r_k$, $c'_n=c_k$ and $(VAV)(n,n)= A(k,k)$.
\end{proof}

\begin{corollary}
\label{cor:inversediagonal}
Let $X_n$ be an Hermitian $n\times n$ matrix, then it holds for $z\in\C\backslash \R$:
\[
\tr (X_n-z)^{-1} = \sum_{k=1}^n \frac{1}{X_n(k,k)-z-x^*_k (X^{(k)}_n-z)^{-1}x_k},
\]
\end{corollary}
where $X_n^{(k)}$ denotes the $k$-th principal minor of $X_n$ and $x_k$ the $k$-th column of $X_n$ without the $k$-th entry.

\begin{proof}
$X_n$ and all $X_n^{(k)}$ are Hermitian, thus $X_n-z$ and $X_n^{(k)}-z= (X_n-z)^{(k)}$ are invertible for all $k$. We also know that the $k$-th column (resp.\ row) of $X_n$ without the $k$-th entry is also the  $k$-th column (resp.\ row) of $(X_n-z)$ without the $k$-th entry. Therefore, the statement follows directly with Lemma~\ref{lem:inversediagonal}.
\end{proof}

\begin{lemma}
\label{lem:tracedifference}
Let $A$ be an invertible $n\times n$ matrix and $k\in\{1,\ldots,n\}$, such that $A^{(k)}$ is invertible. Then we obtain:
\[
\tr A^{-1} - \tr A^{(k)-1} = \frac{1 + r_k A^{(k)-2}c_k}{A(k,k) - r_k A^{(k)-1} c_k},
\]
where $r_k$ denotes the $k$-th row of $A$ without the $k$-th entry and $c_k$ denotes the $k$-th column of $A$ without the $k$-th entry.
\end{lemma}
\begin{proof}
We first prove the statement vor $k=n$. The Schur complement inversion formula for 
\[
A =
\begin{pmatrix}
A^{(n)} & c_n\\
r_n & A(n,n)
\end{pmatrix}
\]
 yields with $B\defeq A(n,n) - r_n A^{(n)-1} c_n\in\C$, that
 \[
 A^{-1} = 
 \begin{pmatrix}
 A^{(n)-1} + A^{(n)-1}c_n B^{-1} r_n A^{(n)-1} & - A^{(n)-1}c_n B^{-1}\\
 - B^{-1} r_n A^{(n)-1} & B^{-1}
 \end{pmatrix}.
 \]
 Therefore, since the trace is linear and only depends on the diagonal block matrices, we find
 \begin{align*}
 \tr A^{-1} - \tr A^{(n)-1} 
 & = \tr
 \begin{pmatrix}
   A^{(n)-1}c_n B^{-1} r_n A^{(n)-1} & 0\\
0 & B^{-1}
 \end{pmatrix}\\
 & = \frac{1}{B}\tr
 \begin{pmatrix}
   A^{(n)-1}c_n r_n A^{(n)-1} & 0\\
0 & 1
 \end{pmatrix}\\
 & = \frac{1}{B} \left(1 + \sum_{k,l,m=1}^{n-1} A^{(n)-1}(k,l)c_n(l)r_n(m)A^{(n)-1}(m,k)\right)\\
 & = \frac{1}{B} \left(1 + \sum_{k,l,m=1}^{n-1} r_n(m)A^{(n)-1}(m,k)A^{(n)-1}(k,l)c_n(l)\right)\\
 & = \frac{1}{B} \left(1 + r_nA^{(n)-2}c_n\right),	
 \end{align*}
which concludes the statement for $k=n$. Now if $k<n$, let $V$ be the permutation matrix as in the proof of Lemma~\ref{lem:inversediagonal}, then since $A^{(k)}=(VAV)^{(n)}$, we obtain with first part that 
\begin{align*}
\tr A^{-1} - \tr A^{(k)-1} 
&= \tr VA^{-1}V -\tr(VAV)^{(n)-1}\\
&= \tr (VAV)^{-1}-\tr(VAV)^{(n)-1}\\
&= \frac{1+r'_n(VAV)^{(n)-2}c'_n}{(VAV)(n,n) - r'_n (VAV)^{(n)-2} c'_n} 	
\end{align*}
where $r'_n$ (resp. $c'_n$) is the $n$-th row (resp. column) of $VAV$ without the $n$-th entry. This concludes the statement, since $r'_n=r_k$, $c'_n=c_k$, and $(VAV)(n,n) = A(k,k)$.
\end{proof}

\begin{corollary}
\label{cor:tracedifference}
Let $X_n$ be an Hermitian $n\times n$ matrix, $z=E+i\eta$ where $E\in\R$ and $\eta>0$, then we find for any $k\in\{1,\ldots,n\}$:
\[
\bigabs{\tr (X_n-z)^{-1} - \tr(X_n^{(k)}-z)^{-1}} \leq\frac{1}{\eta},
\]
where for all $k\in\{1,\ldots,n\}$, $X^{(k)}_n$ denotes the $k$-th principal minor of $X_n$ and $x_k$ denotes the $k$-th column of $X_n$ without the $k$-th entry.
\end{corollary}

\begin{proof}
By Lemma~\ref{lem:tracedifference}, we know that
\begin{align}
\bigabs{\tr (X_n-z)^{-1} - \tr(X_n^{(k)}-z)^{-1}} 
&= \bigabs{\frac{1 + x^*_k (X_n^{(k)}-z)^{-2}x_k}{X_n(k,k) - z- x^*_k (X_n^{(k)}-z)^{-1} x_k}}\notag\\
&\leq \frac{1+\abs{x^*_k (X_n^{(k)}-z)^{-2}x_k}}{\abs{-\eta-\Im(x^*_k (X_n^{(k)}-z)^{-1} x_k)}}\label{eq:tracediffstepone}
\end{align}
where $x_k$ denotes the $k$-th column of $X_n$ without the $k$-th entry. We also used that $X_n(k,k)\in\R$, since $X_n$ is Hermitian. We proceed by inspecting the numerator and the denominator separately. For the numerator, Let $U$ be unitary such that $U X_n^{(k)} U^* = \diag(\lambda_1,\ldots,\lambda_{n-1})=:D$, where $\lambda_1,\ldots,\lambda_{n-1}$ are the eigenvalues of  $X_n^{(k)}$. Since $X_n^{(k)}$ is Hermitian, these eigenvalues are real, and such a $U$ actually exists. Set $x_k^*U^* = (y_1,\ldots,y_{n-1})$, then we get
\begin{align*}
x^*_k (X_n^{(k)}-z)^{-2} x_k &= x_k^*(U^*(D-z)U)^{-2}x_k = x_k^*U^*(D-z)^{-2}Ux_k \\
&= \sum_{\ell=1}^{n-1}\frac{\abs{y_{\ell}}^2}{(\lambda_{\ell}-z)^2} \underset{\abs{\ldots}}{\leq}\sum_{\ell=1}^{n-1}\frac{\abs{y_{\ell}}^2}{(\lambda_{\ell}-E)^2 + \eta^2}\\
& = x_k^*[(X_n^{(k)}- E I_{n-1})^2 + v^2I_{n-1}]^{-1}x_k,
\end{align*}
where the last equality follows with
\[
[(X_n^{(k)}- E I_{n-1})^2 + v^2I_{n-1}]^{-1} = [U^*[(D-EI_{n-1})^2 + \eta^2 I_{n-1}]U]^{-1} = U^*[(D-EI_{n-1})^2 + \eta^2 I_{n-1}]^{-1}U.
\]
With the exact arguments we just used, we further obtain for the denominator in \eqref{eq:tracediffstepone} that
\begin{align*}
x^*_k (X_n^{(k)}-z)^{-1} x_k &= \sum_{\ell=1}^{n-1}\frac{\abs{y_{\ell}}^2}{\lambda_{\ell}-z} \\
&= \sum_{\ell=1}^{n-1}\frac{\abs{y_{\ell}}^2}{(\lambda_{\ell}-E)^2 + \eta^2}(\lambda_{\ell}-E) + i\sum_{\ell=1}^{n-1}\frac{\abs{y_{\ell}}^2}{(\lambda_{\ell}-E)^2 + \eta^2}\eta,
\end{align*}
so that
\[
-\eta-\Im(x^*_k (X_n^{(k)}-z)^{-1} x_k) = -\eta (1+x_k^*[(X_n^{(k)}- E I_{n-1})^2 + v^2I_{n-1}]^{-1}x_k).
\]

\end{proof}

\chapter{The Semicircle and MP Laws by the Stieltjes Transform Method}

\label{chp:StieltjesSCLMPL}

\section{General Strategy and Quadradic Form Estimates}
\label{sec:genstrat}
In this very short section we introduce a general strategy behind the proofs of limit laws in random matrix theory utilizing Stieltjes transforms. We also introduce some versatile quadratic form estimates which allow us to carry out smooth proofs of the semicircle law and Marchenko-Pastur law in the following sections. Assume that $(\sigma_n)_n$ is a sequence of ESDs of Wigner matrices and $(\mu_n)_n$ is a sequence of ESDs of MP matrices. We would like to argue that $\sigma_n \to \sigma$ or $\mu_n\to\mu^y$ weakly for some $y>0$, and in some stochastic sense, for example in probability or almost surely. To this end, we carry out the following three steps, where notationally, either $\rho_n =\sigma_n$ and $\rho=\sigma$, or $\rho_n=\mu_n$ and $\rho=\mu^y$:
\begin{enumerate}
\item We show that the Stieltjes transform $S_{\rho}$ of the limit measure $\rho$ satisfies a self-consistent quadratic equation and that the solutions can be separated so that if some $S_{\nu}$ solves the equation for some probability measure $\nu$ on $\R$ (if $\rho=\sigma$) or on $\R_+$ (if $\rho=\mu^y$), then necessarily $S_{\nu}=S_{\rho}$.
\item Applying the Schur complement formula, the Stieltjes transforms of the ESDs $\rho_n$ can be written as a sum of inverses of complex numbers. We decompose each summand into a part pertaining to the self-consistent equation (the wanted part $w$) and an error term (the remainder $r$), using
\begin{equation}
\label{eq:splitfraction}
\frac{1}{w+r} = \frac{1}{w} - \frac{r}{w(w+r)}.
\end{equation} 
We establish that if the error term converges to zero in probability resp.\ almost surely, the limit law holds in probability resp.\ almost surely.
\item We establish that the error term converges to zero almost surely by employing quadratic form estimates in combination with an estimate on the difference of Stieltjes transforms of the ESD of a random matrix and its minors. Quadratic form estimates are elementary yet very powerful. They belong to the main ingredients to prove some of the most fruitful results in contemporary random matrix theory - namely the so-called local laws, see \autocite{AnttiLLSurvey} or \autocite{Erdos}.
\end{enumerate}

For the third step, we use quadratic form estimates which can later be applied to Wigner and Marchenko-Pastur matrices. In the following, for $p\geq 1$ the norm $\norm{\cdot}_p$ \label{sym:pnorm} shall denote the $\Lcal_p(\Prob)$-seminorm, so for any random variable $Y:(\Omega,\Acal,\Prob)\longrightarrow \C$, $\norm{Y}_p = (\E\abs{Y}^p)^{1/p}$.

\begin{theorem}[Marcinkiewicz-Zygmund Inequality]
\label{thm:marcinzygmund}
If $Y_1,\ldots,Y_n$ are independent, centered and complex-valued random variables with existing absolute moments, then for every $p\geq 1$ there exists a positive constant $A_p$ which depends only on $p$, such that
\[
\bignorm{\sum_{i=1}^n Y_i}_p \leq A_p \bignorm{\left(\sum_{i=1}^n \abs{Y_i}^2\right)^{\frac{1}{2}}}_p
\]
\end{theorem}
\begin{proof}
In \autocite[386]{ChowTeicher}, the statement is proved for independent real-valued random variables. The statement is extended to the complex valued case in \autocite[33]{BaiSi}.
\end{proof}

We now formulate an important lemma, which is mainly based on Theorem~\ref{thm:marcinzygmund}.

\begin{lemma}
\label{lem:quadradicform}
 Let $Y_1,\ldots, Y_n$ be independent, centered and complex-valued random variables which are uniformly $\norm{\cdot}_p$-bounded for all $p \geq 2$. Then it holds for any complex numbers $(b_i)_{i\in\oneto{n}}$  and $(a_{i,j})_{i,j\in\oneto{n}}$
\begin{align*}
i)\quad &\forall\ p\geq 2: \bignorm{\sum_{i=1}^n b_i Y_i}_p \leq A_p \left(\sum_{i=1}^n \abs{b_i}^2 \right)^{\frac{1}{2}},\\
ii)\quad & \forall\ p\geq 2: \bignorm{\sum_{i\neq j=1}^n a_{i,j} Y_i Y_j}_p \leq A_p \left(\sum_{i\neq j=1}^n \abs{a_{i,j}}^2 \right)^{\frac{1}{2}},
\end{align*}
where $A_p$ is a constant depending only on $p$ and the uniform $\norm{\cdot}_p$-bound.
\end{lemma}
\begin{proof}
The proofs of the two statements can be found in \autocite{AnttiLLSurvey}.
\end{proof}

The following theorem establishes deviation bounds for the expressions in Lemma~\ref{lem:quadradicform}, when the constants $b_i$ and $a_{i,j}$ are replaced by functions of random variables.

\begin{theorem}
\label{thm:largedev}
Let for all $n\in\N$, $Y$ and $W$ be $n$-dependent objects ($Y=Y^{(n)}, W=W^{(n)}$) that satisfy the following for all $n\in\N$:
\begin{itemize}
\item $W=W^{(n)}$ is a finite index set.
\item $Y_W=(Y_i)_{i\in W}=(Y^{(n)}_i)_{i\in W^{(n)}}=Y^{(n)}_{W^{(n)}}$  is a family of independent, real-valued and centered random variables, so that for all $p\geq 2$, the family $(Y^{(n)}_i:i\in W^{(n)}, n\in\N)$ is uniformly $\Lcal_p$-bounded.
\end{itemize}
Further, denote for all subsets $K\subseteq W$ by $\Fcal_W(\R^K)$ the set of tuples $C=(C_i)_{i\in W}$, where for each $i\in W$, $C_i :\R^K\to \C$ is a complex-valued measurable function. Analogously, define for all subsets $K\subseteq W$ by $\Fcal_{W\times W}(\R^K)$ the set of tuples $C=(C_{i,j})_{i,j\in W}$, where for all $i,j\in W$, $C_{i,j} :\R^K\to \C$ is a complex-valued measurable function. 
Then we obtain the following probability bounds:
\begin{enumerate}[i)]
\item For all $\varepsilon,D>0$ there is a constant $C_{\varepsilon,D}\geq 0$, such that for all $n\in\N$, all disjoint subsets $I,K\subseteq W$ and all function tuples $B\in\Fcal_W(\R^K)$ the following holds:
\[
\Prob\left(\bigabs{\sum\limits_{i\in I} B_i[Y_K] Y_i} \ >\ n^{\varepsilon}\,\cdot\,  \sqrt{\sum\limits_{i\in I} \abs{B_i[Y_K]}^2}\right) \leq \frac{C_{\varepsilon,D}}{n^D}.
\]

\item For all $\varepsilon,D>0$ there is a constant $C_{\varepsilon,D}\geq 0$, such that for all $n\in\N$, all disjoint subsets $I,K\subseteq W$, and all function tuples $A\in\Fcal_{W\times W}(\R^K)$ the following holds:
\[
\Prob\left(\bigabs{\sum\limits_{i,j\in I, i\neq j} Y_i A_{i,j}[Y_K] Y_j} \ >\ n^{\varepsilon}\,\cdot\,   \sqrt{\sum\limits_{i,j\in I,i\neq j} \abs{A_{i,j}[Y_K]}^2}\right) \leq \frac{C_{\varepsilon,D}}{n^D}.
\]
\end{enumerate}
\end{theorem}
\begin{proof}
We only prove $ii)$, since $i)$ can be proved analogously. Let $\varepsilon,D>0$ be arbitrary and choose $p\in \N$ with $p\geq 2$ so large that $p\varepsilon>D$. Then we pick an $n\in\N$, disjoint subsets $I,K\subseteq W^{(n)}$ and a function tuple $A\in \Fcal_{W\times W}(\R^K)$. To avoid division by zero, we define the set:
\[
\Acal_{2}\defeq\left\{y_K\in\R^K \,\vert\, \sum\limits_{i,j\in I,i\neq j} \abs{A_{i,j}[y_K]}^2 > 0\right\}.
\]
Then we conduct the following calculation (explanations are found below the calculation; the sums over "$i\neq j$" are over all $i,j\in I$ with $i\neq j$):

\begin{align*}
&\Prob\left(
\bigabs{\sum_{i\neq j} Y_i A_{i,j}[Y_K] Y_j} > n^{\varepsilon} \left(\sum_{i\neq j}\abs{A_{i,j}[Y_K]}^2\right)^{\frac{1}{2}}
\right)\\
& =\Prob\left(
\bigabs{\frac{\sum_{i\neq j} Y_i A_{i,j}[Y_K] Y_j}{\left(\sum_{i\neq j}\abs{A_{i,j}[Y_K]}^2\right)^{\frac{1}{2}}}}^p \one_{\Acal_{2}}(Y_K)> n^{p\varepsilon}
\right)\\
& \leq\frac{1}{n^{p\varepsilon}} \E{\bigabs{\frac{\sum_{i\neq j} Y_i A_{i,j}[Y_K] Y_j}{\left(\sum_{i\neq j}\abs{A_{i,j}[Y_K]}^2\right)^{\frac{1}{2}}}}^p}\one_{\Acal_{2}}(Y_K)\\
& =\frac{1}{n^{p\varepsilon}} \int_{\R^K}\int_{\R^I} \bigabs{\frac{\sum_{i\neq j} y_i A_{i,j}[y_K] y_j}{\left(\sum_{i\neq j}\abs{A_{i,j}[y_K]}^2\right)^{\frac{1}{2}}}}^p\text{d} \Prob^{Y_I}(y_I) \one_{\Acal_{2}}(y_K)\text{d} \Prob^{Y_K}(y_K)\\
& \leq \frac{(A_p)^p}{n^{p\varepsilon}} \leq \frac{(A_p)^p}{n^D}
 \end{align*}
where the first step follows from the fact that for 
\[
\bigabs{\sum_{i\neq j} Y_i A_{i,j}[Y_K] Y_j} > n^{\varepsilon} \left(\sum_{i\neq j}\abs{A_{i,j}[Y_K]}^2\right)^{\frac{1}{2}}
\]
to hold, not all $A_{i,j}[Y_K]$ may vanish, in the second step we used Markov's inequality,
in the third step we used Fubini, in the fourth step we applied Lemma~\ref{lem:quadradicform} and in the last step we used the choice of $p$ in the beginning of the proof. Note that $(A_p)^p$
 denotes a constant which depends only on $p$, which in turn depends only on the choices of $\varepsilon$ and $D$. In particular, this constant does not depend on the choice of $n\in\N$, the sets $I$ and $K$ or the function tuple $A$.  This shows $ii)$.
 \end{proof}

\section{The Semicircle Law}
\label{sec:SemicircleByStieltjes}
We follow the general strategy outlined in Section~\ref{sec:genstrat}.

\subsection*{Step 1: Self-Consistent Equation and Separation of Solutions.}

The first step of the proof consists of the following lemma:
\begin{lemma}
\label{lem:factsSCD}
The Stieltjes transform of the semicircle distribution $\sigma$ is given by
\[
\forall\, z\in\C_+: S_{\sigma}(z) = \frac{-z+\sqrt{z^2-4}}{2}.
\]
Consider the equation equation in $m\in\C$, where $z\in\C_+$ is fixed:
\begin{equation}
\label{eq:sceSCD} 
m=\frac{1}{-z-m}	
\end{equation}
Then the following statements hold:
\begin{enumerate}[i)]
\item The solutions to \eqref{eq:sceSCD} are given by
\[
m_{+,-} = \frac{-z \pm \sqrt{z^2-4}}{2}.
\]
\item $S_{\sigma}(z)$ is the positive branch of the solution in \eqref{eq:sceSCD}, that is, $S_{\sigma}(z)=m_+$.
\item For the denominator in \eqref{eq:sceSCD} it holds $\Im(-z-m_{-}) \geq -\frac{1}{2}\Im(z)$
\item If $\nu\in\Mcal_1(\R)$, then for all $z\in\C_+$ it holds
\[
\Im(-z-S_{\nu}(z)) \leq -\Im(z).
\]
In particular, if $S_{\nu}(z)$ satisfies \eqref{eq:sceSCD}, we must have $S_{\nu}(z)=S_{\sigma}(z)$.
\end{enumerate}
\end{lemma}
\begin{proof}
The Stieltjes transform of the semicircle distribution is derived in Lemma 2.11 in \autocite{BaiSi}. 	Statements $i)$ and $ii)$ can be shown directly by solving the quadratic equation \eqref{eq:sceSCD}. Statement $iii)$ follows since
\[
\Im\left(-z - \frac{-z - \sqrt{z^2-4}}{2}\right)= -\Im(z) + \frac{\Im(z)}{2} + \frac{\Im\sqrt{z^2-4}}{2} \geq - \frac{\Im(z)}{2},
\]
since we defined the complex square root $\sqrt{\cdot}$ to be the square root with non-negative imaginary part. Statement $iv)$ follows trivially since $\Im S_{\nu}(z)\geq 0$.
\end{proof}

\subsection*{Step 2: Derivation of the Error Term}

By Corollary~\ref{cor:inversediagonal}, the Stieltjes transform $s_n$ of a Wigner matrix $\frac{1}{\sqrt{n}}X_n$ is given by
\begin{equation}
\label{eq:snWIG}
s_n(z) = \frac{1}{n}\sum_{k\in\oneto{n}} \frac{1}{\frac{1}{\sqrt{n}}X_n(k,k)-z-\frac{1}{n} x^T_k\left(\frac{1}{\sqrt{n}}X_n^{(k)}-z\right)^{-1}x_k},
\end{equation}
where $X_n^{(k)}$ denotes the $k$-th principle minor of $X_n$ and $x_k$ the $k$-th column of $X_n$ without the $k$-th entry.
The desired denominator in each summand of \eqref{eq:snWIG} is 
\[
 - z - s_n(z),
\]
stemming from the self-consistent equation \eqref{eq:sceSCD}. In the $k$-th summand for $k\in\{1,\ldots,n\}$, we obtain the remainder term
\[
\Omega_n^{(k)}(z) = \frac{1}{\sqrt{n}}X_n(k,k) + s_n(z) -\frac{1}{n}x_k^T\left(\frac{1}{\sqrt{n}}X_n^{(k)}-z\right)^{-1}x_k.
\]
Using \eqref{eq:splitfraction}, we conclude
\begin{equation}
\label{eq:snsplitWIG}
s_n(z) = \frac{1}{- z - s_n(z)} - \delta_n(z)
\end{equation}
with
\[
\delta_n(z) = \frac{1}{n} \sum_{k\in\oneto{n}} \frac{\Omega_n^{(k)}(z)}{(-z-s_n(z))(-z- s_n(z)+\Omega^{(k)}_n(z))}.
\]
The error term $\delta_n(z)$ can be bounded in absolute terms as follows: By Lemma~\ref{lem:factsSCD}, we find
\[
\forall\, n\in \N: \Im(-z-s_n(z))\leq -\Im(z).
\]
If we assume 
\begin{equation}
\label{eq:omegasmallWIG}	
\max_k\abs{\Omega^{(k)}_n(z)} \leq \frac{1}{2}\Im(z)
\end{equation}
we may therefore conclude
\begin{align}
\abs{\delta_n(z)} &= \bigabs{\frac{1}{n} \sum_{k\in\oneto{n}} \frac{\Omega_n^{(k)}(z)}{(-z-s_n(z))(-z- s_n(z)+\Omega^{(k)}_n(z))}}\notag\\
&\leq
\frac{1}{n} \sum_{k\in\oneto{n}}\frac{\abs{\Omega_n^{(k)}}}{\Im(z)^2/2} \ \leq\  \frac{2}{\Im(z)^2} \max_k\abs{\Omega^{(k)}_n}.\label{eq:deltaboundWIG}
\end{align}

The following lemma puts our findings into perspective:
\begin{theorem}
\label{thm:reductionWIG}
In above situation, the following statements hold for any fixed $z\in\C_+$:
\begin{enumerate}[i)]
\item If in \eqref{eq:snsplitWIG}, $\delta_n(z)\to 0$ in probability resp.\ almost surely, then $s_n(z)\to s(z)$ in probability resp.\ almost surely.
\item If $\max_{k\in\oneto{n}}\abs{\Omega_n^{(k)}(z)}\xrightarrow[n\to\infty]{} 0$ in probability resp.\ almost surely, then $\delta_n(z)\to 0$ in probability resp.\ almost surely.
\end{enumerate}
\end{theorem}
\begin{proof}
Statement $ii)$ follows with \eqref{eq:deltaboundWIG}, using \eqref{eq:omegasmallWIG}. 
We proceed to show statement $i)$ in the almost sure sense.
Fix $z\in\C_+$. Let $A$ be a measurable set with $\Prob(A)=1$, on which $\delta_n(z) \to 0$. Let $\omega\in A$ be arbitrary, and denote by $s^{\omega}_n(z)$ the realization of $s_n(z)$ at $\omega$. To show that $s^{\omega}_n(z)$ converges to $s(z)$, we show that any subsequence of $s^{\omega}_n(z)$ contains another subsequence that converges to $s(z)$. To this end, let $J\subseteq\N$ be a subsequence. Then $(s^{\omega}_n(z))_{n\in J}$ is a bounded sequence of complex numbers (with absolute bound $\Im(z)^{-1}>0$), therefore has a convergent subsequence $(s^{\omega}_n(z))_{n\in I}$, $I\subseteq J$, with some limit $t\in\C$ (Bolzano-Weierstrass). Considering \eqref{eq:snsplitWIG}, $t$ satisfies
\[
t = \frac{1}{-z-t}.
\]
Since $\Im(- z - s_n^{\omega}(z)) \leq -\Im(z)$ by Lemma~\ref{lem:factsSCD}, we find $\Im(-z-t) \leq -\Im(z)$, so $t=s(z)$ by Lemma~\ref{lem:factsSCD}. We have seen that any subsequence of $(s^{\omega}_n(z))_{n\in\N}$ has a subsequence which converges to $s(z)$. Therefore, $s_n(z)\to s(z)$ on $A$, that is, almost surely. Statement $i)$ in probability follows from the almost sure version we just proved, using Lemma~\ref{lem:subsequence}: To show that $s_n(z) \to s(z)$ in probability, it suffices to show that for any subsequence $I\subseteq\N$ there is a subsequence $J\subseteq I$ such that $s_n(z)\to s(z)$ for $n\in J$. So let $I\subseteq \N$ be an arbitrary subsequence. Since $\delta_n(z)\to 0$ in probability, there is a subsequence $J\subseteq I$ with $\delta_n(z)\to 0$ almost surely for $n\in J$. But then $s_n(z)\to s(z)$ for $n\in J$ almost surely as we just proved above. This completes the argument.
\end{proof}

\subsection*{Step 3: Analysis of the Error Term}

By Theorem~\ref{thm:reductionWIG}, it suffices to show that $\max_{k\in\oneto{n}}\abs{\Omega_n^{(k)}(z)}\xrightarrow[n\to\infty]{} 0$ in probability or almost surely, where
\begin{equation}
\label{eq:remainderWIG}	
\Omega_n^{(k)}(z) = \frac{1}{\sqrt{n}}X_n(k,k) + s_n(z) -\frac{1}{n}x_k^T\left(\frac{1}{\sqrt{n}}X_n^{(k)}-z\right)^{-1}x_k.
\end{equation}
In this subsection, we will show almost sure convergence. Note that
\begin{align}
\Omega_n^{(k)}(z)
&= \frac{1}{\sqrt{n}}X_n(k,k)\notag\\
&\quad\,  - \frac{1}{n} \sum_{i\neq j}^n x_k(i) \left(\frac{1}{\sqrt{n}}X_n^{(k)}-z\right)^{-1}(i,j) x_k(j)\notag\\
&\quad\, - \frac{1}{n} \sum_{i=1}^n (x_k(i)^2-1) \left(\frac{1}{\sqrt{n}}X_n^{(k)}-z\right)^{-1}(i,i)\notag\\
&\quad  - \frac{1}{n} \tr \left(\frac{1}{\sqrt{n}}X_n^{(k)}-z\right)^{-1} + s_n(z)\notag\\
&=: A(n,k) + B(n,k,z)+ C(n,k,z) + D(n,k,z) \label{eq:ABCD-WIG}.
\end{align}

We will analyze these four terms separately and show that their maxima over $k\in\{1,\ldots,n\}$ converge to zero almost surely. For $B$ and $C$ we will use Theorem~\ref{thm:largedev}, whereas for $D$ we will use Corollary~\ref{cor:tracedifference}. 

\begin{lemma}
\label{lem:SummandA-WIG}
In \eqref{eq:ABCD-WIG}, $\max_{k\in\oneto{n}}\abs{A(n,k)}\to 0$ almost surely as $n\to\infty$.
\end{lemma}
\begin{proof}
Let $C_8$ be an upper bound of $\E (X_n(i,j))^8$ for all $n$, $i$, $j$, then we find
\[
\forall\, n\in\N: \forall\, k\in\oneto{n}:\ \Prob\left(\bigabs{\frac{1}{\sqrt{n}}X_n(k,k)} > \frac{1}{n^{\frac{1}{8}}}\right) 
=\Prob\left(\abs{X_n(k,k)}^8 > \frac{n^4}{n}\right)
  \leq \frac{C_8}{n^3}.
\]
Therefore, taking the union bound, we obtain for all $n\in\N$:
\[
\Prob\left(\max_{k
\in\oneto{n}}\bigabs{\frac{1}{\sqrt{n}}X_n(k,k)} > \frac{1}{n^{\frac{1}{8}}}\right)\leq\sum_{k\in\oneto{n}}\Prob\left(\bigabs{\frac{1}{\sqrt{n}}X_n(k,k)} > \frac{1}{n^{\frac{1}{8}}}\right)\leq\frac{nC_{8}}{n^3}
\]
which converges to zero summably fast. This concludes the proof by Borel-Cantelli. 
\end{proof}
For $B(n,k,z)$ we define the terms
\[
S(n,k,z) \defeq
\sum_{i\neq j}^n x_k(i) \left(\frac{1}{\sqrt{n}}X_n^{(k)}-z\right)^{-1} (i,j) x_k(j)	
\]
and 
\[
R(n,k,z) \defeq \sqrt{\sum_{i\neq j}^n \bigabs{\left(\frac{1}{\sqrt{n}}X_n^{(k)}-z\right)^{-1}(i,j)}^2}
\]	
to employ Theorem~\ref{thm:largedev} $ii)$. To bound $R(n,k,z)$, we use the following trivial lemma:
\begin{lemma}
\label{lem:RnBoundWIG}
	Let $X$ be an Hermitian $n\times n$ matrix and $z\in\C_+$. Then
\[
\sqrt{\sum_{i, j\in\oneto{n}}\abs{(X-z)^{-1}(i,j)}^2}\leq \frac{\sqrt{n}}{\Im(z)}.
\]
\end{lemma}
\begin{proof}
For $n\times n$ matrices $X$, the general inequality $\norm{X}_{F}\leq\sqrt{n}\norm{X}_{op}$ holds, where $\norm{\cdot}_F$ is the Frobenius norm and $\norm{\cdot}_{op}$ is the operator norm. On the other hand, for any \emph{Hermitian} $n\times n$ matrix $X$, $\norm{(X-z)^{-1}}_{op}\leq\Im(z)^{-1}$. Combining these facts, we obtain 
\[
\sqrt{\sum_{i, j\in\oneto{n}}\abs{(X-z)^{-1}(i,j)}^2} = \norm{(X-z)^{-1}}_{F} \leq \sqrt{n} \norm{(X-z)^{-1}}_{op} \leq \frac{\sqrt{n}}{\Im(z)}.
\]
\end{proof}

\begin{lemma}
\label{lem:SummandB-WIG}
In \eqref{eq:ABCD-WIG}, $\max_{k\in\oneto{n}}\abs{B(n,k,z)}\to 0$ almost surely as $n\to\infty$.
\end{lemma}
\begin{proof}
Using the terms $S(n,k,z)$ and $R(n,k,z)$ defined above, we	find by Lemma~\ref{lem:RnBoundWIG} that
\[
\abs{R(n,k,z)} \leq \frac{\sqrt{n}}{\Im(z)},
\]
Therefore, using Theorem~\ref{thm:largedev} $ii)$ with $\varepsilon=1/4$ and $D=3$, we obtain a constant $C_{\frac{1}{4},3}\geq 0$, such that for all $n\in\N$ and all $k\in\oneto{n}$:
\[
\Prob\left(\abs{B(n,k,z)} > \frac{n^{\frac{1}{4}}\sqrt{n}}{n\Im(z)}\right) \leq \Prob\left(\abs{S(n,k,z)}>n^{\frac{1}{4}}R(n,k,z)\right) \leq \frac{C_{\frac{1}{4},3}}{n^3}.
\]
Applying the union bound as in the proof of Lemma~\ref{lem:SummandA-WIG} concludes the statement.
\end{proof}
For $C(n,k,z)$, we define the terms
\[
S'(n,k,z) \defeq
\sum_{i\in\oneto{n}} (x_k(i)^2-1) \left(\frac{1}{\sqrt{n}}X_n^{(k)}-z\right)^{-1}(i,i) 	
\]
and 
\[
R'(n,k,z) \defeq \sqrt{\sum_{i\in\oneto{n}} \bigabs{\left(\frac{1}{\sqrt{n}}X_n^{(k)}-z\right)^{-1}(i,i)}^2}
\]	
to employ Theorem~\ref{thm:largedev} $i)$.

\begin{lemma}
\label{lem:SummandC-WIG}
In \eqref{eq:ABCD-WIG}, $\max_{k\in\oneto{n}}\abs{C(n,k,z)}\to 0$ almost surely as $n\to\infty$.
\end{lemma}
\begin{proof}
Using the terms $S'(n,k,z)$ and $R'(n,k,z)$ defined above, we	find by Lemma~\ref{lem:RnBoundWIG} that
\[
\abs{R'(n,k,z)} \leq \frac{\sqrt{n}}{\Im(z)},
\]
Therefore, using Theorem~\ref{thm:largedev} $i)$ with $\varepsilon=1/4$ and $D=3$, we obtain a constant $C_{\frac{1}{4},3}\geq 0$, such that for all $n\in\N$ and all $k\in\oneto{n}$:
\[
\Prob\left(\abs{C(n,k,z)} > \frac{n^{\frac{1}{4}}\sqrt{n}}{n}\right) \leq \Prob\left(\abs{S'(n,k,z)}>n^{\frac{1}{4}}R'(n,k,z)\right) \leq \frac{C_{\frac{1}{4},3}}{n^3}.
\]
Applying the union bound as in the proof of Lemma~\ref{lem:SummandA-WIG} concludes the statement.
\end{proof}

\begin{lemma}
\label{lem:SummandD-WIG}
In \eqref{eq:ABCD-WIG}, $\max_{k\in\oneto{n}}\abs{D(n,k,z)}\to 0$ almost surely as $n\to\infty$.
\end{lemma}
\begin{proof}
With Corollary~\ref{cor:tracedifference}, we obtain for any $n\in\N$ and $k\in\oneto{n}$ that $\abs{D(n,k,z)} \leq (n\Im(z))^{-1}$, so that
\[
\max_{k\in\oneto{n}}\abs{D(n,k,z)} \leq \frac{1}{n\Im(z)} \xrightarrow[n\to\infty]{} 0 \quad \text{almost surely.}
\]	
\end{proof}

\begin{theorem}
\label{thm:sufficient-WIG}
In above situation, we find for any fixed $z\in\C_+$ that 
\[
\max_{k\in\oneto{n}}\abs{\Omega_n^{(k)}(z)}\xrightarrow[n\to\infty]{} 0 \quad \text{almost surely}.
\]
\end{theorem}
\begin{proof}
This follows directly by the decomposition \eqref{eq:ABCD-WIG} with Lemma~\ref{lem:SummandA-WIG}, Lemma~\ref{lem:SummandB-WIG}, Lemma~\ref{lem:SummandC-WIG} and Lemma~\ref{lem:SummandD-WIG}.
\end{proof}

\section{The Marchenko-Pastur Law}
\label{sec:MPByStieltjes}
Again, we follow the general strategy outlined in Section~\ref{sec:genstrat}.

\subsection*{Step 1: Self-Consistent Equation and Separation of Solutions.}

The first step of the proof consists of the following lemma:
\begin{lemma}
\label{lem:factsMPD}
Fix $y>0$. The Stieltjes transform of the Marchenko-Pastur distribution $\mu^y$ is given by
\[
\forall\ z\in\C_+:~S_{\mu^{y}}(z) = \frac{1-y-z+\sqrt{(z-1-y)^2-4y}}{2yz}.
\]
Consider the equation in $m\in\C$, where $z\in\C_+$ is fixed:
\begin{equation}
\label{eq:sceMPD} 
m = \frac{1}{1-z-y-yzm}	
\end{equation}
Then the following statements hold:
\begin{enumerate}[i)]
\item The solutions to \eqref{eq:sceMPD} are given by
\[
m_{+,-} = \frac{1-y-z\pm \sqrt{(1-y-z)^2-4yz}}{2yz}.
\]
\item $S_{\mu^{y}}(z)$ is the positive branch of the solutions to \eqref{eq:sceMPD}, that is, $S_{\mu^{y}}(z) = m_+$.	
\item For the denominator in \eqref{eq:sceMPD} it holds $\Im(1-z-y-yzm_{-})\geq -\frac{1}{2}\Im(z)$.
\item If $\nu\in\Mcal_1([0,\infty))$, then for all $z\in\C_+$ we find
\[
\Im(1-z-y-yz S_{\nu}(z)) \leq -\Im(z)
\]
In particular, if $S_{\nu}(z)$ satisfies \eqref{eq:sceMPD}, we must have $S_{\nu}(z)=S_{\mu^y}(z)$.
\end{enumerate}
\end{lemma}

\begin{proof}
The Stieltjes transform $S_{\mu^y}$ is derived in Lemma 3.11 in \autocite{BaiSi}.
Statement $i)$ is verified by solving the quadratic equation \eqref{eq:sceMPD}. For $ii)$, we calculate $(1-y-z)^2 - 4yz = z^2 - 2yz + y^2 -2y-2z +1 = (z-y-1)^2-4y$. For $iii)$, we calculate
\begin{align*}
&\Im(1-z-y-yzm_{-})\\
&=\Im\left(1-z-y-yz\frac{1-y-z-\sqrt{(1-y-z)^2-4yz}}{2yz}\right)\\
&=\Im\left(\frac{1-y-z+\sqrt{(1-y-z)^2-4yz}}{2}\right)\\
&=\frac{1}{2}\left(-\Im(z)+\underbrace{\Im\sqrt{(1-y-z)^2-4yz}}_{\geq 0 \text{ per definition of } \sqrt{\cdot}}\right) \geq -\frac{\Im(z)}{2}.
\end{align*}
For $iv)$, note that with $z=E+i\eta$, where $E\in\R$ and $\eta>0$, we find
\begin{align*}	
\Im(z S_{\nu}(z)) &= \Re(z)\Im S_{\nu}(z) + \Im(z)\Re S_{\nu}(z)\\
&= E\int\frac{\eta}{(x-E)^2+\eta^2}\nu(\de x) + \eta \int \frac{x-E}{(x-E)^2+\eta^2}\nu(\de x)\\
&=\eta \int\frac{x}{(x-E)^2+\eta^2}\nu(\de x).
\end{align*}
Therefore,
\begin{align*}
\Im(1-z-y-yz S_{\nu}(z)) & = -\eta - y\eta \int_{[0,\infty)}\frac{x}{(x-E)^2 + \eta^2}\nu(\de x)\\
&= -\eta\left(1+y\int_{[0,\infty)}\frac{x}{(x-E)^2+\eta^2}\nu(\de x)\right)\\
& \leq -\eta.	
\end{align*}
The very last statement follows with part $iii)$.

\end{proof}

\subsection*{Step 2: Derivation of the Error Term}

By Corollary~\ref{cor:inversediagonal}, the Stieltjes transform $s_n$ of an MP matrix $\frac{1}{n}X_nX_n^T$ is given by
\begin{equation}
\label{eq:snMP}
s_n(z) = \frac{1}{p}\sum_{k=1}^p \frac{1}{\frac{1}{n}\alpha_k^T\alpha_k-z-\frac{1}{n^2}\alpha_k^T X_n^{(k)T}\left(\frac{1}{n}X_n^{(k)}X_n^{(k)T}-z\right)^{-1}X_n^{(k)}\alpha_k},
\end{equation}
where $\alpha_k^T$ is the $k$-th row of $X_n$ (note that $\alpha_k$ also depends on $n$, which we drop from the notation), $X_n^{(k)}$ is $X_n$ with $k$-th row removed (thus a $(p-1)\times n$-matrix). 
The desired denominator in each summand of \eqref{eq:snMP} is 
\[
1 - z - y_n - y_n z s_n(z),
\]
stemming from the self-consistent equation \eqref{eq:sceMPD}, where $y_n\defeq p/n$ and it is assumed that there exists a $y\in(0,\infty)$ such that $y_n\to y$. (It is favorable to work with $y_n$ instead of $y$, since this leads to a cancellation within the error term $\Omega_n^{(k)}(z)$ we define below, see the proof of Lemma~\ref{lem:SummandD-MP} below.) In the $k$-th summand for $k\in\{1,\ldots,p\}$, we obtain the remainder term
\[
\Omega_n^{(k)}(z) = \frac{1}{n}\alpha_k^T\alpha_k - 1 - \frac{1}{n^2}\alpha_k^T X_n^{(k)T}\left(\frac{1}{n}X_n^{(k)}X_n^{(k)T}-z\right)^{-1}X_n^{(k)}\alpha_k + y_n + y_n z s_n(z).
\]
Using \eqref{eq:splitfraction}, we conclude
\begin{equation}
\label{eq:snsplitMP}
s_n(z) = \frac{1}{1 - z - y_n - y_n z s_n(z)} - \delta_n(z)
\end{equation}
with
\[
\delta_n(z) = \frac{1}{p} \sum_{k=1}^p \frac{\Omega_n^{(k)}(z)}{(1-z-y_n-y_nzs_n(z))(1-z-y_n-y_n z s_n(z)+\Omega^{(k)}_n(z))}.
\]
The error term $\delta_n(z)$ can be bounded in absolute terms as follows: By Lemma~\ref{lem:factsMPD}, we find
\[
\forall\, n\in \N: \Im(1-z-y_n-y_n z s_n(z))\leq -\Im(z).
\]
If we assume 
\begin{equation}
\label{eq:omegasmallMP}	
\max_{k\in\oneto{p}}\abs{\Omega^{(k)}_n(z)} \leq \frac{1}{2}\Im(z)
\end{equation}
we may therefore conclude
\begin{align}
\abs{\delta_n(z)} &= \bigabs{\frac{1}{p} \sum_{k=1}^p \frac{\Omega_n^{(k)}(z)}{(1-z-y_n-y_nzs_n(z))(1-z-y_n-y_n z s_n(z)+\Omega^{(k)}_n(z))}}\notag\\
&\leq
\frac{1}{p} \sum_{k=1}^p\frac{\abs{\Omega_n^{(k)}}}{\Im(z)^2/2} \ \leq\  \frac{2}{\Im(z)^2} \max_{k\in\oneto{p}}\abs{\Omega^{(k)}_n}.\label{eq:deltaboundMP}
\end{align}

The following lemma puts our findings into perspective:
\begin{theorem}
\label{thm:reductionMPL}
In above situation, the following statements hold for any fixed $z\in\C_+$:
\begin{enumerate}[i)]
\item If in \eqref{eq:snsplitMP}, $\delta_n(z)\to 0$ in probability resp.\ almost surely, then $s_n(z)\to s(z)$ in probability resp. almost surely.
\item If $\max_{k\in\oneto{p}}\abs{\Omega_n^{(k)}(z)}\xrightarrow[n\to\infty]{} 0$ in probability resp.\ almost surely, then $\delta_n(z)\to 0$ in probability resp.\ almost surely.
\end{enumerate}
\end{theorem}
\begin{proof}
Statement $ii)$ follows with \eqref{eq:deltaboundMP}, using \eqref{eq:omegasmallMP}. 
We proceed to show statement $i)$ in the almost sure sense.
Fix $z\in\C_+$. Let $A$ be a measurable set with $\Prob(A)=1$, on which $\delta_n(z) \to 0$. Let $\omega\in A$ be arbitrary, and denote by $s^{\omega}_n(z)$ the realization of $s_n(z)$ at $\omega$. To show that $s^{\omega}_n(z)$ converges to $s(z)$, we show that any subsequence of $s^{\omega}_n(z)$ contains another subsequence that converges to $s(z)$. To this end, let $J\subseteq\N$ be a subsequence. Then $(s^{\omega}_n(z))_{n\in J}$ is a bounded sequence of complex numbers (with absolute bound $\Im(z)^{-1}>0$), therefore has a convergent subsequence $(s^{\omega}_n(z))_{n\in I}$, $I\subseteq J$, with some limit $t\in\C$ (Bolzano-Weierstrass). Also, as $n\to\infty$ we find $y_n\to y$. Therefore, considering \eqref{eq:snsplitMP}, $t$ satisfies
\[
t = \frac{1}{1-z-y-yzt}.
\]
For all realizations of the ESDs $\mu_n$ of $\frac{1}{n}X_nX_n^T$, $\mu_n([0,\infty))=1$, since the matrix has only non-negative spectrum. Therefore, by Lemma~\ref{lem:factsMPD},
\[
\forall\, n\in I: \Im(1-z-y-yz s^{\omega}_n(z))\leq -\Im(z),
\]
so also $\Im(1-z-y-yz t)\leq -\Im(z)$, hence $t=s(z)$ by Lemma~\ref{lem:factsMPD}. We have seen that any subsequence of $(s^{\omega}_n(z))_{n\in\N}$ has a subsequence which converges to $s(z)$. Therefore, $s_n(z)\to s(z)$ on $A$, that is, almost surely. The statement about convergence in probability can be proved verbatim as in the proof of Theorem~\ref{thm:reductionWIG}.
\end{proof}

\subsection*{Step 3: Analysis of the Error Term}

By Theorem~\ref{thm:reductionMPL}, it suffices to show that $\max_{k\in\oneto{p}}\abs{\Omega_n^{(k)}(z)}\xrightarrow[n\to\infty]{} 0$ in probability or almost surely, where
\begin{equation}
\label{eq:remainderMP}	
\Omega_n^{(k)}(z) = \frac{1}{n}\alpha_k^T\alpha_k - 1 - \frac{1}{n^2}\alpha_k^T X_n^{(k)T}\left(\frac{1}{n}X_n^{(k)}X_n^{(k)T}-z\right)^{-1}X_n^{(k)}\alpha_k + y_n + y_n z s_n(z).
\end{equation}
In this subsection, we show almost sure convergence. Note that
\begin{align}
\Omega_n^{(k)}(z)
&= \frac{1}{n}\alpha_k^T\alpha_k - 1\notag\\
&\quad\,  - \frac{1}{n^2} \sum_{i\neq j}^n \alpha_k(i) \left[ X_n^{(k)T}\left(\frac{1}{n}X_n^{(k)}X_n^{(k)T}-z\right)^{-1} X_n^{(k)}\right](i,j) \alpha_k(j)\notag\\
&\quad\, - \frac{1}{n^2} \sum_{i=1}^n (\alpha_k(i)^2-1) \left[ X_n^{(k)T}\left(\frac{1}{n}X_n^{(k)}X_n^{(k)T}-z\right)^{-1} X_n^{(k)}\right](i,i)\notag\\
&\quad  - \frac{1}{n^2} \tr \left[ X_n^{(k)T}\left(\frac{1}{n}X_n^{(k)}X_n^{(k)T}-z\right)^{-1} X_n^{(k)}\right] + y_n + y_n z s_n(z) + y_n + y_n z s_n(z)\notag\\
&=: A(n,k) + B(n,k,z)+ C(n,k,z) + D(n,k,z) \label{eq:ABCD-MP}.
\end{align}

We will analyze these four terms separately and show that their maxima over $k\in\{1,\ldots,p\}$ converge to zero almost surely. For $A$, $B$ and $C$ we will use Theorem~\ref{thm:largedev}, whereas for $D$ we will use Corollary~\ref{cor:tracedifference}. 

\begin{lemma}
\label{lem:SummandA-MP}
In \eqref{eq:ABCD-MP}, $\max_{k\in\oneto{p}}\abs{A(n,k)}\to 0$ almost surely as $n\to\infty$.
\end{lemma}
\begin{proof}
We employ Theorem~\ref{thm:largedev} $i)$ with $B_i\equiv 1$, $\varepsilon=1/4$ and $D=3$ to obtain a constant $C_{\frac{1}{4},3}\geq 0$ such that
\[
\forall\, n\in\N: \forall\, k\in\oneto{p}:\ \Prob\left(\bigabs{\frac{1}{n}\sum_{i\in\oneto{n}} (X_n(k,i)^2-1)} > \frac{n^{1/4}\sqrt{n}}{n}\right)\leq \frac{C_{\frac{1}{4},3}}{n^3}.
\]
Therefore, taking the union bound, we obtain for all $n\in\N$:
\[
\Prob\left(\max_{k
\in\oneto{p}}\bigabs{\frac{1}{n}\sum_{i\in\oneto{n}} (X_n(k,i)^2-1)} > \frac{n^{1/4}\sqrt{n}}{n}\right)\leq\frac{pC_{1/4,3}}{n^3}
\]
which converges to zero summably fast. This concludes the proof by Borel-Cantelli. 
\end{proof}
For $B(n,k,z)$, we define the terms
\[
S(n,k,z) \defeq
\sum_{i\neq j}^n \alpha_k(i) \left[X_n^{(k)T}\left(\frac{1}{n}X_n^{(k)}X_n^{(k)T}-z\right)^{-1} X_n^{(k)}\right](i,j) \alpha_k(j)	
\]
and 
\[
R(n,k,z) \defeq \sqrt{\sum_{i\neq j}^n \bigabs{\left[X_n^{(k)T}\left(\frac{1}{n}X_n^{(k)}X_n^{(k)T}-z\right)^{-1} X_n^{(k)}\right](i,j)}^2}
\]	
to employ Theorem~\ref{thm:largedev} $ii)$. To bound $R(n,k,z)$, we formulate the following lemma, which is taken from \cite{Fleermann:Heiny:2020}:
\begin{lemma}
\label{lem:RnBoundMP}
	Let $X$ be a $p\times n$ matrix with real-valued entries, $z\in\C_+$. Define
\begin{equation}
\label{eq:Fdefined}
F(X)\defeq X^T\left(\frac{1}{n}XX^T-z\right)^{-1}X.
\end{equation}
Then we obtain the following bound:
\[
\quad\sqrt{\sum_{i, j\in\oneto{n}}\abs{F_{ij}(X)}^2}\leq n\sqrt{p} \left(1+\frac{\abs{z}}{\Im(z)}\right)
\]
\end{lemma}
\begin{proof}
We recall that
\begin{enumerate}[a)]
\item $\text{Spectrum}(X^T(XX^T-z)^{-1}X)\cup\{0\} = \text{Spectrum}((XX^T-z)^{-1}XX^T)\cup\{0\}$,
\item $(XX^T-z)^{-1}XX^T= I + z (XX^T-z)^{-1}$,
\end{enumerate}
and that $\norm{\cdot}_F\leq\sqrt{m}\norm{\cdot}_{op}$ for $m\times m$ matrices, where $\norm{\cdot}_F$ denotes the Frobenius norm and $\norm{\cdot}_{op}$ denotes the operator norm. Therefore,

\begin{align*}
&\sqrt{\sum_{i,j\in\oneto{n}}\abs{F_{ij}(X)}^2} = n \bignorm{\frac{1}{n}X^T\left(\frac{1}{n}XX^T-z\right)^{-1} X}_F\notag\\
&= n \bignorm{\left(\frac{1}{n}XX^T-z\right)^{-1} \left(\frac{1}{n}XX^T\right)}_F= n \bignorm{I_{p} + z\left(\frac{1}{n}XX^T-z\right)^{-1} }_F\\
&\leq n\sqrt{p} \bignorm{I_{p} + z\left(\frac{1}{n}XX^T-z\right)^{-1}}_{op}\leq n\sqrt{p} \left(1+\frac{|z|}{\Im(z)}\right).
\end{align*}
\end{proof}

\begin{lemma}
\label{lem:SummandB-MP}
In \eqref{eq:ABCD-MP}, $\max_{k\in\oneto{p}}\abs{B(n,k,z)}\to 0$ almost surely as $n\to\infty$.
\end{lemma}
\begin{proof}
Using the terms $S(n,k,z)$ and $R(n,k,z)$ defined above, we	find by Lemma~\ref{lem:RnBoundMP} that
\[
\abs{R(n,k,z)} \leq n\sqrt{p}c(z),
\]
where $c(z)=1+\abs{z}/\Im(z)$.
Therefore, using Theorem~\ref{thm:largedev} $ii)$ with $\varepsilon=1/8$ and $D=3$, we obtain a constant $C_{\frac{1}{8},3}\geq 0$, such that for all $n\in\N$ and all $k\in\oneto{p}$:
\[
\Prob\left(\abs{B(n,k,z)} > \frac{n^{\frac{1}{8}}n\sqrt{p}}{n^2}c(z)\right) \leq \Prob\left(\abs{S(n,k,z)}>n^{\frac{1}{8}}R(n,k,z)\right) \leq \frac{C_{\frac{1}{8},3}}{n^3}.
\]
Using the union bound as in the proof of Lemma~\ref{lem:SummandA-MP} concludes the statement.
\end{proof}
For $C(n,k,z)$, we define the terms
\[
S'(n,k,z) \defeq
\sum_{i\in\oneto{n}} (\alpha_k(i)^2-1) \left[X_n^{(k)T}\left(\frac{1}{n}X_n^{(k)}X_n^{(k)T}-z\right)^{-1} X_n^{(k)}\right](i,i) 	
\]
and 
\[
R'(n,k,z) \defeq \sqrt{\sum_{i\in\oneto{n}} \bigabs{\left[X_n^{(k)T}\left(\frac{1}{n}X_n^{(k)}X_n^{(k)T}-z\right)^{-1} X_n^{(k)}\right](i,i)}^2}
\]	
to employ Theorem~\ref{thm:largedev} $i)$.

\begin{lemma}
\label{lem:SummandC-MP}
In \eqref{eq:ABCD-MP}, $\max_{k\in\oneto{p}}\abs{C(n,k,z)}\to 0$ almost surely as $n\to\infty$.
\end{lemma}
\begin{proof}
Using the terms $S'(n,k,z)$ and $R'(n,k,z)$ defined above, we	find by Lemma~\ref{lem:RnBoundMP} that
\[
\abs{R'(n,k,z)} \leq n\sqrt{p}c(z),
\]
where $c(z)=1+\abs{z}/\Im(z)$.
Therefore, using Theorem~\ref{thm:largedev} $i)$ with $\varepsilon=1/8$ and $D=3$, we obtain a constant $C_{\frac{1}{8},3}\geq 0$, such that for all $n\in\N$ and all $k\in\oneto{p}$:
\[
\Prob\left(\abs{C(n,k,z)} > \frac{n^{\frac{1}{8}}n\sqrt{p}}{n^2}c(z)\right) \leq \Prob\left(\abs{S'(n,k,z)}>n^{\frac{1}{8}}R'(n,k,z)\right) \leq \frac{C_{\frac{1}{8},3}}{n^3}.
\]
Using the union bound as in the proof of Lemma~\ref{lem:SummandA-MP} concludes the statement.
\end{proof}

\begin{lemma}
\label{lem:SummandD-MP}
In \eqref{eq:ABCD-MP}, $\max_{k\in\oneto{p}}\abs{D(n,k,z)}\to 0$ almost surely as $n\to\infty$.
\end{lemma}
\begin{proof}
With the observations a) and b) in the proof of Lemma~\ref{lem:RnBoundMP} and setting $X\defeq X_n^{(k)}$, we calculate
\begin{align*}
&-\frac{1}{n^2}\tr \left[X^T\left(\frac{1}{n}XX^T-z\right)^{-1}X\right]
= -\frac{1}{n^2} \tr\left[\left(\frac{1}{n}XX^T-z\right)^{-1}XX^T\right]\\
&= -\frac{1}{n} \tr\left[I_{p-1} + z \left(\frac{1}{n}XX^T-z\right)^{-1}\right]
=-\frac{p}{n} + \frac{1}{n} - \frac{z}{n}\tr\left(\frac{1}{n}XX^T-z\right)^{-1}
\end{align*}
Hence, using that $y_n=p/n$ and with Corollary~\ref{cor:tracedifference} (Note that our construction of $X_n^{(k)}$ differs from that in the corollary), we obtain
\begin{align*}
\abs{D(n,k,z)} &= \bigabs{ -\frac{p}{n} + \frac{1}{n} - \frac{z}{n}\tr\left(\frac{1}{n}X_n^{(k)}X_n^{(k)T}-z\right)^{-1} + y_n + y_n z\frac{1}{p}\tr\left(\frac{1}{n}X_nX_n^T-z\right)^{-1}}\\
&\leq \frac{1}{n} + \frac{\abs{z}}{n\Im(z)}.
\end{align*}
Since this bound holds uniformly for all $k\in\{1,\ldots,p\}$, it follows that 
\[
\max_{k\in\oneto{p}}\abs{B(n,k)} \leq \frac{1}{n} + \frac{\abs{z}}{n\Im(z)} \xrightarrow[n\to\infty]{} 0 \quad \text{surely.}
\]
	
\end{proof}

\begin{theorem}
\label{thm:sufficient}
In above situation, we find for any fixed $z\in\C_+$ that 
\[
\max_{k\in\oneto{p}}\abs{\Omega_n^{(k)}(z)}\xrightarrow[n\to\infty]{} 0 \quad \text{almost surely}.
\]
\end{theorem}
\begin{proof}
This follows directly by the decomposition \eqref{eq:ABCD-MP} with Lemma~\ref{lem:SummandA-MP}, Lemma~\ref{lem:SummandB-MP}, Lemma~\ref{lem:SummandC-MP} and Lemma~\ref{lem:SummandD-MP}.
\end{proof}

\addchap{Bibliography}
\defbibheading{bibliography}{}
\printbibliography

\end{document}